%% file: additive_polynomials_for_finite_groups_of_lie_type.tex
\documentclass[10pt]{article}

\usepackage[totalwidth=345pt,totalheight=550pt]{geometry} 
\usepackage{rotating}
\usepackage[english]{babel}
\usepackage[utf8]{inputenc}
\usepackage{amsthm}
\usepackage{amsmath}
\usepackage{amssymb}
\usepackage{amsfonts}
\usepackage{bbm}      
\usepackage{mathdots} 
\usepackage{stmaryrd} 
\usepackage{calc} 

\input{abbrv.tex}

\title{Additive Polynomials for Finite Groups of Lie Type}
\author{%
\begin{minipage}[t]{\textwidth}
\begin{center}
  \textsc{Maximilian Albert}\\[1ex]
  \emph{University of Southampton, School of Engineering Sciences\\ SO17 1BJ, Southampton, UK\\
    email: maximilian.albert@gmail.com}
\end{center}
\end{minipage}\\[5em]
  and\\[1em]
\begin{minipage}[t]{\textwidth}
\begin{center}
  \textsc{Annette Maier}\\[1ex]
  \emph{Rheinisch-Westfälische Technische Hochschule (RWTH)\\
    Templergraben 55, 52062 Aachen\\
    email: annette.maier@matha.rwth-aachen.de}
\end{center}
\end{minipage}\\[6.5em]
}

\date{January 29, 2010}

\begin{document}

\maketitle

{
\abstract{%
\noindent
This paper provides a realization of all classical and most
exceptional finite groups of Lie type as Galois groups over function
fields over $\F_q$ and derives explicit additive polynomials for the
extensions. Our unified approach is based on results of Matzat which
give bounds for Galois groups of Frobenius modules and uses the
structure and representation theory of the corresponding connected
linear algebraic groups.
}
}

\clearpage
\tableofcontents

\clearpage
\input{introduction.tex}
\input{notation.tex}
\input{tools.tex}

\input{polynomials.tex}

\input{references.tex}
\end{document}

%% file: abbrv.tex
\newcommand{\F}{\ensuremath{\mathbb{F}}}
\newcommand{\Fq}{\ensuremath{\F_q}}
\newcommand{\Fqbar}{\ensuremath{\overline{\F}_q}}
\newcommand{\G}{\ensuremath{\mathcal{G}}}
\newcommand{\T}{\ensuremath{\mathcal{T}}}
\newcommand{\B}{\ensuremath{\mathcal{B}}}
\newcommand{\W}{\ensuremath{\mathcal{W}}}

\newcommand{\Ocal}{\ensuremath{\mathcal{O}}}

\newcommand{\mfrak}{\ensuremath{\mathfrak{m}}}
\newcommand{\Pcal}{\ensuremath{\mathcal{P}}}
\newcommand{\wperm}{\ensuremath{\sigma_w}}
\newcommand{\Nor}[2]{\ensuremath{\mathcal{N}_{#1}(#2)}}
\newcommand{\Cen}[2]{\ensuremath{\mathcal{C}_{#1}(#2)}}
\newcommand{\Gal}{\ensuremath{\mathrm{Gal}}}
\newcommand{\GalPhi}{\ensuremath{\Gal^\Phi}}
\newcommand{\coloneqq}{\mathrel{\mathop:}=}

\newcommand{\ideal}{\trianglelefteqslant}

\newcommand{\charpoly}[1]{\chi_{#1}}

\newcommand{\toq}[1][]{^{\;q^{#1}}}
\newcommand{\phiq}{\ensuremath{\phi_q}}

\newcommand{\SolPhi}{\ensuremath{\operatorname{Sol}^\Phi}}
\newcommand{\Sym}{\ensuremath{\mathrm{S}}}

\newcommand{\GL}{\operatorname{GL}}
\newcommand{\SL}{\operatorname{SL}}
\newcommand{\Sp}{\operatorname{Sp}}
\newcommand{\SU}{\operatorname{SU}}
\newcommand{\SO}{\operatorname{SO}}
\newcommand{\Suz}{{}^{2\!}\operatorname{B}_2}
\newcommand{\Dickson}{\operatorname{G}_2}
\newcommand{\Ree}{{}^{2\!}\operatorname{G}_2}
\newcommand{\Triality}{{}^{3\!}\operatorname{D}_4}

\newcommand{\id}{\mathbbm{1}}
\newcommand{\underdash}{\ensuremath{\,\underline{~~}\,}}
\newcommand{\idblock}[2]{\left[\;#2\;\right]_{#1}}

\newcommand{\tr}{^\mathrm{tr}}
\newcommand{\rank}{\operatorname{rank}}
\newcommand{\diag}{\operatorname{diag}}
\newcommand{\ord}{\operatorname{ord}}
\newcommand{\bdiag}[2]{\ensuremath{\Delta_{#1}(#2)}}
\newcommand{\bdiaginv}[2]{\ensuremath{\nabla_{#1}(#2)}}
\newcommand{\bdiagtilde}[2]{\ensuremath{\widetilde{\Delta}_{#1}(#2)}}

\newcommand{\cyclesymbol}{\ensuremath{\rho}}
\newcommand{\cycle}[2]{\ensuremath{\cyclesymbol_{#1,#2}}}
\newcommand{\cyclelong}[4]{\ensuremath{\cyclesymbol_{#1,#2;#3,#4}}}

\newcommand{\blockdiag}[1]{\big[#1\big]} 
\newcommand{\singleton}[3][]{\ensuremath{E_{#2}^{#1}(#3)}}
\renewcommand{\root}[3][]{\singleton[#1]{#2}{#3}}
\newcommand{\weylneg}{\ensuremath{w^{-}}}
\newcommand{\weylpos}{\ensuremath{w^{+}}}
\newcommand{\weylpm}{\ensuremath{w^{\pm}}}
\newcommand{\weyl}{\weylneg}
\newcommand{\kxk}[1]{\ensuremath{#1\!\times\!#1}}
\newcommand{\pcs}{\ensuremath{S}} 
\newcommand{\proj}[1]{\overline{#1}} 

\newcommand{\genname}{\gamma}
\newcommand{\gendiag}{\widetilde{\genname}}
\newcommand{\gentilde}{\widetilde{\genname}}
\newcommand{\charpolyorth}{\ensuremath{\tilde{h}}} 
\newcommand{\minus}{\!-\!}
\newcommand{\plus}{\!+\!}
\newcommand{\q}{^{\:q}}

\newtheorem{thm}{Theorem}[section]
\newtheorem{lemma}[thm]{Lemma}
\newtheorem{prop}[thm]{Proposition}
\newtheorem{cor}[thm]{Corollary}
\newtheorem{rem}[thm]{Remark}

\renewcommand{\labelenumi}{\alph{enumi})}

\newcommand{\companionsymbol}{%
\raisebox{-0.2em}{%
\hspace{0.2em}%
\begin{picture}(0.9,1)
  \put(0,1){\line(1,-1){1}}
  \put(1,0){\line(0,1){1}}
  \linethickness{0.15em}
  \put(-0.019,1.07){\line(1,0){1.039}}
\end{picture}
}}
\newcommand{\companionrevsymbol}{%
\raisebox{-0.2em}{%
\hspace{0.2em}%
\begin{picture}(0.8,1)
  \put(0,1){\line(1,-1){1}}
  \put(0,0){\line(1,0){1}}
  \linethickness{0.15em}
  \put(0,0){\line(0,1){1}}
\end{picture}
}}
\newcommand{\companiontwosymbol}{%
\raisebox{-0.2em}{%
\hspace{0.2em}%
\begin{picture}(0.9,1)
  \put(0,1){\line(1,-1){1}}
  \put(0,1){\line(1,0){1}}
  \linethickness{0.15em}
  \put(1,0){\line(0,1){1}}
\end{picture}
}}
\newcommand{\companionrevtwosymbol}{%
\raisebox{-0.2em}{%
\hspace{0.2em}%
\begin{picture}(0.8,1)
  \put(0,1){\line(1,-1){1}}
  \put(0,0){\line(0,1){1}}
  \linethickness{0.15em}
  \put(0,0){\line(1,0){1}}
\end{picture}
}}

\newcommand{\companion}[2][]{\companionsymbol^{#1}\big(#2\big)}
\newcommand{\companionrev}[2][]{\companionrevsymbol^{#1}\big(#2\big)}
\newcommand{\companiontwo}[2][]{\companiontwosymbol^{#1}\big(#2\big)}
\newcommand{\companionrevtwo}[2][]{\companionrevtwosymbol^{#1}\big(#2\big)}
\newcommand{\osize}[1]{\mbox{\ensuremath{#1}}} 
\newcommand{\compmat}[4][]{\ensuremath{%
  \left(\begin{smallmatrix}
           \osize{#2} & \osize{\ldots} &   \osize{#3}  & \osize{#4} \\[0.5ex]
        {#1}1 &        &       &    \\[-1ex]
              & \osize{\ddots} &       &    \\
              &        & {#1}1 &  0 \\
        \end{smallmatrix}\right)}}
\newcommand{\compmatrev}[4][]{\ensuremath{%
  \left(\begin{smallmatrix}
           \osize{#4} & {#1}1 &        &       \\[-0.5ex]
       \osize{\vdots} &       & \ddots &       \\[0.5ex]
           \osize{#3} &       &        & {#1}1 \\[0.5ex]
           \osize{#2} &       &        &   0   \\
        \end{smallmatrix}\right)}}
\newcommand{\compmattwo}[4][]{\ensuremath{%
  \left(\begin{smallmatrix}
            0 &        &       & \osize{#4} \\
        {#1}1 &        &       & \osize{#3} \\[-1ex]
              & \osize{\ddots} &       & \osize{\vdots} \\
              &        & {#1}1 & \osize{#2} \\
        \end{smallmatrix}\right)}}
\newcommand{\compmatrevtwo}[4][]{\ensuremath{%
  \left(\begin{smallmatrix}
            0 & {#1}1 &        &       \\
              &       & \ddots &       \\
              &       &        & {#1}1 \\[0.5ex]
           \osize{#2} &    \osize{#3} & \osize{\ldots} &    \osize{#4} \\
        \end{smallmatrix}\right)}}

\newcommand{\twomat}[5][rnd]{
  \ifthenelse{\equal{#1}{rnd}}{\left(}{} 
  \ifthenelse{\equal{#1}{sq}}{\left[}{} 
  \begin{smallmatrix}
    #2 & #3 \\
    #4 & #5
  \end{smallmatrix}
  \ifthenelse{\equal{#1}{rnd}}{\right)}{}
  \ifthenelse{\equal{#1}{sq}}{\right]}{}}
\newcommand{\threemat}[9]{
  \left(
  \begin{smallmatrix}
    #1 & #2 & #3 \\
    #4 & #5 & #6 \\
    #7 & #8 & #9
  \end{smallmatrix}
  \right)}

%% file: introduction.tex
\section{Introduction}

The present work focuses on inverse Galois theory in positive characteristic.
This case was treated by Abhyankar in a large series of papers where he gave particularly nice polynomials for a number of classical groups over function fields in one or two variables over $\Fqbar$ (see \cite{Abhyankar01} and the references cited there).
Unfortunately, there is no obvious generalization of his techniques to other groups (although recently Conway et al.~\cite{McKay08} realized the Mathieu groups using methods inspired by his and Serre's work).
Using different methods, Elkies~\cite{Elkies} constructed polynomials with Galois groups $\SL_n(q)$ and $\operatorname{Sp}_{2k}(q)$ over function fields in several variables over $\mathbb{F}_{q}$.

In this work we present a novel approach based on Frobenius modules.
It uses criteria of Matzat~\cite{Matzat} to derive bounds for their Galois groups (which are matrix groups) and naturally leads to additive polynomials for these groups.
Note that in positive characteristic every finite Galois extension can be obtained by adjoining the roots of an additive polynomial (see \cite[§1.4.1]{Goss}).
The roots of such a polynomial form an $\Fq$-vector space and the Galois group, which acts linearly on the roots, can therefore be naturally embedded into $\GL_n$.
A result of Steinberg which we use for the construction of our Frobenius modules allows us to treat all groups in a unified way.
We derive explicit polynomials for
$\SL_n(q)$, $\Sp_{2k}(q)$, $\SO_n^{\pm}(q)$ (for $q$ odd), $\SU_n(q)$, $\Suz(q)$, $\Dickson(q)$, $\Ree(q)$ and $\Triality(q)$ (for $q$ odd) over a function field over~$\Fq$, where the number of indeterminates depends on the group.
There also is some ongoing work on the groups $\operatorname{F}_4(q)$ and $^{2\!}\operatorname{F}_4(q)$.

Matzat's results were already used by Malle~\cite{Malle} to realize~$\Dickson(q)$ but our method leads to a different Frobenius module and a polynomial with smaller coefficients.
In the case of $\SL_n(q)$ and $\Sp_{2k}(q)$ we obtain the same polynomials as Elkies.

It is an open question whether the polynomials derived in this way are generic.
This is likely to be the case for $\SL_n(q)$ where we essentially recover Dickson's polynomials (with a more specialized scalar term).
We conjecture that it is also true for $\Sp_{2n}(q)$, $\Suz(q)$, $\SU_n(q)$ (for odd~$n$ at least) and possibly $\Dickson(q)$, $\Ree(q)$, since in these cases the resulting polynomials are particularly ``nice'' and suggest that they encode the right kind of symmetries.
The polynomials obtained for~$\SO_n^{\pm}$ are considerably more complicated and their shape makes it less clear whether they are as ``general'' as required for being generic (even though their coefficients may still contain enough symmetries and degrees of freedom for this to be the case).

\medskip
In Section~\ref{sec:notation} we fix our notation, in particular we introduce some specialized abbreviations for certain often occurring matrices.

Section~\ref{sec:the method} starts by collecting all the tools needed for our approach.
It first gives a brief overview of Frobenius modules and their Galois groups and presents Steinberg's result on which the construction of our Frobenius modules is based.
Then it discusses so-called strong generators for groups of Lie type which are needed to make the lower bound criterion work.
After presenting some auxiliary material, it concludes with a detailed overview of the method, describing the basic plan which all the constructions in later sections follow.

In Section~\ref{sec:polynomials for the classical groups} we construct Frobenius modules for the classical groups $\SL_n(q)$, $\Sp_{2k}(q)$, $\SO_n^+(q)$ and derive additive polynomials for each of them.
The latter is achieved by solving the system of equations which defines the so-called solution space of the Frobenius module.
Section~\ref{sec:polynomials for the twisted classical groups} performs the construction for the twisted classical groups $\SU_n(q)$, $\SO_{2k}^-(q)$.
Although the basic approach is the same, some subtleties occur which we have to deal with.
Section~\ref{sec:polynomials for exceptional groups of lie type} finally treats the exceptional groups $\Suz(q)$, $\Dickson(q)$, $\Ree(q)$, $\Triality(q)$.
Apart from $\Triality(q)$ we give explicit polynomials for all groups which actually turn out to be quite simple.
Using a computer algebra system, it is also possible to compute a polynomial which has Galois group $\Triality(q)$ (or possibly~$\Triality(q) \times \operatorname{Z}_2$). However, its coefficients are much too large to be printed.

\medskip
\noindent\textbf{Acknowledgments.} The authors wish to thank B.\,H.~Matzat and J.~Hartmann for many helpful comments and suggestions.
We would also like to thank P.~M\"uller for solving an equation crucial for the realization of the Triality groups $\Triality(q)$ (see Prop.~\ref{lsgi12}). 

%% file: notation.tex
\section{Notation}
\label{sec:notation}

\newlength{\widest}
\newlength{\width}
\settowidth{\widest}{$\idblock{j,\textbf{2},\underdash,\textbf{2},j}{\twomat{1}{a}{0}{1}, \twomat{1}{b}{0}{1}}^{n \times n}$}
\setlength{\width}{\textwidth - \widest - 1cm}
\begin{table}[htb]
\begin{tabular}{p{\widest}p{\width}}
$\Fq(\underline t)$ & a function field in several variables over $\Fq$, where $\underline t=(t_1,\dots,t_r)$
for some $r$. \\
$\G$ & a linear algebraic group over $\overline{\Fq(\underline t)}$, if not otherwise stated. \\
$\phiq$ & the standard Frobenius map, $\phiq: a \mapsto a^q$.\\
$F$     & a general Frobenius map (i.e., $F^s = \phiq$ for some~$s$).\\
$\G^F$ & the set of $F$-invariant points of $\G$, where $F$ denotes some Frobenius map; we consider these groups merely as abstract groups and not as linear algebraic groups. \\
$\T_0$ & the maximal torus of $\G$ consisting of all diagonal matrices contained in $\G$, where we fix an embedding \mbox{$\G \hookrightarrow \GL_n$}.\\
$\B_0$ & the Borel subgroup of $\G$ consisting of all upper triangular matrices contained in $\G$, where we fix an embedding \mbox{$\G \hookrightarrow \GL_n$}.\\
$\Nor{\G}{\T_0}$ & normalizer of $\T_0$ in $\G$\\
$\Cen{\G}{\T_0}$ & centralizer of $\T_0$ in $\G$\\
$\W$ & Weyl group\\
$\charpoly{A}$ & characteristic polynomial of~$A$\\
$E_{i,j}(t)$   & a matrix having 1's on the diagonal and one further nonzero entry $t$ at position $(i,j)$\\
$\diag(\lambda_1,\dots,\lambda_n)$ & diagonal matrix with entries $\lambda_1,\dots,\lambda_n$.\\[1ex]
$\idblock{j,\textbf{2},\underdash,\textbf{2},j}{\twomat{0}{1}{1}{0}, \twomat{0}{1}{1}{0}}^{n \times n}$ & a typical block matrix; see below\\[1ex]
$\companionsymbol^\pm$, $\companionrevsymbol^\pm$, etc. & companion matrices; see below
\end{tabular} 
\caption{Common notation}
\label{table:symbols}
\end{table}

\bigskip \noindent
Table~\ref{table:symbols} explains some commonly used symbols.
By $\singleton[\,n]{i,j}{t}$ we denote an $\kxk{n}$-matrix having 1's on the diagonal and one further nonzero entry~$t$ in position~$(i,j)$.
If no confusion can arise the size is omitted.
In the usual case that the nonzero entry is directly above the main diagonal we abbreviate~$\root{i}{t} \coloneqq \singleton{i,i+1}{t}$.
We further define $\weylpm_{i,n} \coloneqq \singleton[n]{i,i+1}{1} \pm \singleton[n]{i+1,i}{1}$. This is an identity matrix where the trivial \kxk{2}-block on the diagonal at position~$(i,i$) is replaced with $\left(\begin{smallmatrix} 0 & 1 \\ \pm 1 & 0 \end{smallmatrix}\right)$.
Naturally, in both abbreviations the index $i$ is always positive.
To simplify notation we agree, however, that negative indices are allowed and should be interpreted as ``counting backwards'' from the matrix size.
That is, we define  $\root{-i}{t} \coloneqq \root{(n+1)-i}{t}$ and $\weylpm_{-i} \coloneqq \weylpm_{(n+1)-i}$ if these matrices have dimensions~\kxk{n}.

If $A_1, \ldots, A_n$ are any square matrices (not necessarily of equal size), we write $\blockdiag{A_1,\ldots,A_n}$ to denote the corresponding block diagonal matrix.

We often need a sort of companion matrix, similar to those occurring in linear algebra. To this end, we define
\begin{alignat*}{2}
  \companion[\pm]{a_1,\ldots,a_n} &= \compmat[\pm]{a_1}{a_{n-1}}{a_n}, \quad
  &\companionrev[\pm]{a_1,\ldots,a_n} &= \compmatrev[\pm]{a_1}{a_2}{a_n},\\[1ex]
  \companiontwo[\pm]{a_1,\ldots,a_n} &= \compmattwo[\pm]{a_1}{a_{n-1}}{a_n}, \quad
  &\companionrevtwo[\pm]{a_1,\ldots,a_n} &= \compmatrevtwo[\pm]{a_1}{a_2}{a_n}.
\end{alignat*}
We frequently omit the plus sign in the notation when the 1's on the \mbox{sub-/}super\-diagonal have positive sign.
Note that the shapes of the symbols reflect the arrangements of the nonzero matrix entries.
Also note that in the second and third form, the entries $a_1,\ldots,a_n$ are arranged \emph{from bottom to top}.

We introduce another symbol for a block diagonal matrix where only a few blocks differ from the identity matrix.
The following example denotes such a matrix with two \kxk{2}-blocks.
$$
\idblock{j,\textbf{2},\underdash,\textbf{2},j}{
    \twomat{1}{a}{0}{1},
    \twomat{1}{b}{0}{1}
  }^{n \times n} =
  \diag\Big(
    \underbrace{1 , \,\ldots\, , 1}_{j},
    \framebox{$\begin{smallmatrix} 1 & a \\[1ex] 0 & 1 \end{smallmatrix}$}\,,
    \underbrace{1 , \,\ldots\, , 1}_{n-2j-4},
    \framebox{$\begin{smallmatrix} 1 & b \\[1ex] 0 & 1 \end{smallmatrix}$}\,,
    \underbrace{1 , \,\ldots\, , 1}_{j},
    \Big).
$$
The superscript \kxk{n} gives the total dimensions of the matrix and is frequently omitted.
The indices occurring in the subscript specify the dimensions of the individual diagonal blocks, from upper left to lower right.
By default, each block equals an identity matrix and only the non-trivial blocks are printed between the square brackets (in the order in which they appear).
As a further indication, their dimensions are printed in bold.
Usually the size of one block is rather complicated compared to the others because it equals $n$ minus the sum of all other block sizes (e.g., the middle block above has size $n-2j-4$).
In such a case, we normally replace it with a dash as in the above example because it can easily be inferred from the remaining data.

This notation makes it convenient to multiply two matrices of this form.
We will encounter two different cases.
First, if both matrices have the same block pattern then the multiplication can be performed block-wise.
On the other hand, we also often have a block matrix with companion matrices as upper left and lower right blocks,
$  \idblock{\textbf{j},\underdash,\;\textbf{j}}{
    \companion{a_1,\ldots,a_{j-1},1},
    \companionrev{1,b_1,\ldots,b_{j-1}}
  }$,
which we want to multiply by some matrix of the previous form, namely
\begin{equation*}
  \idblock{(j-1),\textbf{2},\underdash,\textbf{2},(j-1)}{
    \twomat{a_{j}}{1}{1}{0},
    \twomat{b_{j}}{1}{1}{0}
  }.
\end{equation*}
Note that the \kxk{2}-blocks of the second matrix overlap with the block boundaries of the first.
It is easy to see that the product has the same shape as the first matrix again, but the upper left and lower right blocks are each extended by one row/column:
\begin{multline*}
\idblock{\textbf{j},\underdash,\;\textbf{j}}{
    \companion{a_1,\ldots,a_{j-1},1},
    \companionrev{1,b_1,\ldots,b_{j-1}}
  }\\
 \;\cdot\;
 \idblock{(j-1),\textbf{2},\underdash,\textbf{2},(j-1)}{
    \twomat{a_{j}}{1}{1}{0},
    \twomat{b_{j}}{1}{1}{0}
  }
\\
  =
\idblock{\textbf{(j+1)},\underdash,\textbf{(j+1)}}{
    \companion{a_1,\ldots,a_j,1},
    \companionrev{1,b_1,\ldots,b_j}
  }.
\end{multline*}
This is the primary case which repeatedly occurs during the constructions in
sections~\ref{sec:polynomials for the classical groups}-\ref{sec:polynomials for exceptional groups of lie type}.

%% file: tools.tex
\section{The Method}
\label{sec:the method}

\input{frobenius_modules.tex}
\input{steinberg.tex}
\input{generators.tex}
\input{auxiliary_material.tex}
\input{overview_of_the_method.tex}

%% file: frobenius_modules.tex
\subsection{Frobenius Modules}
\label{sec:frobenius modules}

We briefly recall the essential facts and properties about Frobenius modules.
For an extensive treatment see~\cite{Matzat} or~\cite{Albert}.

\newcommand{\grdfield}{K}%
\newcommand{\extfield}{L}%
\newcommand{\zwmat}{D_0}%
\newcommand{\zwmattilde}{\widetilde{D}_0}%
Let $\grdfield \geq \Fq$ be a field, and let $\phiq$ denote the standard Frobenius endomorphism on~$\grdfield$ which raises every element to its $q$-th power (we will sometimes refer to $K$ as a \emph{Frobenius field}).
A vector space~$M$ of finite dimension~$n$ over~$\grdfield$ together with an injective $\phiq$-semilinear map $\Phi\colon M \to M$ is called a \emph{Frobenius module} (over~$\grdfield$).
The \emph{solution space} of $(M,\Phi)$ is defined as the set of points fixed by~$\Phi$:
\begin{equation}
  \label{eqn:def_solution_space}
  \SolPhi(M) = \big\{ x \in M \;\big|\; \Phi(x) = x \big\}.
\end{equation}
After fixing a $\grdfield$-basis $B=\{b_1,\ldots,b_n\}$ of $M$, we can define the \emph{representing matrix} $D=D_B(\Phi)$ of $\Phi$ with respect to this basis as the matrix which collects in its columns the coefficients of the images~$\Phi(b_j)$ of the basis vectors ($j=1,\ldots,n$).
Hence, $\Phi(x) = D \cdot \phiq(x) = D \cdot \big(x_1^q,\ldots,x_n^q\big)\tr$ for $x=(x_1,\ldots,x_n) \in M$.

For any field extension $\extfield\geq \grdfield$ with extended Frobenius endomorphism, the tensor product $M_\extfield \coloneqq \extfield \otimes_\grdfield M$ becomes a Frobenius module over~$\extfield$ in a natural way.
We call $\SolPhi_\extfield(M) \coloneqq \SolPhi(M_\extfield)$ the \emph{solution space of $M$ over $\extfield$}.
For any extension $\extfield$ this solution space is an $\Fq$-vector space whose dimension satisfies
\begin{equation}
  \label{eqn:dim of solution space}
  \dim_{\Fq}\big(\SolPhi_\extfield(M)\big) \leq \dim_\extfield(M_\extfield) = \dim_\grdfield(M).
\end{equation}
If equality holds (i.e., if the dimension of the solution space is maximal) then $M$ is called \emph{trivial over $\extfield$}.
One can show\footnote{See \cite{Matzat}, Thm.~1.1.} that there always exists a finite extension $\extfield/\grdfield$ of Frobenius fields such that $M$ becomes trivial over $\extfield$.
Furthermore, the minimal such extension is unique (inside a given algebraic closure of~$\grdfield$) and Galois over $\grdfield$.
It is called the \emph{solution field} of $(M,\Phi)$.
The \emph{Galois group of~$M$} is defined as the Galois group of this extension and is denoted $\GalPhi(M)$.

Note that by fixing a basis and using~(\ref{eqn:def_solution_space}), we see that each solution vector $x=(x_1,\ldots,x_n)\tr$ must satisfy
\begin{equation}
  \label{eqn:system of solutions in coordinates}
  D \cdot (x_1^q, \ldots, x_n^q)\tr = (x_1, \ldots, x_n)\tr.
\end{equation}
This yields a system of polynomial equations for the coordinates of a solution vector, and the solution field is generated by the coordinates of a full system of solutions, i.e., a set of solutions which forms a basis of~$\SolPhi(M)$.
We sometimes collect such a full system of solutions in the columns of a \emph{(fundamental) solution matrix~$Y$}, which is characterized by $D\cdot \phiq(Y) = Y$.
In our applications, we will determine actual polynomials with Galois group $\GalPhi(M)$ by eliminating all but one of the~$x_i$, thereby reducing the system~(\ref{eqn:system of solutions in coordinates}) to a single equation involving a univariate additive polynomial (of much higher degree).

\subsection{Galois Groups of Frobenius Modules}

The following theorems make clear why Frobenius modules are particularly suited to study linear groups as Galois groups.
Throughout this section let $(M,\Phi)$ be an $n$-dimensional Frobenius module over $(\grdfield, \phiq)$ with solution field~$\extfield$ and fundamental solution matrix~$Y$.%
\begin{prop}
\label{prop:galois linrep}
There is a faithful linear representation
$\Gamma: \GalPhi(M)=\Gal(\extfield/\grdfield) \hookrightarrow \GL_n(\Fq)$, $\sigma \mapsto Y^{-1} \sigma(Y)$.
\end{prop}
\begin{proof}
  \cite{Matzat}, Corollary~4.2.
\end{proof}

\noindent
There exist powerful criteria to derive bounds for the Galois group of a Frobenius module.
The next one for an upper bound is particularly simple and easy to verify.
\begin{thm}[Matzat]\label{thm:upperbound}
Let $\mathcal G$ be a reduced connected linear algebraic group defined over\/ $\Fq$ and suppose that $D_B(\Phi) \in \mathcal G(\grdfield)$ for some basis $B$ of $M$.
Then $\Gal^\Phi(M) \leq \mathcal G(\Fq)$.
\end{thm}
\begin{proof}
  \cite{Matzat}, Theorem~4.3.
\end{proof}
\noindent
This theorem is appropriate for the untwisted groups.
Since in the twisted cases $\G^F$ is smaller than $\G(\Fq)$ we also need the following variant.
\begin{thm}\label{thm:twupperbound}
Let $\mathcal G$ be a linear algebraic group defined over $\Fq$ with Frobenius map $F: \mathcal G\rightarrow \mathcal G$, such that $F^s=\phi_q$.
Let further $B$ be a basis of M such that
\begin{center}
$D_B(\Phi)=\zwmat\cdot F(\zwmat)\cdots F^{s-1}(\zwmat)$ \end{center}
for some $\zwmat \in \G(K)$. Then $\Gal^\Phi(M) \leq \mathcal G^F$.
\end{thm}

\begin{proof}
By the Lang-Steinberg theorem\footnote{See \cite[§1.17]{Carter85}.} there exists an element $Y \in \G(\overline{K})$ such that $F(Y)=\zwmat^{-1}\cdot Y$. Hence
\begin{eqnarray*}
D\cdot \phi_q(Y)&=&D\cdot F^s(Y) \;\;=\;\; D \cdot F^{s-1}(\zwmat^{-1})\cdot F^{s-1}(Y)\\
&=& \cdots \;\;=\;\; D\cdot F^{s-1}(\zwmat^{-1})\cdots F(\zwmat^{-1}) \cdot \zwmat^{-1} \cdot Y
\\
&=&D\cdot \big(\zwmat F(\zwmat) \cdots F^{s-1}(\zwmat)\big)^{-1} \cdot Y \;\;=\;\; Y.
\end{eqnarray*}
Thus the field extension $N$ generated by the coefficients of $Y$ over $K$ is a solution field for $(M,\Phi)$.
We embed $\Gal^\Phi (M) \rightarrow \G(\Fq)$, $\sigma \mapsto Y^{-1} \sigma (Y)$ as in Prop.~\ref{prop:galois linrep}.
For $\sigma \in \Gal^\Phi (M)=\Gal(N/K)$, we have
\begin{center}
$F(Y^{-1}\cdot \sigma(Y))=F(Y)^{-1}\cdot F(\sigma (Y)) =F(Y)^{-1} \cdot \sigma (F(Y))$, 
\end{center}
since $F$ is given polynomially (over $\Fq$) in every coordinate and thus commutes with $\sigma$.
Plugging in $F(Y)=\zwmat^{-1}\cdot Y$ and using that the coordinates of $\zwmat$ lie in the fixed field of $\Gal^\Phi (M)$, we see that the last term equals $Y^{-1} \cdot \sigma(Y)$.
Hence $F(Y^{-1}\cdot \sigma(Y))=Y^{-1}\cdot \sigma(Y)$ for all $\sigma \in \Gal^\Phi (M)$.
\end{proof}

\noindent
Finding lower bounds for the Galois group of a Frobenius module is more difficult.
We present a special case of Thm.~4.5 in~\cite{Matzat} which can be regarded as a ``linear variant'' of Dedekind's classical criterion.

We use the following notation.
Let $R$ be a subring of~$\grdfield$ and $\mfrak \ideal R$ a maximal ideal.
By Chevalley's theorem\footnote{See \cite[Thm.\,3.1.1]{EnglerPrestel}.} there exists a valuation ring~$\Ocal$ of~$\grdfield$ such that $R \subseteq \Ocal$ and $\Pcal \cap R = \mfrak$, where $\Pcal$ is the unique maximal ideal of~$\Ocal$.
We denote by $E=\Ocal/\Pcal$ the residue field and assume that this is a finite extension of~$\Fq$ of degree $f \coloneqq [E:\Fq]$. We let $\psi\colon \Ocal \to E$ be the canonical projection.
\begin{thm}[Matzat]\label{thm:lowerbound}
Assume that $(M,\Phi)$ is an $n$-dimensional Frobenius module over $(\grdfield,\phi_q)$ such that $D=D_B(\Phi)$ lies in $\GL_n(\Ocal)$ for some basis $B$ of M.
Denote by $D_{\Pcal} = \psi(D) \in \GL_n(E)$ the residue matrix of~$D$.
Then
\renewcommand{\labelenumi}{(\roman{enumi})}%
\begin{enumerate}
  \item $\Gal^\Phi(M) \leq \GL_n(\Fq)$ contains an element which is conjugate to
    $$\widehat{D} \coloneqq D_{\Pcal} \cdot \phiq(D_{\Pcal}) \cdot \ldots \cdot \phiq^{f-1}(D_{\Pcal})$$
    inside $\GL_n(\Fqbar)$.
  \item If furthermore~$\G$ is a connected linear algebraic group defined over $\Fq$ and $D$ lies in $\G(\Ocal)$ then $\Gal^\Phi(M) \leq \GL_n(\Fq)$ contains an element which is $\G(\Fqbar)$-conjugate to $\widehat{D} \in \G(\Fqbar)$.
\end{enumerate}
\end{thm}
\renewcommand{\labelenumi}{\alph{enumi})}

\noindent
We only need this theorem in the special case where $\grdfield=\Fq(t_1,\ldots,t_k)$ and $R=\Fq\left[t_1,\ldots,t_k\right]_\mfrak$ is the localization of the polynomial ring in $k$ variables at the maximal ideal $(t_1-a_1,\ldots,t_k-a_k)$, where $a_i \in \Fq$.
Consider the rank-$k$ valuation on $F$ (with $\mathbb{Z} \times \dots \times \mathbb{Z}$ ordered lexicographically) which maps an element in $F$ to its $(t_1-\alpha_1)$-adic valuation $n_1$ in the first component, to the $(t_2-\alpha_2)$-adic valuation of $\overline{f\cdot(t_1-\alpha_1)^{-n_1}}$ in the second component, and so on (where $\overline{f\cdot(t_1-\alpha_1)^{-n_1}}$ denotes the residue of $f\cdot(t_1-\alpha_1)^{-n_1}$ modulo $(t_1-\alpha_1)$, which is contained in $\Fq(t_2, \dots, t_k)$). 
Let $\Ocal$ and $\Pcal$ be the corresponding valuation ring and valuation ideal. Then we have $\Ocal/ \Pcal=\Fq$, by construction. 
Thus $f=1$ (hence~$E=\Fq$) and $\psi$ restricted to $R$ is the homomorphism induced by the specializations \mbox{$t_i\mapsto a_i$}.
Thus in this case the theorem asserts that $\GalPhi(M)$ contains a matrix which is conjugate to the specialization $\widehat{D} = \psi(D)$ if $D$ is contained in $\GL_n(R)$.

\bigskip
\noindent
We also need the following slightly technical lemma to compute polynomials with the desired Galois groups.
\begin{lemma} \label{lem:galoisgroup}
Let $(M, \Phi)$ be an $n$-dimensional Frobenius module over $\grdfield$ with representing matrix $D$. Let further $f$ be a polynomial of degree~$q^n$. Suppose that any solution $x=(x_1,\ldots,x_n)$ to $D \cdot \phi_q(x)=x$ is uniquely determined by its first coordinate~$x_1$, all other coordinates of~$x$ lie in the field extension generated by~$x_1$, and $f(x_1)=0$. Then $\GalPhi(M) \cong \Gal_{\grdfield}(f)$.
\end{lemma}
\begin{proof}
The solution field $\extfield$ of $M$ is generated by the coordinates of all solutions to $D \cdot \phi_q(x)=x$, so it is contained in the splitting field of $f$. On the other hand, $M$ is trivial over $\extfield$, so there are $q^n$ distinct solutions, hence also $q^n$ distinct first coordinates of solutions, all of which are roots of $f$. But $\deg(f)=q^n$, so every root of $f$ occurs as first coordinate of a solution. This implies that the splitting field of $f$ is contained in $\extfield$, thus the Galois groups coincide.
\end{proof}

%% file: steinberg.tex
\subsection{Steinberg's Cross-Section}
\label{sec:steinberg}

We wish to construct a Frobenius module $(M,\Phi)$ such that $\Gal^\Phi(M)\cong \G^F$ is a given finite group of Lie type. Having at hand Theorem~\ref{thm:lowerbound} to obtain lower bounds, we need to construct a representing matrix $D=D_B(\Phi)$ such that any set of conjugates to specializations $\psi_i(D)$ generates $\G^F$.
So we need $D$ to be of a certain generality, as well as of simple form in order to be able to obtain polynomials with this Galois group.
Our construction is based on the following theorem \cite[Thm.\;1.4]{Steinberg}.
\begin{thm}[Steinberg]\label{thm:cross-section}
 Let $\G$ be a connected semisimple linear algebraic group of rank $r$. Let $T$ be a maximal torus in $\G$ and $\left\{\alpha_i \ | \ 1\leq i \leq r\right\}$ a system of simple roots relative to $T$. For each $i$ let $X_i$ and $w_i$ be the root subgroup and Weyl group representative related to $\alpha_i$, respectively. Let
\begin{equation}
  \pcs = \prod\limits_{i\leq r}X_iw_i =X_1w_1X_2w_2 \dots X_rw_r.
\label{eqn:cross-section}
\end{equation}
If $\G$ is a simply-connected group, then $\pcs$ is a \emph{cross section} (i.e., a system of representatives) of the collection of regular conjugacy classes of $\G$.
\end{thm}
\noindent
Since we are only interested in specializations of~$D$ up to conjugation, this theorem promises to be a powerful tool for our purposes.
Note that not all groups considered in this work are simply connected (e.g., $\SO_{2k}^+$ is not).
Moreover, for some of the twisted groups the construction needs to be adapted due to issues with rationality.
But~$\pcs$ as defined above is still a good starting point in all cases; details will be discussed in the corresponding sections.
Note that in our applications we always consider the root subgroups and Weyl group representatives with respect to the standard torus~$\T_0 \leq \G$ of diagonal matrices and the standard Borel subgroup~$\B_0 \leq \G$ of upper triangular matrices.

%% file: generators.tex
\subsection{Strong Generating Systems}
\label{sec:generators}

In order to apply Theorem~\ref{thm:lowerbound} to the groups of our interest, we need generators with the property that any set of $\GL_n(\Fq)$-conjugates of these are generators, too.
Such elements will be called \emph{strong generators}.
In the joint work~\cite{MalleSaxlWeigel} of Malle, Saxl and Weigel it is proved as an intermediate result that for every classical group $G$, any two regular semisimple elements lying in certain maximal tori of $G$ form a set of strong generators (for the regularity condition see the end of this section).
Based on these results, we present pairs of regular strong generators for all classical groups in Table~\ref{table:generators} (the exceptional groups will be dealt with in the corresponding sections).
Before explaining the notation used in the table, we recall some basic facts about maximal tori in finite groups of Lie type.
For details and proofs see~\cite[§3.3]{Carter85}.

Let $\G$ be a linear algebraic group of classical type (of which we think in its natural representation) and $F$ a Frobenius morphism on $\G$.
A maximal torus $T$ of the corresponding \emph{finite} classical group $G=\G^F$ is a subgroup which consists of the $F$-fixed points of an $F$-stable maximal torus of $\G$.
Recall that these tori can be classified via elements of the Weyl group $\W = \Nor{\G}{\T_0} / \Cen{\G}{\T_0}$ as follows.
Let $w \in \Nor{\G}{\T_0}$ be any Weyl group representative (we can choose $w$ as a permutation matrix and write $\wperm \in \Sym_n$ for the associated permutation).
Moreover, choose $h \in \G$ such that $h \cdot F(h)^{-1} = w$, which is possible by the Lang-Steinberg theorem.
Then~$\T = \T_0^{h}$ is another $F$-stable maximal torus, and all such tori can be obtained in this way.
This leads to a classification of all $F$-stable maximal tori of~$\G$.
Now let $T \coloneqq \T^F = (\T_0^{h})^F$ be the corresponding maximal torus of the finite group~$G$.
Its elements can be characterized as follows.
\begin{equation}
  \label{eqn:tor_mat_relation}
  T = h^{-1} \; \big\{ t_0 \in \T_0 \;\big|\; F(t_0) = w^{-1} \cdot t_0 \cdot w \big\} \; h.
\end{equation}
That is, $T$ is conjugate in $\GL_n(\Fqbar)$ to a group of diagonal matrices which satisfy a certain relation involving the permutation matrix~$w$.
We are going to make use of this fact below.

We now explain the conventions used in Table~\ref{table:generators}.
The various groups are distinguished by their Lie type (one out of $A$, $B$, $C$, $D$).
The rank of~$\G$ is always denoted~$k$, whereas $n$ denotes the dimension of the natural representation (the relation between both is indicated for each type).
Each row lists the information for one strong generator~$\genname \in G$.
The columns contain, from left to right:
\begin{itemize}
  \item[-] the order of the maximal torus $T \leq G$ containing $\genname$,
  \item[-] the permutation~$\wperm$ associated to the classifying Weyl group element~$w$ for~$T$, and
  \item[-] the diagonalization $\gendiag$ of $\genname$ (which exists since $\genname$ is semisimple); note that we do not print $\genname$ itself because the diagonalization can be written much more concisely while still containing all necessary information.
\end{itemize}
Concerning the last two items, some further comments are in order.
Firstly, all permutations are considered as elements of~$\Sym_n$.
We use the following notation for cycles in permutations.
For $1 \leq r_1 < r_2 \leq n$, we define
\begin{align*}
  \cycle{r_1}{r_2} &\coloneqq (r_1, r_1 \plus  1, \dots, r_2), \qquad
  \cycle{r_2}{r_1} \coloneqq (r_2, r_2 \minus 1, \dots, r_1).\\
\intertext{Similarly, for $1 \leq r_1 < r_2 < r_3 < r_4 \leq n$:}
  \cyclelong{r_1}{r_2}{r_4}{r_3} &\coloneqq (r_1, r_1 \plus  1, \dots, r_2,\;\; r_4, r_4 \minus 1, \dots, r_3),\\
  \cyclelong{r_2}{r_1}{r_3}{r_4} &\coloneqq (r_2, r_2 \minus  1, \dots, r_1,\;\; r_3, r_3 \plus 1, \dots, r_4).
\end{align*}
We also abbreviate two special permutations which occur more than once:
\begin{equation*}
  \tau_k \coloneqq (k,k \plus 1) \;\in \Sym_n, \qquad\qquad \tau \coloneqq (1,n) (2, n \minus 1) (3, n \minus 2) \cdots \;\in \Sym_n.
\end{equation*}
Secondly, all the diagonal entries in $\gendiag$ are $q^i$-th powers of a single ``fundamental diagonal entry'' $\alpha$, and in some cases of a second one $\beta$.%
\footnote{In the case where $F=\phiq$ is the standard Frobenius endomorphism this is easily derived from the relation $w^{-1} \cdot \gendiag \cdot w = F(\gendiag)$ implied by~(\ref{eqn:tor_mat_relation}), which shows that the $q^i$-th powers of~$\alpha$ and $\beta$ are permuted precisely in the way prescribed by~$\wperm$.}
These elements can be characterized as follows.
If there is only one fundamental entry~$\alpha$ then $\ord(\alpha)=|T|$.
Otherwise, the orders of $\alpha$ and $\beta$ equal the two factors into which $|T|$ naturally split, with $\ord(\alpha) \geq \ord(\beta)$ (for example, in the case of the second torus of $\SO_{2k}^-(q)$ we have $\ord(\alpha)=q^{k-1}+1$ and $\ord(\beta)=q-1$).
Note that any elements~$\alpha$, $\beta$ satisfying these constraints yield strong generators.

Since the block diagonal matrices occurring for~$\gendiag$ are quite large, we use the following abbreviations.
\begin{align*}
\bdiag{k}{\alpha}   &= \diag(\alpha, \alpha^q, \;\dots\;, \alpha^{q^{k-1}}),\\
\bdiaginv{k}{\alpha} &= \diag(\alpha^{-q^{k-1}}, \;\dots\;, \alpha^{-q}, \alpha^{-1} ),\\
\bdiagtilde{k}{\alpha}   &= \diag\big(\alpha, \alpha^{-q}, \alpha^{q^2}, \alpha^{-q^3} \;\dots\;, \alpha^{(-q)^{k-1}}\big).
\end{align*}
Thus $\bdiag{k}{\alpha}$ contains successive $q^i$-th powers of~$\alpha$, whereas $\bdiaginv{k}{\alpha}$ contains successive $q^i$-th powers of~$\alpha^{-1}$, in reverse order.\footnote{Think of the symbols as ``delta'' (for diagonal) and ``delta reverse''.}
The third kind of matrix $\bdiagtilde{k}{\alpha}$ is of similar shape as the first one but with alternating signs in the exponents (adapted for use with the special unitary groups).
Block diagonal matrices are notated using square brackets around the individual blocks, where single elements are to be interpreted as blocks of size~1.%

\newcommand{\grouptype}[2]{\begin{minipage}[t]{2cm}#1\\ {\footnotesize #2}\end{minipage}}
\begin{sidewaystable}
\centering
\begin{tabular}{lccc}
Group & $|T|$ & $w$ & Diagonalized generator \\[1ex]
\hline
\grouptype{$\SL_n(q)$}{$(n=k \plus 1)$}\vphantom{\Big(}
 & $\frac{q^n-1}{q-1}$     & \cycle{n}{1}   & \bdiag{n}{\alpha} \\[1ex]
 & $q^k \minus 1$ & \cycle{k}{1} & $\bdiag{n}{\alpha}$ \\[1em]
\grouptype{$\Sp_n(q)$}{$(n=2k)$}
 & $q^k \minus 1$ & $\cycle{k}{1} \cdot \cycle{(k+1)}{(2k)}$ & $\blockdiag{\bdiag{k}{\alpha}, \bdiaginv{k}{\alpha}}$ \\[1em]
 & $q^k \plus  1$ & \cyclelong{k}{1}{k+1}{2k} & $\blockdiag{\bdiag{k}{\alpha}, \bdiaginv{k}{\alpha}}$ \\[1em]
\grouptype{$\SO_n(q)$}{$(n=2k+1)$}
 & $q^k \minus 1$ & $\cycle{k}{1} \cdot \cycle{(k+2)}{(2k+1)}$ & $\blockdiag{\bdiag{k}{\alpha}, 1, \bdiaginv{k}{\alpha}}$ \\[1em]
 & $q^k \plus  1$ & \cyclelong{k}{1}{k+2}{2k+1} & $\blockdiag{\bdiag{k}{\alpha}, 1, \bdiaginv{k}{\alpha}}$ \\[1em]
\grouptype{$\SO_n^+(q)$}{($n=2k$, $k$ odd)}
 & $q^{k} \minus 1$ & $\cycle{k}{1} \cdot \cycle{k+1}{2k}$ & $\blockdiag{\bdiag{k}{\alpha}, \bdiaginv{k}{\alpha}}$ \\[1ex]
 & $(q^{k-1} \plus  1)(q \plus 1)$ & $\tau_k \cdot \cyclelong{k-1}{1}{k+2}{2k}$ & $\blockdiag{\bdiag{k-1}{\alpha}, \beta, \beta^{-1}, \bdiaginv{k-1}{\alpha}}$ \\[1em]
\grouptype{$\SO_n^+(q)$}{($n=2k=4m$)}
 & $(q^{k-1} \plus  1)(q \plus 1)$ & $\tau_k \cdot \cyclelong{k-1}{1}{k+2}{2k}$ & $\blockdiag{\bdiag{k-1}{\alpha}, \beta, \beta^{-1}, \bdiaginv{k-1}{\alpha}}$ \\[1.5em]
   {\footnotesize $\quad\qquad k \equiv 2 \mod 4$} & $\big(q^{m} \minus 1\big)^2$ & $\cycle{m}{1} \cdot \cycle{k}{m+1} \cdot \cycle{k+1}{k+m} \cdot \cycle{k+m+1}{2k}$ & $\blockdiag{\bdiag{m}{\alpha}, \bdiag{m}{\beta}, \bdiaginv{m}{\beta}, \bdiaginv{m}{\alpha}}$ \\
   {\footnotesize $\quad\qquad k \equiv 0 \mod 4$} & $\big(q^{m} \plus 1\big)^2$ & $\cyclelong{m}{1}{k+m+1}{2k} \cdot \cyclelong{k}{m+1}{k+1}{k+m}$ & $\blockdiag{\bdiag{m}{\alpha}, \bdiag{m}{\beta}, \bdiaginv{m}{\beta}, \bdiaginv{m}{\alpha}}$ \\[1em]
\grouptype{$\SU_n(q)$}{$(n=k+1)$}
 & $\frac{q^{n} - (-1)^n}{q+1}$ & $\tau \cdot \cycle{1}{n}$ & $\bdiagtilde{n}{\alpha}$ \\[1em]
 & $q^{k} \minus  (-1)^k$            & $\tau \cdot \cycle{1}{k}$ & $\blockdiag{\bdiagtilde{k}{\alpha}, \alpha^{-1+q-q^2+\ldots-(-q)^{k-1}}}$ \\[1em]
\grouptype{$\SO_n^-(q)$}{$(n=2k)$}
 & $q^{k} \plus 1$        & $\cycle{1}{k} \cdot \cycle{2k}{k+1}$ &
            $\blockdiag{\bdiag{k-1}{\alpha}, \alpha^{-q^{k-1}}, \alpha^{q^{k-1}}, \bdiaginv{k-1}{\alpha}}$ \\[1ex]
 & $(q^{k-1} \plus 1)(q-1)$ & $\tau_k \cdot \cyclelong{1}{k-1}{2k}{k+2}$ & $\blockdiag{\bdiag{k-1}{\alpha}, \beta, \beta^{-1}, \bdiaginv{k-1}{\alpha}}$
\end{tabular}
\caption{%
Pairs of strong generators for the classical groups.
For any group \mbox{$G=\G^F$}, denote the two diagonal elements listed in the last column by $d_{1/2}$, the Weyl group representatives in the third column by $w_{1/2}$ and let $h_i$ satisfy $h_i \cdot F(h_i)^{-1} = w_i$ ($i=1,2$).
Then the elements $d_i^{\,h_i}$ lie in $G$ and generate~$G$.
By~\cite{MalleSaxlWeigel}, the same is true of any two $\GL_n$-conjugates of~$d_i^{\,h}$ which lie in $G$.}
\label{table:generators}
\end{sidewaystable}

Note that the results in~\cite{MalleSaxlWeigel} on which our arguments are based require each strong generator~$\genname$ to be \emph{regular} (i.e., its centralizer in~$\G$ should have minimal dimension).
A sufficient condition for this is that the diagonal entries of~$\gendiag$ be pairwise distinct.
From Table~\ref{table:generators} it is obvious that for all but the last six tori this can be satisfied by choosing $\alpha$ and~$\beta$ in the way described above.
The other cases require some elementary but sometimes lengthy arguments.
We present an exemplary reasoning for $\SO_{4m}^+(q)$, the remaining groups are treated similarly.
We assume that $q$ is odd.
Let~$\varphi$ denote Euler's totient function (that\nolinebreak{} is, $\varphi(r)$ is the number of positive integers~$<r$ which are coprime to~$r$).
Recall that the number of elements of order~$q^m+1$ in $\F_{q^k}$ equals $\varphi(q^m+1)$.
It suffices to show that this number is greater than~$2m$, for then we can choose $\beta$ such that it differs from all diagonal entries~$\alpha^{\pm q^i}$, $0\leq i \leq m \minus 1$, hence all diagonal entries of~$\gendiag$ are pairwise distinct.
Now the inequality~$\phi(r) \geq \sqrt{r}$ holds for all~$r\not\in\{2,6\}$ (see~\cite[p.\,9]{Sandor_et_al}).
This implies $\phi\,(q^m \plus 1) \geq \sqrt{q^m \plus 1}$ if~\mbox{$(q,m) \neq (5,1)$}.
Moreover, the right hand side is~$>2m$ for~$q \geq 5$ or~\mbox{$q=3$}, $m\geq 4$, which is what we wanted to show.
Using Magma~\cite{Magma}, it is easily checked that there still exist two regular strong generators in~$\SO_{4m}^+(q)$ for all excluded cases except the trivial one $(q,m)=(3,1)$.

%% file: auxiliary_material.tex
\subsection{Auxiliary Material}

We conclude this section with several lemmas.
The first three make statements about the coefficients of polynomials whose roots have certain properties.
\begin{lemma}
\label{lem:symmetric coefficients}
Let \mbox{$\displaystyle f(X) = X^{n} + a_{n-1}\cdot X^{n-1} + \dots + a_1\cdot X + a_0$} be a polynomial such that for every root~$\alpha$, its inverse~$\alpha^{-1}$ is also a root but $\alpha \neq \alpha^{-1}$.
Then $f$ is ``palindromic'' in the sense that its coefficients are symmetric, i.e., that~$a_{n-i} = a_i$ holds for all~$i\leq n$ (with $a_n = 1$).
\end{lemma}
\begin{proof}
Up to sign, the coefficient $a_{n-i}$ equals the sum of all possible products of precisely~$i$ roots.
Since by assumption all roots multiply up to~$1$, every product of $i$ roots is at the same time the product of the inverses of the remaining~$(n-i)$ roots, hence again by assumption a product of~$(n-i)$ (different) roots.
This means that the sum of all possible products of~$i$ roots equals the sum of all possible products of~$(n-i)$ roots, which proves~$a_{n-i} = a_i$.
Note that the signs of the two coefficients coincide since $f$ has even degree.
\end{proof}

\begin{lemma}
\label{lem:antisymmetric coefficients}
If $f(X) = X^{n} + a_{n-1} \cdot X^{(n-1)} + \dots + a_{1} \cdot X + a_{0}$ is divisible by \mbox{$(X-1)$} and all roots~\mbox{$\alpha\!\neq\!1$} of~$f$ have the property that~$\alpha^{-1}$ is also a root but that \mbox{$\alpha \neq \alpha^{-1}$}, then the coefficients of~$f$ are antisymmetric in the sense that $a_{n-i} = -a_i$ for all~$i\leq n$ (where again we set $a_0 \coloneqq 1$).
\end{lemma}
\begin{proof}
By the previous lemma, the polynomial~$g=\frac{1}{(X-1)}\cdot f$ is symmetric.
Comparing coefficients yields the result.
\end{proof}

\noindent By an argument analogous to the proof of Lemma~\ref{lem:symmetric coefficients} we also obtain the following result.
\begin{lemma}\label{lem:symmetry SU}
Let $f=X^n + a_{n-1} X^{n-1} + \dots + a_1 X + (-1)^n \in \mathbb{F}_{q^2}[X]$ be a separable polynomial such that for every root $\alpha$ of $f$, $\alpha^{-q}$ is also a root.
Then its coefficients have the following symmetry: $a_i=(-1)^n \cdot a_{n-i}^q$ for all $i\leq n$.
\end{lemma}

\noindent
The next two lemmas will be used to simplify certain matrix transformations and determinant computations for matrices of a very special shape.
Let $K$ be a field, and consider a square matrix~$B$ of the following form.
\newcommand{\Bbitriang}{\widetilde{B}}
\begin{equation}
  \label{eqn:almost bitriangular matrix}
  B = \left(
  \begin{array}{c@{\;\;}c@{\;\;}c@{\;\;}c@{\;}c |c| cc@{\;\;\;}c@{\;\;\;}c}
       X+b_1 & b_2 &\ldots & b_{{r}-1} & b_{r} & \bigstar & \beta & & &\\
       -1 &  X &        &    &    &    & & & &\\[-0.5em]
          & -1 & \ddots &    &    &    & & & &\\[-0.5em]
          &    & \ddots &  \hspace*{-0.5em}X &    & & & &\\
          &    &        & \hspace*{-0.5em}-1 & X  & & & &\\
\hline
          &    &        &    & \bigstar & \bigstar & \bigstar & &\\
\hline
     & & & & c_{{r}-1}&   & X & -1 &        &    \\
     & & & & \vdots &   &   & \ddots & \ddots &    \\
     & & & & c_1    &   &   &   &  X     & -1  \\
     & & & & c_0    &   &   &   &        & X
\end{array}\right).
\end{equation}
Here $X$ is an indeterminate, $b_i, c_i \in K$ and blank fields indicate zeros.
To account for all necessary cases, we adopt the convention that~$B$ may possess one or more optional rows and columns in the middle, as indicated by the asterisks (provided that it remains a square matrix).
We stipulate, however, that in the upper middle extra block only the first row contains additional nonzero elements, in the left middle extra block only the ${r}$-th column contains nonzero entries, and in the right middle extra block only the first column contains nonzero elements (as suggested by the asterisks).
Solely the central block is allowed to be filled arbitrarily.

As preparation for the first lemma we define the following two families of polynomials.
By convention, we set~$b_0 = 1$.
The other coefficients are the entries $b_i$ and $c_i$ appearing in the definition of~$B$.
\begin{equation}
  g_j(X) \coloneqq \sum_{i=0}^{j} b_i X^{j-i} \;\;\quad (j \leq r),
  \qquad
  h_j(X) \coloneqq \sum_{i=0}^{j} c_i X^i \;\;\quad (j \leq r\minus 1).
  \label{eqn:polynomials g_j and h_j}
\end{equation}
It is easy to see that $g_j$ and~$h_j$ satisfy the following recursive relations, which will be used in the proof of the next lemma.
\begin{equation}
   g_{j+1}(X) = X \cdot g_j(X) + b_{j+1}, \qquad h_{j}(X) = c_{j} \cdot X^j +h_{j-1}(X).
   \label{eqn:recursive relations of g_j and h_j}
\end{equation}
\begin{lemma}
\label{lem:transformation into bitriangular matrix}
If~$B$ is as above in~(\ref{eqn:almost bitriangular matrix}), it can be transformed by elementary row transformations into a matrix~$\Bbitriang$ of the following shape.
\begin{equation}
  \label{eqn:bitriangular matrix}
  \Bbitriang = \left(
  \begin{array}{c@{\;\;}c@{\;\;}c@{\;\;}c@{\;}c |c| c@{}c@{\;\;}c@{\;\;}c}
       g_1 & b_2 & \ldots & b_{{r}-1} &   b_{r}   &    \bigstar   &  \beta  & & & \\
           & g_2 &        & \bullet & \bullet & \bullet & \bullet & & & \\[-0.5em]
           &     & \ddots & \bullet & \bullet & \bullet & \bullet & & & \\[-0.0em]
           &     &        & g_{{r}-1} & \bullet & \bullet & \bullet & & & \\
           &     &        &         &   g_{r}   & \bullet & \bullet & & & \\
\hline
           &    &        &  &  \bigstar  & \bigstar &   \bigstar   & &\\
\hline
     & & & & h_{{r}-1}&   & X^{{r}} &        &     &   \\[-0.3em]
     & & & &     \vdots     &   &               & \ddots &     &   \\[0.2em]
     & & & &     h_1        &   &               &        & X^2 &   \\[0.2em]
     & & & &     h_0        &   &               &        &     & X
\end{array}\right),
\end{equation}
where $g_i$ and $h_i$ are the polynomials defined above and the symbol $\bullet$ is a placeholder for the topmost entry of the column in which it occurs.
The determinants of the original and the transformed matrix are related by
\begin{equation}
  \label{eqn:relationship between original and transformed bitriangular matrix}
  \det(\Bbitriang) \;\;=\;\; \det(B) \;\;\cdot\;\; \prod_{i=1}^{{r}-1} g_i \;\;\cdot\;\; \prod_{i=1}^{{r}-1} X^i.
\end{equation}
\end{lemma}
\begin{proof}
By definition of $g_1$ and $h_0$, the first and last row of $B$ and $\Bbitriang$ already agree.
The claim follows by inductive elimination of the subdiagonal entries in the upper left block (from top to bottom) and of the superdiagonal entries in the lower right block (from bottom to top), using~(\ref{eqn:recursive relations of g_j and h_j}).
The row multiplications required during these transformations obviously alter the determinants as specified by~(\ref{eqn:relationship between original and transformed bitriangular matrix}).
\end{proof}

\begin{cor}
\label{cor:det of bitriangular matrix}
Let~$B$ be as in~(\ref{eqn:almost bitriangular matrix}). Then $\det(B) = \det(B')$, where
\begin{equation*}
  B' = \left(
  \begin{matrix}
      g_r    & \bullet  & \beta \\
    \bigstar & \bigstar & \bigstar \\
     h_{r-1} &  0       & X^r
  \end{matrix} \right),
\end{equation*}
is the middle block of~$\Bbitriang$ as in~(\ref{eqn:bitriangular matrix}).
\end{cor}
\begin{proof}
$\Bbitriang$ is a block matrix of the form $\Bbitriang=$%
{\footnotesize$\left(
  \begin{array}{c|c|c}
    U & \bigstar & 0 \\
    \hline
    0 & B'\vphantom{\Big(} & 0 \\
    \hline
    0  & \bigstar & L
  \end{array}
  \right)$}, where $U$ is upper triangular with diagonal entries $g_1,\ldots,g_{r-1}$ and $L$ is lower triangular with diagonal entries $X^{r-1},\ldots,X$.
We compute $\det(\Bbitriang) = \det(U) \cdot \det(B') \cdot \det(L) = \big(\prod_{i=1}^{r-1} g_i\big) \cdot \det(B') \cdot \big(\prod_{i=1}^{r-1} X^i\big)$ and use~(\ref{eqn:relationship between original and transformed bitriangular matrix}).
\end{proof}

%% file: overview_of_the_method.tex
\subsection{Overview of the Method}
\label{sec:overview of the method}

In the following sections we are going to compute for certain groups of Lie type~$G=\G^F$ a\nolinebreak{} Frobenius module~$(M,\Phi)$ over a rational function field~$\grdfield$ such that $\GalPhi(M)=G$.
This is achieved by using as representing matrix~$D$ of~$\Phi$ the cross section for this group obtained from Theorem~\ref{thm:cross-section} (in some cases with slight adaptions).
Using the upper bounds Theorem~\ref{thm:upperbound} and~\ref{thm:twupperbound} we conclude that~$(M,\Phi)$ has Galois group a subgroup of~$G$.
In order to prove equality, we show that $D$ can be specialized to the two strong generators~$\genname_{1/2}$ listed in Table~\ref{table:generators}.
The argument goes via the characteristic polynomial~$\charpoly{D}$ of~$D$.
Its coefficient usually have a certain kind of symmetry and the same holds for the characteristic polynomials of~$\genname_{1/2}$.
We use this to prove that~$D$ can be specialized to matrices which have the same separable characteristic polynomial as~$\genname_{1/2}$ and hence the same Jordan normal form.
Thus the specializations are conjugate to~$\genname_{1/2}$ so that by definition they form a set of strong generators, too.
By Theorem~\ref{thm:lowerbound}, we have~$\GalPhi(M) = G$.

The construction of the cross section matrix~$D$ always follows a similar pattern.
First we fix the particular ground field for the to-be-constructed Frobenius module~$M$.
In most cases this is~$\grdfield=\Fq(t_1,\ldots,t_k)$, where $k=\rank(\G)$, but sometimes we need an additional parameter to ensure that certain equations are solvable.
Next we fix the particular representation for~$\G$ with which we are going to work and list the \mbox{1-parameter} subgroups (which are parametrized by $t_1,\ldots,t_k$) associated to all simple roots of~$\G$ as well as the corresponding Weyl group representatives.
Being well-known, they are presented without further calculations; they can easily be computed following the instructions in~\cite{Carter89} (recall that they are always considered with respect to the standard torus and standard Borel subgroup of~$\G$).
Next we define ``partial'' cross sections
$\pcs_j \coloneqq X_1 w_1 X_2 w_2 \dots X_j w_j$
which arise by considering only the first $j$ factors in Steinberg's cross section~(\ref{eqn:cross-section}), for some $j \leq k$.
By a simple inductive argument, we derive the general shape of a typical element of~$\pcs_j$.
The interim approach then is to choose $D=\pcs_k(t_1,\ldots,t_k)$.
For most of the groups this definition already works.
In the case of the orthogonal groups we need an additional parameter to achieve sufficient specializations.
In some of the twisted cases (more precisely, in the cases where one or more roots are fixed by the action of $F$ on the simple roots) it is not possible to choose $D$ such that it both lies in the cross section and satisfies the upper bound criterion~\ref{thm:upperbound}.
The final definition of~$D$ then depends on the group and is usually close to Steinberg's cross section.
As a last step we compute the characteristic polynomial of $D$ and show that it specializes to those of the strong generators.
This completes the proof that $\GalPhi(M)=G$.

Afterwards, we compute explicit polynomials for each group.
This is done by considering the equation~$D\cdot\phiq(x) = x$ $\;(x \in M)$ which defines the solution space of~$M$ (see section~\ref{sec:frobenius modules}).
Recall that \mbox{$\GalPhi(M)$} is defined as $\Gal(\extfield/\grdfield)$, where $\extfield$ is the solution field of~$M$, i.e., the smallest field containing all coordinates of a full system of solutions to the above vector equation.
We expand it into a system of equations for the coordinates of~$x$ and eliminate all but one of the variables, ending up with a single univariate polynomial~$f$ whose roots generate~$\extfield$.
This polynomial has~$\extfield$ as its splitting field, hence~$\Gal(f) = \Gal(\extfield/\grdfield) = G$.

In the following sections we freely mix the use of $k$ (the rank of $\G$) and $n$ (the dimension of its natural representation).
Note that while calculating the cross section matrices, we make heavy use of the notation introduced in section~\ref{sec:notation}.

%% file: polynomials.tex
\section{Polynomials for the Classical Groups}
\label{sec:polynomials for the classical groups}
\input{polynomial_for_Ak.tex}
\input{polynomial_for_Ck.tex}
\input{polynomial_for_Bk.tex}
\input{polynomial_for_Dk.tex}

\section{Polynomials for the Twisted Classical Groups}
\label{sec:polynomials for the twisted classical groups}
\input{polynomial_for_2Ak.tex}

\input{polynomial_for_2Dk.tex}

\section{Polynomials for Exceptional Groups of Lie Type}
\label{sec:polynomials for exceptional groups of lie type}
\input{polynomial_for_2B2.tex}
\input{polynomial_for_G2_2G2.tex}

\input{polynomial_for_3D4.tex}

%% file: polynomial_for_Ak.tex
\subsection{The Special Linear Groups $\SL_{k+1}(q)$}

The special linear group is defined as \mbox{$\SL_n(q) = \{ A \in \GL_n(q) \;|\; \det(A) = 1 \}$}.
Its root system has $k$ simple roots, where $n=k+1$.
As ground field of the Frobenius module we choose $\grdfield = \Fq(t_1,\ldots,t_k)$.
For the $i$-th root subgroup and the $i$-th Weyl group representative we choose $\root{i}{t_i}$ and~$\weyl_i$, respectively
(the matrices $E_i$ and $\weyl_i$ are defined in section~\ref{sec:notation}).
With the exemplary calculation of section~\ref{sec:notation} in mind, an easy inductive argument shows that the partial cross section matrix $\pcs_j$, which depends on the parameters~$t_1,\ldots,t_j$, has the following shape for $1\leq i\leq n$.
\begin{equation*}
\pcs_j(t_1,\ldots,t_j) = \idblock{\textbf{(j+1)},\underdash}{\companion[-]{-t_1, \ldots, -t_j, 1}} \quad \in\; \SL_n(\grdfield).
\end{equation*}
Indeed, the case $j=1$ is clear by definition, and multiplication by the next factor $\root{j+1}{t_{j+1}} \cdot \weyl_{j+1}$ extends the companion matrix in the upper left block by one row and one column.
In principle, according to the strategy outlined in Section~\ref{sec:overview of the method}, we would choose the full cross section matrix $\pcs_k(t_1,\ldots,t_k)$ as the representing matrix of our Frobenius module.
However, to make subsequent computations slightly easier, we adapt the the signs of the entries to obtain the following matrix.
\begin{equation*}
D =
\left(\begin{array}{c@{\;\;}c@{\;\;}cc}
          (-1)^k\, t_1 & \ldots & (-1)^k\, t_k & (-1)^k\\
 \phantom{(-1)^k}    1 &        &     &   \\
            & \ddots &     &   \\
            &        &  1  & 0 \\
    \end{array}\right) \quad \in\; \SL_n(\grdfield).
\end{equation*}
The general form of the characteristic polynomial of a matrix of this shape is well-known.
For the specific matrix~$D$ we obtain
\begin{equation*}
  h(X) = X^{n} + (-1)^n\, t_1 \cdot X^{n-1} + \ldots + (-1)^n\, t_k \cdot X + (-1)^{n}.
\label{eqn:charpoly of special linear cross section}
\end{equation*}

\smallskip
\begin{thm}
Let $(M,\Phi)$ be the n-dimensional Frobenius module over $\grdfield = \Fq(t_1,\ldots,t_k)$ such that the representing matrix of $\Phi$ with respect to some basis equals $D$.
Then
\begin{enumerate}
  \item $\GalPhi(M) \cong \SL_n(q).$
  \item The solution field of $M$ is generated by the roots of the additive polynomial
    $$\displaystyle f(X) \;=\; X^{q^n} +\; \sum_{i=1}^{n-1} t_{i} \cdot X^{q^{i}} \;+\; (-1)^n \cdot X.$$
    In particular, $\Gal(f) \cong \SL_n(q)$.
\end{enumerate}
\end{thm}
\noindent
Note that~$f$ is a kind of Dickson polynomial\footnote{For a very nice reference of the related theory see~\cite{Elkies}, for example.} of dimension~$n$.
Incidentally, the polynomial for~$\SL_n(q)$ given by Elkies~\cite{Elkies} in his Theorem~2 is the same as the one we derived above using completely different methods.
\begin{proof}
a) We have $D \in \SL_n(\grdfield)$, hence $\GalPhi(M) \leq \SL_n(q)$ by Theorem~\ref{thm:upperbound}.
In order to prove equality, we need to show that for each member~$\genname$ of the pair of strong generators of~$\SL_n(q)$ (see Table~\ref{table:generators}), the indeterminates~$t_1,\ldots,t_k$ can be specialized such that~$h(X)$ specializes to the characteristic polynomial~$g(X)$ of~$\genname$.
This is the case because all~$t_i$ can be chosen arbitrarily and the constant coefficient of~$g$ equals $(-1)^n\cdot\det(\genname) = (-1)^n$.

\medskip\noindent
b) We solve the defining equation $D\cdot\phiq(x) = x$ for the solution space of the Frobenius module~$(M,\Phi)$.
With $x=(x_1,\ldots,x_n)$, it is equivalent to the following system of equations.
\begin{subequations}
\begin{align}
  (-1)^k \cdot \Big(\sum_{i=1}^{k}   \;t_i \cdot x_i^{\;q} \Big) +  (-1)^k \cdot (x_{k+1})^{\;q} \quad &= \quad x_1,
               \label{eqn:special linear frobenius equation for x_1} \\[0ex]
     (x_{i-1})^{q} \quad &= \quad x_i \qquad \textrm{for $i\in\{2,\ldots,k\plus1\}$}.
	       \label{eqn:special linear frobenius equation for x_i}
\end{align}
\end{subequations}

\noindent
For simplicity, we abbreviate~$x_1$ by~$X$.
By induction, (\ref{eqn:special linear frobenius equation for x_i}) implies $x_i = X^{q^{i-1}}$.
We insert this into the first equation, bring everything to the left hand side and normalize.
This leaves us with the following additive polynomial.
\begin{equation*}
  f(X) = X^{q^n} \;+\; \sum_{i=1}^k t_i \cdot X^{q^i} \;+\; (-1)^n \cdot X \;\;=\;\; 0.
\end{equation*}
Our derivations above show that $M$ and $f$ satisfy the assumptions of Lemma~\ref{lem:galoisgroup}.
Hence $\Gal(f) = \GalPhi(M) = \SL_n(q)$.
\end{proof}

%% file: polynomial_for_Ck.tex
\subsection{The Symplectic Groups $\Sp_{2k}(q)$}

Let $n=2k$ and denote by $J = \big(\begin{smallmatrix} 0  & J_0 \\ -J_0 & 0 \end{smallmatrix}\big)$ the representing matrix of a bilinear form on~$\Fq^n$, where $J_0 = \left(\begin{smallmatrix} & & 1 \\[-1ex] & \iddots & \\ 1 & & \end{smallmatrix}\right) \in \GL_k(q)$.
The symplectic group is defined as $\Sp_{2k}(q) = \{ A \in \GL_{2k}(q) \;|\; A\tr \cdot J \cdot A = J \}$.
Note that many authors use a slightly different bilinear form but the resulting groups are isomorphic.
We choose $\grdfield=\Fq(t_1,\ldots,t_k)$ as the ground field of the Frobenius module.
The root subgroups for the short roots (that is, for $1 \leq i \leq k-1$) have the form $\blockdiag{\root[k]{i}{t_i},\root[\,k]{-i}{-t_i}}$ whereas the corresponding Weyl group representatives are~$\blockdiag{\weylpos_{i,k},\weylpos_{-i,k}}$ (recall our convention about negative indices from section~\ref{sec:notation}).
The single long root has associated root subgroup $\root[n]{k}{t_k}$ and Weyl group representative $\weyl_{k,n}$.
With this choice, the partial cross section matrices~$\pcs_j$ assume the following shape for~$j<k$.
\begin{equation*}
  \pcs_j(t_1,\ldots,t_j) = \idblock{\mathbf{(j+1)},\underdash,\mathbf{(j+1)}}{\companion{t_1,\ldots,t_j,1},\companionrev{1,-t_1,\ldots,-t_j}}.
\end{equation*}
In other words, $\pcs_j$ is obtained from a $\kxk{(2k)}$-identity matrix by replacing both the upper left and lower right $\kxk{(j\plus 1)}$-block with a kind of companion matrix.
This is easy to verify since multiplication by the next factor simply extends both the upper left and lower right blocks by one row and column.
To obtain the full cross-section, we need to multiply $\pcs_{k-1}(t_1,\ldots,t_{k-1})$ by the last factor~$\root[n]{k}{t_k} \cdot \weyl_{k,n}$.
This yields
\begin{align*}
\pcs_{k}(t_1,\ldots,t_k) &=
  \left(
  \begin{array}{c@{\;\;}c@{\;\;}c@{\;}c | cc@{\;\;\;}c@{\;\;\;}c}
       t_1 & \ldots & t_{k-1} & -t_k & 1 & & &\\
         1 &        &         &      &   & & &\\
           & \ddots &         &      &   & & &\\
           &        &      1  &   0  &   & & &\\
      \hline
     & & & t_{k-1} &    0   & 1 &        &    \\
     & & & \vdots  & \vdots &   & \ddots &    \\
     & & &   t_1   &    0   &   &        & 1  \\
     & & &    -1   &    0   &   &        & 0
  \end{array}
  \right)
\end{align*}
The characteristic polynomial $h(X)$ of this matrix is the determinant of~$B \coloneqq (X \cdot \id_{2k} - \pcs_k)$.
Note that~$B$ has precisely the form of~(\ref{eqn:almost bitriangular matrix}) (with~\mbox{$r=k$} and without any extra rows or columns), where
\begin{alignat*}{4}
  b_i     &= -t_i, \qquad &c_i &= -t_i   \quad&&&&\quad\text{(for $1\leq i\leq k\minus 1$)}, \\
  b_k     &= \phantom{-}t_k,
 &\beta   &= -1, \qquad
 &c_0     &= 1.
\end{alignat*}
It therefore satisfies the conditions of Corollary~\ref{cor:det of bitriangular matrix}, which asserts that
\begin{equation*}
  \det(B) =
  \det\begin{pmatrix} 
    g_k & -1 \\
    h_{k-1} & X^k
  \end{pmatrix}
  = g_k \cdot X^k + h_{k-1}.
\end{equation*}
We rewrite this expression by expanding the polynomials~$g_k$ and~$h_{k-1}$ according to~(\ref{eqn:polynomials g_j and h_j}) and resubstituting~$b_i$ and~$c_i$, which yields the generic characteristic polynomial~$h(X)$ of the cross section of~$\Sp_{2k}$.%
\begin{align}
     h(X)
          &= X^{2k} - \sum_{i=1}^{k-1} t_i X^{2k-i} + t_k X^k
                    - \sum_{i=1}^{k-1} t_i X^{i} + 1.
	    \label{eqn:characteristic polynomial of symplectic cross section}
\end{align}

\newcommand{\eqdist}{\;\;}%
\bigskip
\begin{thm}
\label{thm:frobenius module for C_k}
Let $(M,\Phi)$ be the n-dimensional Frobenius module over $\grdfield=\Fq(t_1,\ldots,t_k)$ such that the representing matrix of $\Phi$ equals $D=\pcs_k(t_1,\ldots,t_k)$ for some basis.
Then
\begin{enumerate}
  \item $\GalPhi(M) \cong \Sp_{2k}(q)$.
  \item The solution field of $M$ is generated by the roots of the additive polynomial
    \begin{equation*}
      f(X) =
      X \eqdist-\eqdist \sum_{i=1}^{k-1} t_i\,X^{q^i}
        \eqdist+\eqdist t_k\,X^{q^k}
        \eqdist-\eqdist \sum_{i=1}^{k-1} (t_{k-i})^{q^{i}}\,X^{q^{k+i}}
        \eqdist+\eqdist X^{q^{2k}} \eqdist=\eqdist 0.
    \end{equation*}
    In particular, $\Gal(f) \cong \Sp_{2k}(q)$.
\end{enumerate}
\end{thm}
\begin{proof}
a) We have $D \in \Sp_{2k}(\grdfield)$ by construction.
Hence, $\GalPhi(M) \leq \Sp_{2k}(q)$ by Theorem~\ref{thm:upperbound}.
Equality is proved again via the characteristic polynomial.
First note that the coefficients of~$h(X)$ are symmetric.
Table~\ref{table:generators} shows that for each strong generator~$A$ of~$\Sp_{2k}(\Fq)$, the characteristic polynomial~$\charpoly{A}$ satisfies the assumptions of Lemma~\ref{lem:symmetric coefficients} since the diagonalization of~$A$ has the form $\diag(\alpha,\alpha^q,\ldots,\alpha^{q^{k-1}},\alpha^{-q^{k-1}},\ldots,\alpha^{-q},\alpha^{-1})$.
Hence the coefficients of~$\charpoly{A}$ are symmetric, too, and we can specialize the indeterminates $t_i$ in such a way that $h(X)$ specializes to $\charpoly{A}(X)$, whereby $D$ specializes to a conjugate of~$A$.
Since this applies to both generators, Theorem~\ref{thm:lowerbound} asserts that $\GalPhi(M) \geq \Sp_{2k}(q)$.

\medskip\noindent
b) We consider the equation~$D\cdot\phiq(x) = x$ which defines the solution space $(M,\Phi)$.
Letting $x=(x_1,\ldots,x_k,y_1,\ldots,y_k)\tr$, we expand it into the following system of equations
\begin{subequations}
\begin{align}
  \sum_{i=1}^{k-1} t_i \cdot x_i^{\;q} - t_k \cdot x_k^{\;q}
                                       + y_1^{\;q} \quad &= \quad x_1,
               \label{eqn:symplectic frobenius equation for x_1} \\[-1ex]
  (x_{i-1})^{q} \quad &= \quad x_i \qquad \textrm{for $i\in\{2,\ldots,k\}$},
	       \label{eqn:symplectic frobenius equation for x_i} \\[1em]
  t_{k-i}\cdot x_k^{\;q} \;\;+\;\;  (y_{i+1})^{q} \quad &= \quad y_i \qquad \textrm{for $i\in\{1,\ldots,k\minus 1\}$},
               \label{eqn:symplectic frobenius equation for y_i} \\
  -x_k^{\;q} \quad &= \quad y_k.
               \label{eqn:symplectic frobenius equation for y_k}
\end{align}
\end{subequations}%
Abbreviate $x_1$ by~$X$ and observe that equation~(\ref{eqn:symplectic frobenius equation for x_i}) recursively implies
\begin{equation}
\label{eqn:symplectic frobenius equation for x_i, direct}
  x_i = X^{\,q^{(i-1)}} \qquad\textrm{for $i\in\{1,\ldots,k\}$}.
\end{equation}
In particular,~$x_k=X^{q^{k-1}}$.
This trans\-forms~(\ref{eqn:symplectic frobenius equation for y_k}) into~$y_k = -X^{q^k}$.
Working our way backwards through~(\ref{eqn:symplectic frobenius equation for y_i}) by letting $i= (k\minus1),\ldots,1$ and reusing the expression for~$y_{i+1}$ obtained in the previous step, we deduce
\begin{equation}
\label{eqn:symplectic frobenius equation for y_i, direct}
  y_{k-\nu} \eqdist = \eqdist \sum_{i=0}^{\nu-1}
                            (t_{\nu-i})^{q^i} \cdot
                            X^{q^{(k+i)}} \eqdist -
                            \eqdist X^{q^{k+\nu}}, \qquad
                            \nu \in \{0,\ldots,(k\minus1)\}.
\end{equation}
In particular (setting $\nu = k\minus1$ and shifting the summation limits up by~1),
\begin{equation}
\begin{split}
  y_1 \eqdist &= \eqdist \sum_{i=1}^{k-1} (t_{k-i})^{q^{i-1}} 
	         \cdot X^{q^{(k+i-1)}} \eqdist - \eqdist X^{q^{(2k-1)}}.
\end{split}
\end{equation}
We plug the last expression for~$y_1$ into equation~(\ref{eqn:symplectic frobenius equation for x_1}), make the substitutions~\mbox{$x_i^{\;q}=X^{q^i}$} implied by~(\ref{eqn:symplectic frobenius equation for x_i, direct}) and bring everything to the right hand side to obtain the polynomial $f(X)$ from the statement of the theorem.
By~(\ref{eqn:symplectic frobenius equation for x_i, direct}) and~(\ref{eqn:symplectic frobenius equation for y_i, direct}) it is clear that every choice of $X=x_1$ uniquely determines the remaining coordinates.
The assertion now follows from Lemma~\ref{lem:galoisgroup}.
\end{proof}

\noindent
Since the $t_i$ are algebraically independent, we are free to adapt the signs of the coefficients of~$f$ without altering its Galois group.
We present a few examples of~$f$ (with altered signs) for small ranks~$k$.
\begin{alignat*}{5}
&k=2: &\qquad  f(X) &= X^{q^4} &&+ t_1^{\;q} X^{q^3}
                                 + t_2 X^{q^2} 
			         + t_1 X^q
                                 + X, \\
&k=3: &\qquad  f(X) &= X^{q^6} &&+ \big( t_1^{\;q^2} X^{q^5} + t_2^{\;q} X^{q^4} \big)
                                 + t_3 X^{q^3}
                                 + \big( t_2 X^{q^2} + t_1 X^{q} \big)
			         + X,
 \\
&k=4: &\qquad  f(X) &= X^{q^8} &&+ \big( t_1^{\;q^3} X^{q^7} + t_2^{\;q^2} X^{q^6}  + t_3^{\;q} X^{q^5} \big)
                                 + t_4 X^{q^4}\\
                            &&&&&\qquad\qquad\;\;\;+ \big( t_3 X^{q^3} + t_2 X^{q^2} + t_1 X^{q} \big)
			         + X.
\end{alignat*}%
It is interesting to note that our polynomials are almost identical to the ones obtained by Elkies for~$\Sp_{2k}(q)$; see~\cite{Elkies}, especially Eqn.\;(5.1) and Thm.\;7.
In particular, they all share the special $q$-palindromic shape.

%% file: polynomial_for_Bk.tex
\subsection{The Odd-Dimensional Special Orthogonal Groups $\SO_{2k+1}(q)$}

Throughout this section we assume that $q$ is odd because the groups $\SO_{2k+1}(\Fq)$ only exist in this case.
Let~$n=2k+1$ and denote by $J = \left(\begin{smallmatrix} & & 1 \\[-1ex] & \iddots & \\ 1 & & \end{smallmatrix}\right) \in \GL_n(q)$ the representing matrix of a bilinear form on~$\Fq^n$.
The special orthogonal group is defined as $\SO_{n}(q) = \{ A \in \GL_{n}(q) \;|\; A\tr \cdot J \cdot A = J \}$.
To construct a Frobenius module for it, we work over the ground field $\grdfield=\Fq(t_1,\ldots,t_k,s)$.
Apart from the usual indeterminates~$t_1,\ldots,t_k$ (one for each simple root), it contains an additional indeterminate~$s$ which is needed so that some quadratic equations become solvable later on.
The root subgroups are $\blockdiag{\root[k]{i}{t_i},1,\root[k]{-i}{-t_i}}$ for \mbox{$i<k$} and 
$\idblock{(k-1),\mathbf{3},(k-1)}{\threemat{1}{t_k}{-\frac{t_k^2}{2}}{0}{1}{-t_k}{0}{0}{1}}$ for $i=k$.
As Weyl group representatives we choose the matrices~$\blockdiag{\weylpos_{i,k},1,\weylpos_{-i,k}}$ if $2 \leq i \leq k\minus 1$.
For $i=1$, however, we multiply the ``standard representative'' $\blockdiag{\weylpos_1,1,\weylpos_{-1}}$ by the torus element $\idblock{\mathbf{2},\underdash,\mathbf{2}}{\twomat{1}{0}{0}{s},\twomat{1/s}{0}{0}{1}}$.
This leads to another representative of the Weyl group, namely $w_1 \coloneqq \idblock{\mathbf{2},\underdash,\mathbf{2}}{\twomat{0}{s}{1}{0},\twomat{0}{1}{1/s}{0}}$.
The reason for this adaption is again to ensure the solvability of certain equations.
Finally, the last Weyl group representative for $i=k$ is
$\idblock{(k-1),\mathbf{3},(k-1)}{\threemat{0}{0}{-1}{0}{-1}{0}{-1}{0}{0}}$.
Now an inductive argument shows that the partial cross section $\pcs_j$ has the following shape for $j \leq k\minus 1$.
\begin{equation*}
  \pcs_j(t_1,\ldots,t_j,s) = \idblock{\mathbf{(j+1)},\underdash,\mathbf{(j+1)}}{%
    \companion{t_1, s t_2, \ldots, s t_j, s},
    \companionrev{\frac{1}{s}, -\frac{t_1}{s}, -t_2, \ldots, -t_j}
  }
\end{equation*}
Multiplying this matrix by the last factor
$\idblock{(k-1),\mathbf{3},(k-1)}{
           \threemat{\frac{t_k^2}{2}}{-t_k}{-1}
                          {t_k}{-1}{0}
                          {-1}{0}{0}}$
leads to the full cross-section matrix
\begin{equation}
\label{eqn:odd orthogonal cross section}
  \pcs_k = \left(
    \begin{array}{c@{\;\;}c@{\;\;}c@{\;}c c|c|c cc@{\;\;\;}c@{\;\;\;}c}
       t_1 & s\,t_2 & \ldots & s\,t_{k\minus 1} & \frac{s}{2}t_k^{\;2} & -s\,t_k & -s & & & & \\
         1 &        &        &                  &                      &         &   & & & & \\
           &      1 &        &                  &                      &         &   & & & & \\
           &        & \ddots &                  &                      &         &   & & & & \\
           &        &        &               1  &                   0  &         &   & & & & \\
      \hline
           &        &        &                  &                  t_k &      -1 &   & & & & \\
      \hline
           &        &        &                  &            t_{k-1}   &         & 0 & 1 &   & & \\
           &        &        &                  &            \vdots    &         &   & \ddots & \ddots & \\
           &        &        &                  &              t_2     &         &   &   & 0 & 1 & \\[0.5em]
           &        &        &                  &        \frac{t_1}{s} &         &   &   &   & 0 & 1 \\[0.5em]
           &        &        &                  &        -\frac{1}{s}  &         &   &   &   &   & 0 \\
    \end{array}
    \right).
\end{equation}
The characteristic polynomial~$h$ of~$\pcs_k$ is the determinant of $B=(X \cdot \id_n - \pcs_k)$.
This matrix has the general shape of~(\ref{eqn:almost bitriangular matrix}), with the correspondences:
\begin{alignat*}{5}
  b_1 &= -t_1,                           &   \qquad b_i &= -s t_i
      &  \qquad\text{for } &i\in\{2,\ldots,k\minus1\}, \\
  b_k &= -\frac{s}{2} t_k^{\;2},   &   \qquad \beta   &= s, \\
  c_0 &= \frac{1}{s},              &   c_1 &= -\frac{t_1}{s},
      &   c_i &= -t_i \quad\text{for } i\in\{2,\ldots,k\minus1\}.
\end{alignat*}
Unlike in the symplectic case, $B$ now has one extra middle row and column containing three nonzero additional elements (with values~$s t_k$, $-t_k$ and~\mbox{$X \plus 1$}).
Corollary~\ref{cor:det of bitriangular matrix} asserts that
\begin{align*}
  h(X) = &\det \begin{pmatrix}
                  g_k     & s t_k & s \\
		-t_k     &   X\plus1   &    0    \\
		  h_{k-1} &      0      &   X^k
  \end{pmatrix} \\[1.0em]
    = &g_k(X) \cdot (X\plus1) \cdot X^{k} \;\;+\;\; s t_k^{\;2} \cdot X^{k} \;\;-\;\; s \cdot (X\plus1) \cdot h_{k-1}(X).
\end{align*}
We expand products, substitute $g_k(X)$ and $h_{k-1}(X)$ according to~(\ref{eqn:polynomials g_j and h_j}) and merge the resulting sums, shifting summation indices where necessary.
Then we replace $s \cdot c_i$ with $b_i$ ($0 \leq i < k$) in the right sum (which results from the substitution of $h_{k-1}$) and rearrange terms.
With the conventions $b_{-1} = b_{k+1} = c_{-1} = c_k = 0$ and $b_0=1$ this yields the following very symmetric form of the characteristic polynomial of~$\pcs_k$:
\begin{align}
     h(X) &= \sum_{i=0}^{k}   (b_{i-1} \plus b_i) \cdot X^{(2k+1)-i}
             \;\;-\;\;
            \sum_{i=0}^{k} (b_{i-1} \plus b_i) \cdot X^{i}.
     \label{eqn:charpoly_Bk}
\end{align}

\smallskip
\begin{thm}
Let $(M,\Phi)$ be the n-dimensional Frobenius module over $\grdfield=\Fq(t_1,\ldots,t_k,s)$ such that the representing matrix of $\Phi$ equals $D=\pcs_k(t_1,\ldots,t_k,s)$ for some basis.
Then
\begin{enumerate}
  \item $\GalPhi(M) \cong \SO_{2k+1}(q)$.
  \item The solution field of $M$ is generated by the roots of the additive polynomial
    \begin{alignat}{3}
      f(X) = -\;&\sum_{i=0}^k \Big( b_{i-1}\toq + (s\,t_k)^{\,q-1} \, b_i \Big) \cdot \;X^{q^i} \notag \\
               \eqdist - \; &\Big( \frac{s^q}{2}\,t_k^{\,2q} + s^q\,t_{k-1}\toq\, t_k^{\,q-1} \Big) \cdot \;X^{q^{k+1}} \notag \\
               \eqdist + \;&\sum_{i=2}^{k+1} s^q \cdot \Big( c_{k-i+1}\toq[i] 
             + t_k^{\,q-1} \cdot c_{k-i}\toq[i] \Big) \;\cdot \;X^{q^{k+i}},
      \label{eqn:polynomial with Galois group B_k}
    \end{alignat}
    where $b_i$, $c_i$ are as above.
    In particular, $\Gal(f) \cong \SO_{2k+1}(q)$.
\end{enumerate}
\label{thm:specialorthogonalodd}
\end{thm}
\begin{proof}
a) As usual, the inclusion $\GalPhi(M) \leq \SO_{2k+1}(q)$ follows from Theorem~\ref{thm:upperbound} since by construction $D \in \SO_{2k+1}(\grdfield)$.
Equality is proved once again via the characteristic polynomial.
Let~$\genname$ be one of the two strong generators of~$\SO_{2k+1}(q)$ and let~$g(X)$ be its characteristic polynomial.
Equation~(\ref{eqn:charpoly_Bk}) above shows that the coefficients of~$h$ are antisymmetric.
The same holds for the coefficients of~$g$ by Lemma~\ref{lem:antisymmetric coefficients} (refer to Table~\ref{table:generators} in order to verify that the assumptions of the lemma are satisfied).
Hence it remains to be shown that the indeterminates $t_1,\ldots,t_k,s$ can be chosen in such a way that the coefficients~$(b_{i-1}+b_i)$ of~$h$ specialize to the coefficients of~$g$.
Let
\begin{equation}
  \label{eqn:odd orthogonal cross section characteristic polynomial revisited}
  g(X) = \sum_{i=0}^{k} a_i \cdot X^i \;-\; \sum_{i=0}^{k} a_i \cdot X^{(2k+1)-i}.
\end{equation}
We need to solve the following system of equations:
\begin{subequations}
\begin{align}
  -1 &= a_0, \qquad -1+t_1 = a_1, \qquad t_1+s\,t_2 = a_2, \label{eqn:odd orthogonal equation for a_0, a_1, a_2}\\
s\,(t_{i-1} + t_i) &= a_i \quad\qquad \textrm{for\quad $i\in\{3,\ldots,k-1\}$} ,
                                  \label{eqn:odd orthogonal equation for a_i} \\
s\,t_{k-1} + \frac{s}{2}\,t_k^{\;2} &= a_k. \label{eqn:odd orthogonal equation for a_k}
\end{align}
\end{subequations}
The first equation is satisfied automatically since $\SO_n(q) \leq \SL_n(q)$.
The next $k\minus 1$ ones can be solved inductively, yielding
$t_i = \frac{1}{s} \cdot \sum_{j=0}^i (-1)^{j}\,a_{i-j}$ for $i\in\{1,\ldots,k-1\}$.
The last equation is then of the form
\begin{equation}
  \frac{s}{2}\, t_k^2 = u \qquad \textrm{(for some $u \in \F_q$)}
  \label{eqn:odd orthogonal equation for t_k}
\end{equation}
and can clearly be solved by choosing $s \in \Fq^\times$ such that $\frac{2}{s}\,\alpha$ is a square.
By construction,~$\{t_1,\ldots,t_k,s\}$ is a set of solutions for the system~(\ref{eqn:odd orthogonal equation for a_0, a_1, a_2})-(\ref{eqn:odd orthogonal equation for a_k}).

\medskip\noindent
b) Again we solve the defining equation~$D\cdot\phiq(x) = x$, $x\in M$ for the solution space of~$(M,\Phi)$.
With $x=(x_1,\ldots,x_k,z,y_1,\ldots,y_k)\tr$ this equation is equivalent to the following system.
\begin{subequations}
\begin{align}
  -\sum_{i=1}^{k} b_i \cdot x_i^{\;q} \eqdist - \eqdist s\,t_k \cdot z^{q} 
                                     \eqdist - \eqdist s \cdot y_1^{\:q} \quad &= \quad x_1,
               \label{eqn:odd orthogonal frobenius equation for x_1} \\[-2ex]
  {x_{i-1}}^{q} \quad &= \quad x_i \qquad \textrm{for $i\in\{2,\ldots,k\}$},
	       \label{eqn:odd orthogonal frobenius equation for x_i} \\
  t_k \cdot x_k^{\;q} \eqdist - \eqdist z^q \quad &= \quad z,
	       \label{eqn:odd orthogonal frobenius equation for z} \\
  -c_{k-i} \cdot x_k^{\;q} \;\;+\;\;  y_{i+1}^{\,q} \quad &= \quad y_i \qquad \textrm{for $i\in\{1,\ldots,k\minus 1\}$},
               \label{eqn:odd orthogonal frobenius equation for y_i} \\
  -c_0 \cdot x_k^{\;q} \quad &= \quad y_k.
               \label{eqn:odd orthogonal frobenius equation for y_k}
\end{align}
\end{subequations}%
Again we try to eliminate all variables except~$X = x_1$.
The initial steps are very similar to the ones for~$\Sp_{2k}(q)$.
Recursive substitution using~(\ref{eqn:odd orthogonal frobenius equation for x_i}) leads to
\begin{equation}
\label{eqn:odd orthogonal frobenius equation for x_i, direct}
  x_i = x_1^{\;q^{i-1}} = X^{q^{i-1}} \qquad\textrm{for $i\in\{1,\ldots,k\}$},
\end{equation}
hence~\mbox{$x_k^{\;q}=X^{q^k}$}, and~(\ref{eqn:odd orthogonal frobenius equation for y_k}) turns into~$y_k = -c_0 \cdot X^{q^k}$.
This is used as the starting point for successive re-substitution in~(\ref{eqn:odd orthogonal frobenius equation for y_i}), with~$i$ taking the values $k\minus 1,k\minus 2,\ldots,1$.
In this way we derive the formula
\begin{equation}
\label{eqn:odd orthogonal frobenius equation for y_i, direct}
  y_{k-j} \eqdist = \eqdist -\sum_{i=0}^{j} c_{j-i}^{\,q^i} 
                    \cdot X^{q^{k+i}}
                    \qquad\textrm{for $j\in\{0,\ldots,k\minus1\}$}.
\end{equation}
In particular, setting~$j=k\minus1$ and shifting the summation indices up by~1, we get%
$\displaystyle y_1 = -\sum_{i=1}^{k} c_{k-i}^{\,q^{i-1}} \cdot X^{q^{k+i-1}}.$
Next, we plug all the expressions obtained so far for~$x_i$ and~$y_1$ into equations~(\ref{eqn:odd orthogonal frobenius equation for x_1}) and~(\ref{eqn:odd orthogonal frobenius equation for z}).
We incorporate the single occurrence of~$X = x_1$ on the right hand side of~(\ref{eqn:odd orthogonal frobenius equation for x_1}) into the sum on the left by setting~$b_0 = 1$.
Additionally, the positions of~$z$ and~$z^q$ are swapped in the second equation to isolate the term~$z^q$:
\begin{subequations}
\begin{align}
 -\sum_{i=0}^k b_i \cdot X^{q^i} 
               \eqdist - \eqdist s\,t_k \cdot z^q
               \eqdist + \eqdist \sum_{i=1}^{k} s \cdot c_{k-i}^{\,q^i} \cdot X^{q^{k+i}}
	       \quad &= \quad 0,
	       \label{eqn:odd orthogonal, only two left, first one} \\
  t_k \cdot X^{q^k} \; - \; z \quad &= \quad z^q.
               \label{eqn:odd orthogonal, only two left, second one}
\end{align}
\end{subequations}%
We are left with two equations involving only the two variables~$X$ and~$z$.
We now replace the occurrence of~$z^q$ in the first equation with the corresponding term suggested by the second:%
\renewcommand{\eqdist}{\;}
\begin{equation}
 -\sum_{i=0}^k b_i \cdot X^{q^i} 
               \eqdist - \eqdist s\,t_k^{\,2} \cdot X^{q^k}
	       \eqdist + \eqdist s\,t_k \cdot z
               \eqdist + \eqdist \sum_{i=1}^{k} s \cdot c_{k-i}^{\,q^i} \cdot X^{q^{k+i}}
	       \quad = \quad 0.
	       \label{eqn:odd orthogonal, only two left, first one, z^q replaced}
\end{equation}
Afterwards, we raise the resulting equation to the $q$-th power, shift all summation indices up by~1 and extract the last term from the first sum, getting
\renewcommand{\eqdist}{\;\;}
\begin{equation}
\begin{split}
 -\sum_{i=1}^{k} b_{i-1}\toq \cdot X^{q^{i}}
               \eqdist &- \eqdist b_k^{\;q} \cdot X^{q^{k+1}}
	       \eqdist  - \eqdist s^q\,t_k^{\,2q} \cdot X^{q^{k+1}} \\
               \eqdist &+ \eqdist s^q\,t_k^{\,q} \cdot z^q
               \eqdist  + \eqdist \sum_{i=2}^{k+1} s^{\,q} \cdot c_{k-i+1}^{\,q^{i}} \cdot X^{q^{k+i}}
	       \quad = \quad 0.
	       \label{eqn:odd orthogonal, summation equation one}
\end{split}
\end{equation}
\renewcommand{\eqdist}{\;\;}%
To this we add the~$(s\,t_k)^{\,q-1}$-fold multiple of equation~(\ref{eqn:odd orthogonal, only two left, first one}), which we reprint for the convenience of the reader (this time with the first term of the \emph{second} sum extracted, to make summation easier later on).
\begin{equation}
\begin{split}
 -\sum_{i=0}^k (s\,t_k)^{\,q-1} \, b_i \cdot X^{q^i} 
               \eqdist &- \eqdist s^q\,t_k^{\,q} \cdot z^q
	       \eqdist  + \eqdist s^q\,t_k^{\,q-1} \cdot (c_{k-1})^q \cdot X^{q^{k+1}} \\
               \eqdist &+ \eqdist \sum_{i=2}^{k} s^q\,t_k^{\,q-1} \cdot c_{k-i}\toq[i] \cdot X^{q^{k+i}}
	       \quad = \quad 0.
	       \label{eqn:odd orthogonal, summation equation two}
\end{split}
\end{equation}
Adding equations~(\ref{eqn:odd orthogonal, summation equation one}) and~(\ref{eqn:odd orthogonal, summation equation two}) eliminates the variable~$z$ and results in the additive polynomial $f(X)$ from the statement of the theorem (note that we set $c_{-1}=0$).
Now if~$X=x_1$ is fixed then the remaining coordinates~$x_i$,~$y_i$ and~$z$ are determined by equations~(\ref{eqn:odd orthogonal frobenius equation for x_i, direct}),~(\ref{eqn:odd orthogonal frobenius equation for y_i, direct}), and~(\ref{eqn:odd orthogonal, only two left, first one, z^q replaced}) (which is easily solved for~$z$).
The assertion thus follows from Lemma~\ref{lem:galoisgroup}.
\end{proof}

\noindent
We refrain from printing the polynomial~$f$ with the placeholders replaced by the original variables, which would be slightly ugly because would need to extract some terms from the sums.
Instead, we give a few examples for small values of~$k$ to get a feeling for the resulting polynomials (note that for the sake of a little more simplicity these are not normalized).

{
\small
\begin{multline*}
k=1: \quad f(X) \;\;=\;\;
     s^{q^2-q} \cdot X^{q^3}
 \;\;-\;\;  (s \cdot t_1)^{\,q-1} \cdot \big(\frac{s}{2} \cdot t_1^{\;q+1} -1\big) \cdot X^{q^2}\\
 \;\;+\;\;  \big(\frac{s}{2} \cdot t_1^{\;q+1} - 1 \big) \cdot X^q
 \;\;-\;\;  t_1^{\;\,q-1} \cdot X,
\end{multline*}
\vspace{-1.5em}
\begin{multline*}
k=2: \quad  f(X) \;\;=\;\;
     s^{q-q^3}   \cdot X^{q^5}
 \;\;+\;\; s^{q-q^2}   \cdot \Big(t_2^{\;\,q-1}-t_1^{\;q^2}\Big) \cdot X^{q^4}\\
 \;\;-\;\; \Big( \frac{s^q}{2}\cdot t_2^{\;2q} + t_1^{\;q}\,t_2^{\;\,q-1} \Big) \cdot X^{q^3}
 \;\;+\;\; \Big( t_1^{\;q} + \frac{s^q}{2}\cdot t_2^{\;q+1} \Big) \cdot X^{q^2}\\
 \;\;+\;\; \Big( t_1 \cdot \left(s\,t_2\right)^{q-1} - 1 \Big) \cdot X^{q}
 \;\;-\;\; (s\,t_2)^{\,q-1} \cdot X,
\end{multline*}
\vspace{-1.5em}
\begin{multline*}
k=3: \quad  f(X) \;\;=\;\;
     s^{q-q^4}   \cdot X^{q^7}
 \;\;+\;\; s^{q-q^3}   \cdot \Big(t_3^{\;\,q-1}-t_1^{\;q^3}\Big) \cdot X^{q^6}\\
 \;\;-\;\; s^{q-q^2}   \cdot \Big(\left(s t_2\right)^{q^2} + t_3^{\,q-1}t_1^{\;q^2} \Big) \cdot X^{q^5}
 \;\;-\;\; \Big( \left(s t_2\right)^{q}\,t_3^{\;\,q-1} + \frac{s^q}{2}\cdot t_3^{\;2q} \Big) \cdot X^{q^4}\\
 \;\;+\;\; s^q \Big( t_2^{\;q} + \frac{t_3^{\,q+1}}{2} \Big) \cdot X^{q^3}
 \;\;+\;\; \Big( t_1^{\;q} + s^q\,t_3^{\;\,q-1}\,t_2 \Big) \cdot X^{q^2}\\
 \;\;+\;\; \Big( (s\,t_3)^{\,q-1}\,t_1 - 1 \Big) \cdot X^{q}
 \;\;-\;\; (s \, t_3)^{\,q-1} \cdot X.
\end{multline*}

\noindent
We would like to remark that it is indeed essential to introduce the additional indeterminate~$s$ since there are cases where equation~(\ref{eqn:odd orthogonal equation for t_k}) cannot be solved for~$t_k$ if~$s=1$.
This can be seen as follows.
Let $\genname$ be a strong generator of~$\SO_{2k+1}(\Fq)$ of the first form as given in Table~\ref{table:generators}.
Given that its characteristic polynomial~$g$ has roots $\alpha^{\pm q^i}$ and $+1$, it can be written in the form
$\displaystyle g(X) = (X-1) \cdot
  \prod_{i=0}^{k-1} \big(X-\alpha^{q^i}\big) \cdot 
  \prod_{i=0}^{k-1} \big(X-\frac{1}{\alpha^{q^i}}\big).$
Therefore,
\begin{alignat}{3}
g(-1) &= (-2) \cdot \; (-1)^k \prod_{i=0}^{k-1} \Big(1+\alpha^{q^i}\Big)
	          \cdot \; (-1)^k \prod_{i=0}^{k-1} \bigg(\frac{\alpha^{q^i} + 1}{\alpha^{q^i}}\bigg) \notag \\
          &= (-2) \cdot \; \prod_{i=0}^{k-1} \Big(\alpha^{q^i} + 1\Big)^2
	          \cdot \; \bigg(\prod_{i=0}^{k-1} \alpha^{q^i} \bigg)^{-1}.
\end{alignat}
The first product is always a square in~$\Fq$ (note that it belongs to~$\Fq$ because the Frobenius endomorphism simply permutes its factors, leaving the whole product invariant).
The second product can be written as
$\big(\alpha^{q^{k-1}+\ldots+q+1}\big)^{-1} = \big(\alpha^{\frac{q^k-1}{q-1}}\big)^{-1} = \beta^{-1}$,
where~$\beta=\alpha^{\frac{q^{k-1}}{q-1}}$ is a generator of the multiplicative group~$\Fq^\times$ since $\alpha$ has exact order $q^k-1$.
In particular,~$\beta$ is never a square in~$\Fq$ (since~$q$ is odd), nor is~$\beta^{-1}$.
From~(\ref{eqn:odd orthogonal cross section characteristic polynomial revisited}) and~(\ref{eqn:odd orthogonal equation for t_k}) we now deduce $t_k^{\,2} = (-1)^k \cdot g(-1)$ and conclude that this equation is solvable if and only if $(-1)^k \cdot (-2)$ is a non-square in $\Fq$.
There are plenty of examples where this is false -- consider for instance~$q=11, k=2$.

\bigskip\noindent
} 

%% file: polynomial_for_Dk.tex
\subsection{The Even-Dimensional Special Orthogonal Groups $\SO^+_{2k}(q)$}
\label{sec:Dk}

Let $q$ be odd, $n=2k\geq 4$, $(q,n) \neq (3,4)$ and $J = \left(\begin{smallmatrix} & & 1 \\[-1ex] & \iddots & \\ 1 & & \end{smallmatrix}\right) \in \GL_n(q)$.
The untwisted special orthogonal group is defined as $\SO^+_{2k}(K) = \{ A \in \GL_{n}(K) \;|\; A\tr \cdot J \cdot A = J \}$.
The twisted variant~$\SO^-_{2k}$ is treated in section~\ref{subsect:2D_k}.
As before, we work over the ground field $\grdfield=\Fq(t_1,\ldots,t_k,s)$ which contains an additional indeterminate~$s$.
For $i \leq k-1$, the root subgroups and Weyl group representatives are the ``same'' as in odd dimension, namely $x_i = \idblock{\mathbf{k},\mathbf{k}}{\root[k]{i}{t_i},\root[k]{-i}{-t_i}}$ and $w_i = \idblock{\mathbf{k},\mathbf{k}}{\weylpos_{i,k},\weylpos_{-i,k}}$, where we alter the first representative in the same way as above.
For $i=k$, the matrices are
\begin{equation*}
  x_k = \idblock{(k-2),\mathbf{4},(k-2)}{
    \left(\begin{smallmatrix}
      1 & 0 & t_k &    0 \\
        & 1 &   0 & -t_k \\
        &   &   1 &    0 \\
        &   &     &    1
    \end{smallmatrix}\right)},
  \quad
  w_k = \idblock{(k-2),\mathbf{4},(k-2)}{
    \left(\begin{smallmatrix}
        &   & 1 &    \\
        &   &   & -1 \\
     -1 &   &   &    \\
        & 1 &   &  
    \end{smallmatrix}\right)}.
\end{equation*}
Analogously to the previous case of~$\SO_{2k+1}$, we compute
\begin{equation*}
  \pcs_{k-1} (t_1,\ldots,t_{k-1},s) =
  \idblock{\textbf{k},\textbf{k}}{
    \companion{t_1, s t_2, \ldots, s t_{k-1}, s},
    \companion{\frac{1}{s}, -\frac{t_1}{s}, -t_2, \ldots, -t_{k-1}}
  }.
\end{equation*}
This matrix is multiplied from the right by the last factor
\begin{equation*}
x_k \cdot w_k =
  \idblock{(k-2),\mathbf{4},(k-2)}{
    \left(\begin{smallmatrix}
     -t_k &      & 1 &    \\
          & -t_k &   & -1 \\
      -1  &      & 0 &    \\
          &    1 &   & 0
    \end{smallmatrix}\right)}.
\end{equation*}
The full cross section thus has the shape
\begin{equation}
\label{eqn:even orthogonal cross section}
  \pcs_k(t_1,\ldots,t_k,s)= \left(
    \begin{array}{c@{\;\;}c@{\;\;}c@{\;}c@{\;\;\;}cc|c@{\;\;}c@{\;\;}c@{\;\;}c@{\;\;\;}c@{\;\;\;}cc@{}}
      t_1 & s t_2 & \ldots & s t_{k\minus 2} & -s t_{k-1} t_k & -s t_k & s t_{k-1} & -s & & & & & \\
        1 &       &        &                 &                &        &           &    & & & & & \\
          &     1 &        &                 &                &        &           &    & & & & & \\
          &       & \ddots &                 &                &        &           &    & & & & & \\
          &       &        &              1  &                &        &           &    & & & & & \\[0.5em]
          &       &        &                 &           -t_k &     0  &         1 &    & & & & & \\
     \hline
          &       &        &                 &      t_{k-1}   &     1  &         0 &  0 &        &        &   &   \\
          &       &        &                 &      t_{k-2}   &        &           &  0 &    1   &        &   &   \\
          &       &        &                 &      \vdots    &        &           &    & \ddots & \ddots &   &   \\
          &       &        &                 &       t_2      &        &           &    &        &    0   & 1 &   \\[0.5em]
          &       &        &                 &  \frac{t_1}{s} &        &           &    &        &        & 0 & 1 \\[0.5em]
          &       &        &                 &   -\frac{1}{s} &        &           &    &        &        &   & 0 \\
    \end{array}
  \right).
\end{equation}
Its characteristic polynomial~$h(X)$ is the determinant of $B = (X \cdot \id_n \minus \pcs_k)$. Once more, $B$ has the general shape of~(\ref{eqn:almost bitriangular matrix}), but this time with \emph{two} inserted middle rows and columns (i.e., in the notation of~(\ref{eqn:almost bitriangular matrix}) we have~$r=k\minus1$, where~$\kxk{r}$ are the dimensions of the four corner blocks).
The correspondences between the variables are as follows.
\begin{alignat*}{5}
  b_1     &= -t_1,                           &   \qquad b_i &= -s t_i
          &  \qquad\text{for } &i\in\{2,\ldots,k\minus2\}, \\
  b_{k-1} &= s\,t_{k-1}t_k,   &   \qquad \beta   &= s, \\
  c_0 &= \frac{1}{s},              &   c_1 &= -\frac{t_1}{s},
      &   c_i &= -t_i \quad\text{for } i\in\{2,\ldots,k\minus2\}.
\end{alignat*}
Observe that~$b_i$ and~$c_j$ are only defined for~$1\leq i \leq k\minus1$ and~$0\leq j \leq k\minus2$, respectively;
we define $b_0=1$ (to be used later) and set $b_i=c_j=0$ for other values of $i,j$ (the only purpose of these placeholders is to simplify summations in subsequent computations).
By Corollary~\ref{cor:det of bitriangular matrix}, $h(X)=\det(B)$ equals the following \kxk{4}-determinant.
\begin{align*}
  h(X) &=
  \det \begin{pmatrix}
         g_{k-1} & s\, t_k & -s t_{k-1} & s \\
         t_k     &    X    &     -1     & 0 \\
        -t_{k-1} &   -1    &      X     & 0 \\
         h_{k-2} &    0    &      0     & X^{k-1}
  \end{pmatrix}\\
  &= \left( X^{k+1} \minus X^{k-1} \right) \cdot g_{k-1} 
  \;\;\minus\;\; s \cdot \left( t_k^{\;2} \plus t_{k-1}^{\;2} \right)\cdot X^k\\
  &\hspace*{10.5em}\;\;\plus \;\; 2 s t_k t_{k-1} X^{k-1} 
  \;\;\minus\;\; s \cdot \left( X^2 \minus 1 \right) \cdot h_{k-2}. %
\end{align*}
We resubstitute~$g_{k-1}$ and~$h_{k-2}$, expand products and merge corresponding sums after shifting appropriate indices (recall that we set $b_{k}=b_{k+1}=c_{k-1}=c_k=0$).
\begin{multline*}
 h(X) = \Big(\sum_{i=0}^{k+1} \left(b_i - b_{i-2}\right) X^{2k-i} \Big)
          \;\minus\; s \cdot \left( t_k^{\;2} \plus t_{k-1}^{\;2} \right)\cdot X^k \\
          \;\plus\; 2s t_k t_{k-1} X^{k-1}
          \;\plus\; s \cdot \sum_{i=0}^{k} \left(c_i - c_{i-2}\right) X^{i}. %
\end{multline*}
From both sums we now extract the terms involving the powers $X^k$ and $X^{k-1}$ and combine them with the corresponding ``free'' terms.
Taking into account the definitions of $b_i$ and the fact that~$b_i = s c_i$, we can further collect terms to obtain the final form of the characteristic polynomial:
\begin{equation*}
 h(X) = \sum_{i=0}^{k-1} \left(b_i - b_{i-2}\right) X^{2k-i}
       \;+\; s \cdot \left(2t_{k-2} - t_{k-1}^{\;2} - t_{k}^{\;2} \right) \cdot X^k
       \;+\; \sum_{i=0}^{k-1} \left(b_i - b_{i-2}\right) X^{i}.
       \label{eqn:characteristic polynomial of even-dimensional special orthogonal cross section}
\end{equation*}

\bigskip
\begin{thm}
Let $q$ be odd and let $(M,\Phi)$ be the n-dimensional Frobenius module over $\grdfield=\Fq(t_1,\ldots,t_k,s)$ such that the representing matrix of $\Phi$ equals $D=\pcs_k(t_1,\ldots,t_k,s)$ for some basis.
Then
\begin{enumerate}
  \item $\GalPhi(M) \cong \SO^+_{2k}(q)$.
  \item The solution field of $M$ is generated by the roots of the following additive polynomial $f(X)$;
    in particular, $\Gal(f) = \SO_{2k}^+(q)$.
\begin{multline*}
  \sum_{i=0}^{k-1} \Big( (\delta s)^{q-1} \, t_k^q \, b_i \;+\; \big(\delta^{q-1} \, t_{k-1} - s^{q^2-q} \, t_{k-1}^{q^2}\big) \, b_{i-1}^q \;-\; t_k^q \; b_{i-2}^{q^2} \Big) \cdot X^{q^i} \\
  + \big( (\delta^{q-1} \, t_{k-1} \;-\; s^{q^2-q} \, t_{k-1}^{q^2}) \, b_{k-1}^q \;-\; (b_{k-2}^{q^2} \;-\; 2 \delta^{q-1} \, s^q \, t_{k-1}^{q+1} \;-\;\delta^q \;-\; \delta^{q-1} \, s^q\, c_{k-2}^q) \, t_k^q \big) \cdot X^{q^k} \\
  + (2 s^{q^2} t_{k-1}^{q^2} \, t_k^{q^2+q} \;-\; t_k^q \, b_{k-1}^{q^2} \;+\; \big(s^{q^2} t_{k-1}^{q^2} \;-\; \delta^{q-1} \, s^q \, t_{k-1}\big) \, c_{k-2}^{q^2} \;-\; \delta^{q-1} \, s^q \, t_k^q \, c_{k-3}^{q^2}) \cdot X^{q^{k+1}} \\
  + \sum_{i=2}^{k} \Big(s^{q^2}\, t_k^q\, c_{k-i}^{q^{i+1}} \;+\; \big((s\, t_{k-1})^{q^2} - s^q \delta^{q-1} t_{k-1}\big) c_{k-1-i}^{q^{i+1}} \;-\; \delta^{q-1} \, s^q \, t_k^q \, c_{k-2-i}^{q^{i+1}} \Big) \cdot X^{q^{k+i}},
\end{multline*}
where we define $\delta = s^q\big(t_k^{\,q+1} - t_{k-1}^{\,q+1}\big)$.
\end{enumerate}
\end{thm}
\begin{proof}
a) The inclusion $\GalPhi(M) \leq \SO_{2k}^+(q)$ follows from Proposition~\ref{thm:upperbound}.
For the reverse inclusion, we first note that $h(X)$ has symmetric coefficients.
According to Lemma~\ref{lem:symmetric coefficients} and Table~\ref{table:generators}, the same holds for the characteristic polynomials of the strong generators of~$\SO_{2k}^+(q)$ which have the form
\begin{equation}
  \label{eqn:characteristic polynomial of torus element, reprint}
  X^{2k} + \sum_{i=1}^{k-1} a_i \cdot X^{2k-i} + a_k \cdot X^k + \sum_{i=1}^{k-1} a_i \cdot X^{i} + 1.
\end{equation}
We need to prove that $h$ can be specialized to polynomials of this shape.
Comparison of coefficients leads to the following system of equations
\begin{subequations}
\begin{alignat}{2}
  1 = a_0, \qquad &-t_1 = a_1, \qquad &&-1-s\, t_2 = a_2, \qquad t_1 - s\,t_3 = a_3,
                               \label{eqn:even orthogonal equation for a_1, a2, a3}\\[1ex]
s \cdot (t_{i-2} &- t_i) &&= \phantom{-}a_i \qquad \textrm{for\quad $i\in\{4,\ldots,k\minus2\}$},
                               \label{eqn:even orthogonal equation for a_i} \\
s \cdot (t_{k-3} &+ t_{k-1}\,t_k) &&= \phantom{-}a_{k-1},
                               \label{eqn:even orthogonal equation for a_k-1} \\
s \cdot (2\, t_{k-2} &- t_{k}^{\;2} - t_{k-1}^{\;2}) \phantom{-}&&= \phantom{-}a_k.
                               \label{eqn:even orthogonal equation for a_k}
\end{alignat}
\end{subequations}
These can be solved recursively to find
$\displaystyle t_i = -\frac{1}{s} \cdot \sum_{j=0}^{\lfloor \frac{i}{2} \rfloor} a_{i-2j}$
for \mbox{$2 \leq i \leq k\minus 2$}.
Then~(\ref{eqn:even orthogonal equation for a_k-1}) implies
$\displaystyle t_{k-1} = \frac{1}{s t_k} \cdot \sum_{j=0}^{\lfloor \frac{k}{2} \rfloor} a_{(k-1)-2j}$,
which in turn is plugged into~(\ref{eqn:even orthogonal equation for a_k}):
\begin{equation*}
  t_k^{\,2} = -\frac{1}{s}   \left( a_k + 2\,a_{k-2} + 2\, a_{k-4} + \ldots\;  \right)
             -\frac{1}{(s\, t_k)^2} \left( a_{k-1} + a_{k-3} + a_{k-5} + \ldots\; \right)^2.
\end{equation*}
Bringing all terms to one side yields a biquadratic equation in the
variable~$t_k$:
\begin{equation}
\begin{split}
  \label{eqn:even orthogonal biquadratic equation for t_k}
  h(t_k)  =  
  t_k^{\,4}  & +  \frac{1}{s} 
                                \left( a_k + 2\,a_{k-2} + 2\, a_{k-4} + \ldots\;  \right) \cdot t_k^{\,2}
               +  \frac{1}{s^2}
	                        \left( a_{k-1} + a_{k-3} + \ldots\; \right)^2 
	       =  0.
\end{split}
\end{equation}
Assume that~$s$ has already been specialized to an element of~$\Fq^\times$.
Then the above expression is a biquadratic polynomial in~$t_k$ with coefficients in~$\Fq$.
We abbreviate these coefficients by $u$ and $v$ so that the equation reads as follows.
\[
 h(t_k) = t_k^4 + \frac{u}{s} \cdot t_k^{2} + \left(\frac{v}{s}\right)^2 = 0.
\]
This is a quadratic equation in $t_k^2$ with discriminant $(u^2-4v^2)/s^2$ which we claim is a square in $\Fq$,
say, $u^2-4v^2 = c^2$ with $c \in \Fq$.
Then we have the solutions $t_k^2=\frac{-u \pm c}{2s}$, and the right hand side becomes a square after a suitable choice of $s \in \Fq^\times$ so that the equation is solvable for~$t_k$.

In order to prove that the discriminant is a square it suffices to show that~$(-u-2v)$ and~$(-u+2v)$ are either both squares or non-squares in~$\Fq$, or equivalently that their product is a square.
Let $\genname$ be one of the strong generators of~$\SO_{2k}^+(q)$ from Table~\ref{table:generators} and let $\charpoly{\genname}$ be its characteristic polynomial.
Denote the first $k$ roots of $\charpoly{\genname}$ by $\alpha_1, \ldots, \alpha_k$ so that the remaining $k$ roots are their inverses.
Then $\charpoly{\genname}$ can be written in two ways, as follows.
\begin{subequations}
\begin{align}
\charpoly{\genname}(x) \;&=\; x^{2k} \;\;+\;\; \sum_{i=1}^{k-1} a_i\, x^{2k-i} \;\;+\;\; a_k\, x^k \;\;+\;\; 
                             \sum_{i=1}^{k-1} a_i\, x^i \;\;+\;\;1,
			     \label{eqn:characteristic polynomial of t, first version} \\
\charpoly{\genname}(x) \;&=\; \textstyle \prod_{i=1}^{k} \; \big(x-\alpha_i\big) \cdot \big(x-\alpha_i^{\,-1}\big)\;.
			     \label{eqn:characteristic polynomial of t, second version}
\end{align}
\end{subequations}
Using~(\ref{eqn:characteristic polynomial of t, first version}) we obtain the following identities
\begin{alignat*}{2}
  -u + 2v \;&=\; -a_k + 2 \cdot \sum_{i=1}^{k} (-1)^{i+1} \cdot a_{k-i} &
                \;\;&=\;\; (-1)^{k+1} \cdot \charpoly{\genname}(-1), \\
  -u - 2v
          \;&=\; - a_k - 2\cdot\sum_{i=1}^{k} a_{k-i}&
        \;\;&=\;\; (-1) \cdot \charpoly{\genname}(1).
\end{alignat*}%
Using equation~(\ref{eqn:characteristic polynomial of t, second version}) we can further transform~$\charpoly{\genname}(1)$ and~$\charpoly{\genname}(-1)$ as follows:
\begin{equation*}
  \charpoly{\genname}(-1) \;=\; \prod_{i=1}^k \frac{\big( 1+\alpha_i \big)^2}{\alpha_i},
  \qquad
  \charpoly{\genname}(1)  \;=\; (-1)^k \cdot \prod_{i=1}^k \frac{\big( 1-\alpha_i \big)^2}{\alpha_i}.
\end{equation*}
Putting everything together yields
\begin{equation}
  (-u+2v)\cdot(-u-2v) \;=\; (-1)^{k+2} \cdot \charpoly{\genname}(1) \cdot \charpoly{\genname}(-1)
                      \;=\; \prod_{i=1}^k \frac{\big(1+\alpha_i\big)^2 \cdot \big(1-\alpha_i\big)^2}
			                                {(\alpha_i)^2}. \label{eqn:lalelu}
\end{equation}
Note that the square root of the right hand side belongs to $\Fq$ since the Frobenius endomorphism simply permutes the elements~$\alpha_i$.

\medskip\noindent
b) To derive a polynomial with Galois group $\SO_{2k}^+(q)$ we solve the defining equation~$D\cdot\phiq(x) = x$ for the solution space of~$(M,\Phi)$, where $D=\pcs_k$ is the cross section matrix from~(\ref{eqn:even orthogonal cross section}).
We write $x=(x_1,\ldots,x_k,y_1,\ldots,y_k)\tr$.
Then the equation is equivalent to the following system.
\begin{subequations}
\begin{align}
  -\sum_{i=1}^{k-1} b_i \cdot x_i^{\;q} 
                                    \;-\; s\,t_k \cdot x_k^{\:q}
                                      +   s\,t_{k-1} \cdot y_1^{\:q}
                                      -   s \cdot y_2^{\:q} \; &= \; x_1,
               \label{eqn:even orthogonal frobenius equation for x_1} \\[-2ex]
  x_{i-1}\toq \; &= \; x_i \quad\textrm{for}\;\;{i\in\{2,\ldots,k\minus1\}},
	       \label{eqn:even orthogonal frobenius equation for x_i} \\
  -t_k \cdot x_{k-1}\toq  +  y_1^{\;q} \; &= \; x_k,
	       \label{eqn:even orthogonal frobenius equation for x_k} \\[1em]
  t_{k-1} \cdot x_{k-1}\toq  +   x_{k}^{\;q} \; &= \; y_1,
               \label{eqn:even orthogonal frobenius equation for y_1} \\
  -c_{k-i} \cdot x_{k-1}\toq  +  y_{i+1}\toq
               \; &= \; y_i \quad\textrm{for}\;\;{i\in\{2,\ldots,k\minus 1\}},
               \label{eqn:even orthogonal frobenius equation for y_i} \\
  -c_0 \cdot x_{k-1}\toq \; &= \; y_k.
               \label{eqn:even orthogonal frobenius equation for y_k}
\end{align}%
\end{subequations}%
Let~$X \coloneqq x_1$. Equation~(\ref{eqn:even orthogonal frobenius equation for x_i}) implies that~$x_i=X^{q^{i-1}}$, for~$i\in\{1,\ldots,k\minus1\}$.
Hence,~(\ref{eqn:even orthogonal frobenius equation for y_k}) becomes~$-c_0 \cdot X^{q^{k-1}} = y_k$, which we resubstitute into~(\ref{eqn:even orthogonal frobenius equation for y_i}).
By inductive substitution, we obtain the following formula for the~$y_i$.
\begin{equation}
  \label{eqn:even orthogonal equation closed expression for y_i}
  y_{k-j} = -\sum_{i=0}^j c_{j-i}\toq[i] \cdot X^{q^{(k-1)+i}}, \qquad\textrm{for}\;\;{j\in\{0,\ldots,k\minus2\}}.
\end{equation}
In particular, the case~$j=k\minus2$ yields
$\displaystyle y_{2} = -\sum_{i=0}^{k-2} c_{k-2-i}^{\;q^i} \cdot X^{q^{(k-1)+i}}$.
Using these identities, we transform equations~(\ref{eqn:even orthogonal frobenius equation for x_1}),~(\ref{eqn:even orthogonal frobenius equation for x_k}),~(\ref{eqn:even orthogonal frobenius equation for y_1}) in the following way (note the comments below):
\renewcommand{\eqdist}{\;}
\begin{subequations}
\begin{equation}
  \underbrace{%
  -\sum_{i=0}^{k-1} b_i \cdot X^{q^i} 
          \eqdist + \eqdist \sum_{i=0}^{k-2} s \cdot c_{k-2-i}^{\;q^{i+1}} \cdot X^{q^{k+i}}%
	  }_{g(X)}
          \eqdist + \eqdist s\,t_{k-1} \cdot y_1^{\:q} 
          \eqdist - \eqdist s\, t_k \cdot x_k^{\:q} \quad = \quad 0,
          \label{eqn:even orthogonal linear system, equation 1}
\end{equation}
\vspace*{-1.5em}
\begin{alignat}{2}
  \qquad\qquad y_1^{\;q} &\eqdist=\eqdist x_k \eqdist+\eqdist &t_k \:\cdot\: &X^{q^{k-1}},
                    \label{eqn:even orthogonal blabla1} \\
  \qquad\qquad x_k^{\;q} &\eqdist=\eqdist y_1 \eqdist-\eqdist & t_{k-1}\:\cdot\: &X^{q^{k-1}}.
                    \label{eqn:even orthogonal blabla2}
\end{alignat}
\end{subequations}%
In the first equation we set~$b_0 = 1$, which accounts for the single term~$x_1$ originally found on the right hand side.
We also used the identity~$x_i = X^{q^{i-1}}$ from above.
We now abbreviate the two sums by~$g(X)$ as indicated, substitute the right hand sides of the other two equations for~$y_1\q$ and~$x_k\q$ and raise the whole equation to the $q$-th power.
\begin{equation}
  g(X)^q + (s t_{k-1})^q \; x_k^q - (s t_k)^{\,q} \; y_1^{\,q} + 2 (s t_{k-1} t_k)^q X^{q^k} = 0
  \label{eqn:even orthogonal linear system, equation 2}
\end{equation}
We have now arrived at a system~(\ref{eqn:even orthogonal linear system, equation 1}),~(\ref{eqn:even orthogonal linear system, equation 2}) of two linear equations in the terms~$y_1\q$ and~$x_k\q$.
It has the solution
\begin{equation}
\begin{split}
  y_1\q &= 1/\delta \cdot \big( s^{q-1}\, t_{k-1}^{\,q} \cdot g(X) + t_k \cdot g(X)^q + 2(s\,t_{k-1})^q\, t_k^{\,q+1} X^{q^k} \big)\\
  x_k\q &= 1/\delta \cdot \big( s^{q-1}\, t_k^{\,q} \cdot g(X) + t_{k-1} \cdot g(X)^q + 2(s\,t_k)^q\, t_{k-1}^{\,q+1} X^{q^k} \big)
  \label{eqn:solution of 2x2-system}
\end{split}
\end{equation}
where we set~$\delta = s^q \, \big( t_k^{\,q+1} - t_{k-1}^{\, q+1} \big)$ (which is the negative determinant of the system divided by~$s$).
Equipped with these solutions, we can eliminate~$y_1$ and~$x_k$ altogether by raising equation~(\ref{eqn:even orthogonal blabla1}) to the $q$-th power, substituting according to~(\ref{eqn:solution of 2x2-system}) and multiplying by~$\delta^q$:
\begin{multline}
t_k\q \cdot g(X)^{q^2} + \big( s^{q^2-q} \cdot t_{k-1}^{\,q^2} - \delta^{q-1} \cdot t_{k-1} \big) \cdot g(X)^q - (\delta s)^{q-1} t_k\q \cdot g(X) \\
   + 2 s^{q^2} t_{k-1}^{q^2} t_k^{q^2+q} \cdot X^{q^{k+1}}
   - \big( 2 \delta^{q-1} s^q \, t_{k-1}^{\;q+1} \, t_k^{\;q} + \delta^q t_k^{\;q} \big) X^{q^k} = 0
\end{multline}%
After resubstituting~$g(X)$ and collecting terms (including index-shifts, where special care must be taken with respect to the summands involving the powers~$X^{q^k}$ and~$X^{q^{k+1}}$) we end up with the polynomial $f(X)$ in the statement of the theorem.
By construction, every first coordinate $x_1$ of a solution to $D\cdot \phiq(x) = x$ is a root of $f$.
Vice versa, every choice of $X=x_1$ determines $x_2,\ldots,x_{k-1}$ and $y_2,\ldots,y_k$ through~(\ref{eqn:even orthogonal frobenius equation for x_i}), (\ref{eqn:even orthogonal equation closed expression for y_i});
$x_k$ and $y_1$ are then determined via~(\ref{eqn:even orthogonal blabla1}), (\ref{eqn:even orthogonal blabla2}) and~(\ref{eqn:even orthogonal blabla2}).
Lemma~\ref{lem:galoisgroup} asserts that $\Gal(f) = \GalPhi(M) = \SO_{2k}^+(q)$.
\end{proof}

\medskip
\noindent
Note that the proof that the biquadratic polynomial $h$~always possesses a nonzero root in~$\Fq$ (for a suitable choice of~$s$) only works in odd characteristic.
As a matter of fact,\vspace{0.5ex} if $q$ is even there exist biquadratic polynomials of the form~$h=t^4 + \frac{a}{s} \cdot t^2 + \big(\frac{b}{s}\big)^2 \;\; (a,b \in \Fq)$\vspace{0.5ex} such that~$h$ does not have a root in~$\Fq$ for \emph{any} choice of~$s\in\Fq^\times$.
For example, take~$q=8$ and~$h=t^4+\frac{1}{s} \cdot t^2 + (\frac{1}{s})^2$ (in fact, the same polynomial seems to be a counterexample when~$q$ is any odd power of~2).
However, we still have the possibility of choosing a different semisimple generator (i.e., different values for~$\alpha$ and~$\beta$ in Table~\ref{table:generators}), resulting in different coefficients of its characteristic polynomial~$\charpoly{\genname}$ and hence also in a different biquadratic polynomial~$h$.
One might hope that for an appropriate choice of $\alpha, \beta \in \F_{q^k}$, $h$~has a nonzero root in~$\Fq$.
We conjecture that this is in fact true for \emph{every} choice of~$\alpha$ and~$\beta$.
Experiments with~$\textsc{Magma}$~\cite{Magma} seem to support this.
If this is indeed the case, our method can also be used to realize the groups~$\SO_{2k}^+(q)$ in even characteristic.

\bigskip\noindent
Instead of resubstituting the placeholders by the original variables in the general polynomial, we give an example of~$f$ for the case~$k=2$ (not normalized).
\begin{multline*}
  f(X) = \frac{t_2^q}{s^{q^3-q^2-q}} \cdot X^{q^4}
     \;+\; \big(\frac{t_1^{q^2}}{s^{q^2-q}} - s^q t_1^{q^2} t_2^{q^2+q} + 2 s^q t_1^{q^2} t_2^{q^2+q} - \delta^{q-1}\frac{t_1}{s^{q^2-q}}\big) \cdot X^{q^3} \\
     \;-\; t_2^q \cdot \big(t_1^{q^2+q} + s^q + \delta^{q-1} s + \delta^{q-1} t_1^{q+1} + \delta^q\big) \cdot X^{q^2} \\
     \;+\; \big(\delta^{q-1} s^q t_1 t_2^{q+1} + \delta^{q-1} t_1 - t_1^{q^2}\big) \cdot X^q
     \;+\; \delta^{q-1} s^q t_2^q \cdot X.
\end{multline*}
Here,
$\delta \;=\; \big( (s t_2)^{\,q+1} - t_1^{\;q+1} \big)^{-1}$, and~$f$ has coefficients in the field~$\Fq(t_1, t_2, s)$.

%% file: polynomial_for_2Ak.tex
\subsection{The Special Unitary Groups $\SU_n(q)$}

\noindent In this section we realize the special unitary group $\SU_n(q)$ for $n\geq 3$ and $(n,q)\neq (3,2)$ over $K=\F_{q^2}(t_1,\dots,t_k)$, $k=\left\lfloor \frac{n}{2}\right\rfloor$. 
Let $J$ denote $\left(\begin{smallmatrix} & & 1 \\[-1ex] & \iddots & \\ 1 & & \end{smallmatrix}\right) \in \GL_n$ and define $F: \SL_{n} \rightarrow \SL_{n}$, $\ A\mapsto J \cdot \phi_{q}(A)^{tr,-1} \cdot J$.
Then $F$ is a Frobenius map with $F^2=\phi_{q^2}$ and the special unitary group is defined as $\SU_{n}(q)=\SL_{n}^F\leq \SL_{n}(\mathbb{F}_{q^2})$.
We choose $\root{i}{t_i}$ as the $i$-th root subgroup and $\weylneg_i = \weylneg_{i,k}$ as the corresponding reflection.
Then $F\big(\root{i}{t_i}\big)=\root{n-i}{t_i^{-q}}$ and $F(\weylneg_i)=-\weylneg_{n-i}$.
We would like to construct a Frobenius module whose representing matrix $D$ is the product of some of the elements $\root{i}{t_i}$ and~$\weylneg_i$.
Depending on the parity of $n$, however, $F$ might fix the simple root in the middle.
We therefore distinguish between odd and even~$n$.

\bigskip
\noindent
Let first $n=2k+1$.
We would like to construct a matrix $D=\zwmat \cdot F(\zwmat)$ such that $D$ lies in Steinberg's cross section $\prod_{i=1}^k X_i w_i$.
Therefore, we define
$$\zwmattilde = \prod \limits_{i=1}^{k} \root{i}{t_i} \weylneg_i = \idblock{\mathbf{(k+1)},k}{\companion{-t_1, \ldots, -t_k, 1}}.$$
To make subsequent computations easier, we make the following modification:
\[
  \zwmat=\idblock{k,\mathbf{(k+1)}}{\companionrev{(-1)^k,(-1)^k t_1, \ldots, (-1)^k t_k}} \in \SL_{2k+1}(\mathbb{F}_{q^2}(t_1,\dots,t_k)).
\]
In order to compute $F(\zwmat)=J \cdot \phi_q(\zwmat)^{tr,-1} \cdot J$, we first compute
\[
  \zwmat^{-1} = \idblock{k,\mathbf{(k+1)}}{\companiontwo{-t_1, \ldots, -t_k, (-1)^k}}.
\]
Thus
$\phi_q(\zwmat)^{tr,-1} = \idblock{k,\mathbf{(k+1)}}{\companionrevtwo{-t_1^q, \ldots, -t_k^q, (-1)^k}}$ and conjugating by $J$ yields $F(\zwmat) = \idblock{\mathbf{(k+1)},k}{\companion{-t_1^q, \ldots, -t_k^q, (-1)^k}}$.
Finally, we obtain
\begin{displaymath}
D = \zwmat \cdot F(\zwmat)
  = \left( \begin{array}{cccc|c|cccc}
	-t_1^q&\dots&-t_{k-1}^q&-t_k^q &(-1)^k&&&\\ 
	1&&&&&& \\
	&\ddots& &&&&\\ 
	&&1&0&&&    \\ \hline
	&&&(-1)^kt_{k}&0&1&&    \\ \hline
	&&&(-1)^kt_{k-1}&&0&1& \\
	&&&\vdots &&& \ddots &\ddots &\\
	&&&(-1)^kt_1&&&&0&1  \\
	&&&(-1)^k&&&&&0
\end{array} \right).
\end{displaymath}
The characteristic polynomial $h$ of $D$ is the determinant of $B=(X\cdot \id_{2k+1}-D)$.
Note that $B$ has precisely the form of (\ref{eqn:almost bitriangular matrix}), where $r=k$ and there is one extra row and column.
The entries are $b_i=t_i^q$, $c_0=(-1)^{k-1}$, $c_i=(-1)^{k-1}t_i$.
Applying Corollary~\ref{cor:det of bitriangular matrix} yields 
\[
h(X)=\det
\begin{pmatrix}
	h_k& (-1)^{k-1}&0 \\
	(-1)^{k-1}t_{k}&X&-1 \\
	h_{k-1}&0&X^k \\
\end{pmatrix}=h_k X^{k+1}-X^k t_k + (-1)^k h_{k-1}.
\]
Substituting $g_k$ and $h_{k-1}$, we conclude 
$$h(X)=X^{2k+1}+\sum \limits _{i=1}^{k}t_i^q X^{2k+1-i} - X^kt_k - \sum \limits_{i=1}^{k-1}t_i X^i -1.$$

\begin{thm}\label{galois_su_odd}
Let $n=2k\plus 1$ be odd, $(n,q) \neq (3,2)$, and let $(M,\Phi)$ be the $n$-dimensional Frobenius module over $K=(\mathbb{F}_{q^2}(t_1,\dots,t_k), \phi_{q^2})$ such that the representing matrix of $\Phi$ equals $D$ for some basis $B$ of $M$.
Then
\begin{enumerate}
  \item $\Gal^\Phi(M)\cong \SU_{2k+1}(q)$.
  \item The solution field of $M$ is generated by the roots of the additive polynomial
    $$f(X)=X^{(q^2)^n}+ \sum \limits_{i=1}^{k}t_i^{(q^2)^{k+1-i}}X^{(q^2)^{n-i}}- \sum \limits_{i=1}^{k}t_i^q X^{(q^2)^i} -X \in \Fq(t_1,\dots,t_k)[X].$$ 
    In particular, $\Gal_{K}(f)\cong \SU_{2k+1}(q).$
\end{enumerate}
\end{thm}
\begin{proof}
a) By construction of $D$ and Theorem~\ref{thm:twupperbound}, we have $\Gal^\Phi(M) \leq \SU_{2k+1}(q)$.
By Lemma~\ref{lem:symmetry SU}, the generators of $\SU_{2k+1}(q)$ given in Table~\ref{table:generators} have separable characteristic polynomials of the form
$\displaystyle X^n+ \sum \limits_{j=1}^{k}a_j^qX^{n-j}-\sum_{j=1}^k a_jX^j -1$
for suitable $a_j \in \mathbb{F}_{q^2}$.
Thus the characteristic polynomial of $D$ specializes to the characteristic polynomials of the generators under the specializations $\psi_i: \mathbb{F}_{q^2}[t_1,\dots,t_k] \to \mathbb{F}_{q^2}$, $t_j \mapsto a_j$.

\medskip\noindent
b) We solve the equation $D \cdot \phi_{q^2}(x)=x$, where $x=(x_1,\dots,x_{2k+1})$.
Letting $X=x_1$, we evaluate the second to $k$-th row successively:
$x_i = x_{i-1}^{q^2} = X^{(q^2)^{i-1}}, \ 2\leq i \leq k$.
Plugging this equation for $i=k$ in the last row yields $x_n=(-1)^k x_{k}^{q^2}=(-1)^k X^{(q^2)^{k}}$.
Now we can evaluate the $(2k+1)$-th through $(k+1)$-th row, consecutively:
\begin{eqnarray*}
	x_{n-1}&=&(-1)^k t_1 x_k^{q^2}+x_n^{q^2}=(-1)^kt_1 X^{(q^2)^k} + (-1)^k X^{(q^2)^{k+1}}, \\
	&\vdots&\\
	x_{n-i}&=&(-1)^k t_i x_k^{(q^2)}+x_{n-i+1}^{(q^2)}\\
	&=&(-1)^k \big(t_i X^{(q^2)^k}+t_{i-1}^{q^2} X^{(q^2)^{k+1}}+\dots+t_1^{(q^2)^{i-1}}X^{(q^2)^{k+i-1}}+X^{(q^2)^{k+i}}\big),\\
	&\vdots&\\ 
	x_{k+1}	&=&(-1)^k \big(t_k X^{(q^2)^k}+t_{k-1}^{q^2}X^{(q^2)^{k+1}}+\dots+t_1^{(q^2)^{k-1}}X^{(q^2)^{2k-1}}+X^{(q^2)^{2k}}\big).
\end{eqnarray*}
This proves that the coordinates $x_2,\dots,x_n$ are uniquely and polynomially determined by $X=x_1$.
Plugging all these equations into the first row yields the polynomial $f(X)$ from the statement of the theorem.
The claim follows from Lemma~\ref{lem:galoisgroup}.
\end{proof}

\noindent
To remark the special $q$-palindromic shape of the obtained polynomials, we give some examples for small $n$.
\begin{align*}
  n=3: \quad &X^{(q^2)^3}+t^{q^2}X^{(q^2)^2}-t^qX^{q^2}-X \qquad\in \F_{q^2}(t),\\
  n=5: \quad &X^{(q^2)^5}+t^{(q^2)^2}X^{(q^2)^4}+s^{q^2}X^{(q^2)^3}-s^qX^{(q^2)^2}-t^qX^{q^2}-X \quad\in \F_{q^2}(t,s),\\
  n=7: \quad &X^{(q^2)^7}+t^{(q^2)^3}X^{(q^2)^6}+s^{(q^2)^2}X^{(q^2)^5}+
	r^{q^2}X^{(q^2)^4}-r^qX^{(q^2)^3}\\
             &\hspace{12em} -s^qX^{(q^2)^2}-t^qX^{q^2}-X \quad\in \F_{q^2}(t,s,r).
\end{align*}

\bigskip
\noindent
Let now $n=2k$.
Motivated by the case $n=2k \plus 1$, we define
\[
 \zwmat=\idblock{(k-1),\mathbf{(k+1)}}{%
    \companionrev{(-1)^k,(-1)^k t_1, \ldots, (-1)^k t_k}
     } \in \SL_{2k}(\mathbb{F}_{q^2}(t_1,\dots,t_k)).
\]
As above, we get $F(\zwmat)=\idblock{\mathbf{(k+1)},(k-1)}{\companion{-t_1^q, \ldots, -t_k^q, (-1)^k}}$.
We define $D=\zwmat \cdot F(\zwmat)$ and compute:
\[
D=
\left( \begin{array}{ccccc|ccccc} 
	-t_1^q&\dots&-t_{k-2}^q&-t_{k-1}^q&-t_k^q &(-1)^k&&\\ 
	1&&&&&& \\
	&\ddots&&&&&\\ 
	&&1&&&&&& \\
	&&&(-1)^kt_{k}&1&&&& \\ \hline
	&&&(-1)^kt_{k-1}&0&0&1&&& \\
	&&&(-1)^kt_{k-2}&&&0&1&& \\
	&&&\vdots &&&&\ddots& \ddots \\
	&&&(-1)^kt_1&&&&&&1  \\
	&&&(-1)^k&&&&&&0
\end{array} \right).
\]
The characteristic polynomial $h$ of $D$ is the determinant of $B=(X\cdot \id_{2k+1}-D)$.
Note that $B$ has precisely the form of (\ref{eqn:almost bitriangular matrix}), where $r=k-1$ and there are two extra rows and columns.
The entries are $b_i=t_i^q$, $c_0=(-1)^{k-1}$, $c_i=(-1)^{k-1}t_i$.
Applying Corollary~\ref{cor:det of bitriangular matrix} yields
\[
h(X)= \det
\begin{pmatrix}
	h_{k-1}& t_k^q&(-1)^{k-1}&0 \\
	(-1)^{k-1}t_{k}&X-1&0&0 \\
	(-1)^{k-1}t_{k-1}&0&X&-1 \\
	h_{k-2}&0&0&X^{k-1} 
\end{pmatrix}.
\]
Computing this determinant and substituting~$g_{k-1}$ and~$h_{k-2}$ leads to the following polynomial (note that in the last equation we set $t_0\coloneqq 1$).
\begin{align}
h(X) &=  h_{k-2}\cdot(X-1)(-1)^k + X^{k-1}\cdot\big(g_{k-1}\cdot X(X-1)-t_{k-1}(X-1)+(-1)^k X\, t_k^{q+1}\big) \nonumber \\
     &= X^n \;+\; \sum_{i=1}^{k-1} (t_i^q-t_{i-1}^q)\cdot X^{n-i}
            \;+\; \big(-t_{k-1}-t_{k-1}^q+(-1)^k \cdot t_{k+1}^{q+1}\big) \cdot X^k \nonumber \\
     &\qquad\;\;\;\;+\; \sum \limits _{i=1}^{k-1} (t_i-t_{i-1})\cdot X^{i} \;\;+\; 1 .
\label{eqn:cross section characteristic polynomial for SU_2k}
\end{align}

\begin{thm}
Let $n=2k$ be even and let $(M,\Phi)$ be the $n$-dimensional Frobenius module over $K=(\F_{q^2}(t_1,\dots,t_k), \phi_{q^2})$ such that the representing matrix of $\Phi$ equals $D$ for some basis $B$ of $M$.
Then
\begin{enumerate}
  \item $\Gal^\Phi(M) \cong \SU_{2k}(q)$.
  \item The solution field of $M$ is generated by the roots of the additive polynomial
    \begin{eqnarray*}
      f(X) &=& X^{(q^2)^n}+\sum\limits_{i=1}^{k-1}(t_i^{(q^2)^{k+1-i}}-t_k^{q^3-q}t_{i-1}^{(q^2)^{k+1-i}})X^{(q^2)^{n-i}}\\
           && +(-t_{k-1}^{q^2}t_k^{q^3-q}-t_{k-1}^{q^3} +(-1)^k t_k^{q^2(q+1)})X^{(q^2)^k} \\
           && + \sum\limits_{i=1}^{k-1}(-t_{i-1}^{(q^3)}+t_k^{q^3-q}t_{i}^{q})X^{(q^2)^{i}}+t_k^{q^3-q}X \in \Fq(t_1,\dots,t_k)[X],
    \end{eqnarray*}
    where $t_0=1$.
    In particular, $\Gal_{K}(f)\cong \SU_{2k}(q).$
\end{enumerate}
\end{thm}
\begin{proof}
a) By Theorem~\ref{thm:twupperbound}, we have $\Gal^\Phi(M) \leq \SU_n(q)$.
Using the lower bound criterion Theorem~\ref{thm:lowerbound}, it is sufficient to prove that there are specializations $\psi_i: \mathbb{F}_{q^2}[t_1,\dots,t_k] \to \mathbb{F}_{q^2}$, $t_j \mapsto u_{j}, \ i=1,2,$ such that the specialized polynomial $\psi_i(h)$ equals the characteristic polynomial of the $i$-th generator given in Table~\ref{table:generators}.
By Lemma~\ref{lem:symmetry SU}, these characteristic polynomials have the form
$\displaystyle X^n + \sum_{j=1}^{k-1}a_j^q X^{n-j} + a_k X^k + \sum_{j=1}^{k-1} a_jX^j +1$,
where $a_1,\dots a_{k-1} \in \mathbb{F}_{q^2}$, $a_k \in \Fq$.
From~(\ref{eqn:cross section characteristic polynomial for SU_2k}) we infer by successive substitution that a specialization $\psi_i$ would have to satisfy:
\begin{align*}
u_1 = a_1+1, \qquad u_2 &= a_2+a_1+1, \ldots, \qquad u_{k-1} = a_{k-1}+\dots +a_1+1, \\[1ex]
  u_k^{q+1} &= (-1)^k(a_k + u_{k-1} + u_{k-1}^q).
\end{align*}
We thus define $u_i=a_i+\dots+a_1+1 \in \mathbb{F}_{q^2}$ for $1 \leq i \leq k-1$. Then $(-1)^k(a_k + u_{k-1} + u_{k-1}^q)$ is contained in $\Fq$ and therefore has a $(q+1)$-th root $u_k$ in $\mathbb{F}_{q^2}$. This proves the existence of $\psi_i, i=1,2.$ 

\medskip\noindent
b) We need to solve the equation $D \cdot \phi_{q^2}(x)=x$, where $x=(x_1,\dots,x_{2k})$. 
\noindent This leads to $2k$ equations in $x_1,\dots,x_{2k}$.
We let $X=x_1$ and start by solving the second to $(k\minus 1)$-st equation, successively:
\begin{equation}\label{2bisk}
x_i=x_{i-1}^{q^2}=X^{(q^2)^{i-1}}, \ 2\leq i \leq k-1.
\end{equation}
Plugging this for $i=k-1$ into the last equation yields $x_n=(-1)^k x_{k-1}^{q^2}=(-1)^k X^{(q^2)^{k-1}}$.
Now we can consecutively solve the $(2k)$-th through $(k\plus 1)$-th equation:
\begin{eqnarray*}
x_{n-1}&=&(-1)^k t_1 x_{k-1}^{q^2}+x_n^{q^2} \;=\; (-1)^kt_1 X^{(q^2)^{k-1}} + (-1)^k X^{(q^2)^{k}}, \\
&\vdots& \\
x_{n-i}&=&(-1)^k t_i x_k^{(q^2)}+x_{n-i+1}^{(q^2)} \\
&=&(-1)^k \big(t_i X^{(q^2)^{k-1}}+t_{i-1}^{q^2} X^{(q^2)^{k}}+\dots+t_1^{(q^2)^{i-1}}X^{(q^2)^{k+i-2}}+X^{(q^2)^{k+i-1}}\big),\\
&\vdots&  \\ 
x_{k+1}&=&(-1)^k \big(t_{k-1} X^{(q^2)^{k-1}}+t_{k-2}^{q^2}X^{(q^2)^{k}}+\dots+t_1^{(q^2)^{k-2}}X^{(q^2)^{2k-3}}+X^{(q^2)^{2k-2}}\big)
\label{letztesuek}. 
\end{eqnarray*}

\noindent From these equations, we see that $x_2,\dots,x_{k-1},x_{k+1},\dots x_{2k}$ are uniquely determined by $x_1$.
For the sake of convenience, we do not yet replace them in the remaining two equations involving $x_1$ and $x_k$:
\begin{subequations}
\begin{eqnarray}
x_1&=&-\sum \limits_{i=1}^{k-1} t_i^q x_i^{q^2} + (-1)^k x_{k+1}^{q^2} -t_k^q x_k^{q^2}  \label{zeile1su}\\
x_{k}&=&x_k^{q^2}+(-1)^k t_k x_{k-1}^{q^2}. \label{zeileksu}
\end{eqnarray}
\end{subequations}
Our aim is to eliminate $x_k$. We insert (\ref{zeileksu}) into (\ref{zeile1su}) and raise the resulting equation to its $q^2$-th power:
\begin{equation}\label{zeile1su.2}
0=-x_1^{q^2}-\sum \limits_{i=1}^{k-1} t_i^{q^3} x_i^{q^4} + (-1)^k x_{k+1}^{q^4} -t_k^{q^3} x_k^{q^2}+(-1)^k t_k^{q^3+q^2} x_{k-1}^{q^4}.
\end{equation}
Finally, we multiply~(\ref{zeile1su}) by $(-t_k^{q^3-q})$ and add it to (\ref{zeile1su.2}):
\begin{multline*}
 t_k^{q^3-q}x_1-x_1^{q^2}+\sum \limits_{i=1}^{k-1} t_i^q t_k^{q^3-q} x_i^{q^2}-\sum \limits_{i=1}^{k-1} t_i^{q^3} x_i^{q^4}+(-1)^k t_k^{q^3+q^2} x_{k-1}^{q^4} \\
 - (-1)^kt_k^{q^3-q} x_{k+1}^{q^2} + (-1)^k x_{k+1}^{q^4}=0.
\end{multline*}
We substitute $x_2,\dots,x_{k-1},x_{k+1}$ and apply Lemma~\ref{lem:galoisgroup} to obtain the polynomial $f(X)$ from the statement of the theorem.
\end{proof}

\noindent
We remark that using similar methods as in Section~\ref{subsect:2D_k} we can also construct a $2n$-dimensional Frobenius module with Galois group $\SU_n(q)$ over $\mathbb{F}_q(t_1,\dots,t_{2k+1})$ instead of $\mathbb{F}_{q^2}(t_1,\dots,t_{k})$, where $n$ is either $2k\plus1$ or $2k$; see~\cite{Maier} for details.
So far, we haven't been able to deduce series of additive polynomials for this extension.

%% file: polynomial_for_2Dk.tex
\subsection{The Twisted Orthogonal Groups $\SO_{2k}^-(q)$}
\label{subsect:2D_k}

\noindent In this section we realize $\SO_{2k}^-(q)$ as a Galois group over $\Fq(t_1,\dots,t_k, s)$ for odd~$q$ by using the linear algebraic group $\mathcal{H} \leq \GL_{4k}$	whose $\Fq$-rational points are $\SO_{2k}^{-}(q)$.

\subsubsection*{Embedding $\SO_n^{-}(q)$ into $\GL_{2n}$}

Throughout this section, let $n=2k \geq 8$ be even and let $q$ be a power of an odd prime. Let further $\G=\SO_n^{+}$ be the special orthogonal group over $\overline{\Fq(t_1,\dots,t_k,s)}$ with respect to 
$J = \left(\begin{smallmatrix} & & 1 \\[-1ex] & \iddots & \\ 1 & & \end{smallmatrix}\right) \in \GL_n$. Define $F \colon \G \to \G$, $A \mapsto N^{-1}\phi_q(A)N$, where 
$N=\idblock{(k-1),\mathbf{2},(k-1)}{\big(\begin{smallmatrix} 0&1\\ 1&0 \end{smallmatrix}\big)}$. 
Then $F^2=\phi_{q^2}$, hence $F$ is a Frobenius map. The corresponding finite group of Lie type $\G^F\leq \G(\mathbb{F}_{q^2})$ is the group $\SO_n^{-}(q)$.

Since it worked quite well in case $\G^F=\SU_n(q)$, one might hope to realize $\G^F$ over $\mathbb{F}_{q^2}(\underline t)$ using an $n$-dimensional Frobenius module with representing matrix of the form $\zwmat\cdot F(\zwmat)$, $\zwmat \in \G(\mathbb{F}_{q^2}(\underline t))$. Therefore, we have a look at the action of $F$ on the root subgroups.  The Frobenius map $F$ is the composition of $\phi_q$ with the map which permutes $k$-th and $(k+1)$-th rows and columns of a matrix $A \in \G$. Therefore, $F$ fixes the root subgroups $X_1,\dots X_{k-2}$ and permutes $X_{k-1}$ and $X_k$. Following the approach used for the special unitary groups, we would therefore define $\zwmat=E_1(t_1)w_1^{-}\dots E_{k-1}(t_{k-1})w_{k-1}^{-}$.
We then compute $\zwmat\cdot F(\zwmat)=E_1(t_1)w^{-}_1 \cdot\ldots\cdot  E_{k-1}(t_{k-1})w^{-}_{k-1} \cdot E_1(t_1^q)w^{-}_1\cdot \ldots\cdot E_{k-2}(t_{k-2}^q) w^{-}_{k-2} \cdot E_{k}(t_{k-1}^q)w^{-}_k$.
The terms $X_iw_i$ for $1 \leq i \leq k-2$ occur twice in the product $\zwmat\cdot F(\zwmat)$ which gives the matrix a rather complicated shape. Therefore, we work with another approach using a linear algebraic group $\mathcal H \leq \GL_{2n}$ such that $\mathcal H(\Fq) \cong \SO_n^{-}(q)$.

\bigskip\noindent
Since $q$ is odd, we can fix an element $x$ that generates the extension $\mathbb{F}_{q^2}/ \Fq$ and satisfies $x^q=-x$. Then $\tau: \mathbb{F}_{q^2}(\underline t) \to \mathbb{F}_{q^2}(\underline t)$, $\ a(\underline t)+xb(\underline t) \mapsto a(\underline t)-xb(\underline t)$ extends $\phi_q: \mathbb{F}_{q^2} \to \mathbb{F}_{q^2}.$ We can now define a group extending $\SO_n^{-}(q)$:
$$\SO_n^-(\mathbb{F}_{q^2}(\underline t))=\left\{ A \in \G(\mathbb{F}_{q^2}(\underline t)) \ | \ N^{-1} \cdot \tau(A) \cdot N = A \right\}.$$
Embed $\GL_n(\mathbb{F}_{q^2}(\underline t))$ into $\GL_{2n}(\Fq(\underline t))$ by extending the natural map
$$\phi: \mathbb{F}_{q^2}(\underline t) \to M_{2\times2}(\mathbb{F}_{q}(\underline t)),\ a(\underline{t})+xb(\underline t) \mapsto 
\left(\! \! \! \! \begin{array}{cc} 
	a(\underline t) & x^2b(\underline t)\\ 
	b(\underline t) & a(\underline t) \\
\end{array} \! \! \! \! \right).$$
The image $\tilde{\mathcal{H}}$ is a linear algebraic group over $\mathbb{F}_{q}$. The action of $\tau$ on $\mathbb{F}_{q^2}(\underline t)$ translates to
$$\tilde \tau: \tilde{\mathcal{H}}(\mathbb{F}_{q}(\underline t)) \to \tilde{\mathcal{H}}(\mathbb{F}_{q}(\underline t)), \
\left(\! \! \! \! \begin{array}{cc} 
	a_{ij}(\underline t) & x^2b_{ij}(\underline t)\\
	b_{ij}(\underline t) & a_{ij}(\underline t) \\
\end{array} \! \! \! \! \right)_{i,j \leq n} \mapsto
\left(\! \! \! \! \begin{array}{cc}
	a_{ij}(\underline t) &- x^2b_{ij}(\underline t)\\
	-b_{ij}(\underline t) & a_{ij}(\underline t) \\
\end{array} \! \! \! \! \right)_{i,j \leq n}.$$

\noindent Let $A \in \GL_n(\mathbb{F}_{q^2}(\underline t))$. Then $N^{-1} \cdot \tau(A) \cdot N=A$ is equivalent to 
$\phi(N)^{-1}\cdot \tilde \tau (\phi(A)) \cdot \phi(N)=(\phi(A)).$ Since the determinant is given polynomially, there are polynomials $f_1, f_2 \in \Fq[X_{ij}^{11}, X_{ij}^{21}]$ such that 
$\det((X_{ij}^{11}+ x\cdot X_{ij}^{21})_{ij})=f_1(X_{ij}^{11}, X_{ij}^{21}) + x\cdot f_2 (X_{ij}^{11}, X_{ij}^{21}).$ 
We have now collected all necessary conditions to define $\mathcal H$:
\small $$\mathcal H = \left\{B \in \tilde{\mathcal H}(\overline{\Fq(\underline t)}) \ | \ f_1(B)=1, f_2(B)=0, \ \tilde B\phi(J)B=\phi(J), \ \phi(N)^{-1} \tilde \tau (B)  \phi(N)=B \right\},$$ 
 
 \normalsize
\noindent where $\tilde B$ denotes the matrix obtained by transposing its $(2 \times 2)$-blocks. It is easy to see that $\mathcal{H}$ is a linear algebraic group defined over $\F_q$ and by construction, $\SO_n^-(q) \cong \mathcal{H}(\mathbb{F}_{q}) , \ \SO_n^-(\mathbb{F}_{q^2}(\underline t)) \cong \mathcal{H}(\mathbb{F}_{q}(\underline t)),$ via $\phi$. Moreover, there is a natural isomorphism
$\mathcal{H} \cong \left\{(A_1,A_2) \in \SO_n^{+} \times \SO_n^{+} \ | \ A_1=N^{-1}A_2N  \right\}\cong \SO_n^{+},$ over $\overline{\Fq(\underline t)}$ hence $\mathcal{H}$ is connected and we can apply Theorem~\ref{thm:upperbound}.

\subsubsection*{A Frobenius module with Galois group $\SO_n^-(q)$}
We would like to construct a $2n$-dimensional Frobenius module over $(\Fq(\underline t),\phi_q)$ with representing matrix $\tilde D=\phi(D)$, for a $D \in \SO_n^{-}(\mathbb{F}_{q^2}(\underline t))$ such that $\Gal^\Phi(M)\cong \mathcal H(\Fq)$.
Since $\SO_n^{+}(\Fq(\underline t)) \subseteq \SO_n^{-}(\mathbb{F}_{q^2}(\underline t))$, we start from the generalized Steinberg cross section $\pcs_k(t_1,\dots,t_k,s)$ as defined in section~\ref{sec:Dk}.  

Conjugating a matrix by $N$ permutes the two middle rows and columns. Therefore, $\SO_n^{-}(\mathbb{F}_{q^2}(\underline t))=\left\{ A \in \G(\mathbb{F}_{q^2}(\underline t)) \ | \ N^{-1} \cdot \tau(A) \cdot N = A \right\}$ consists of all special orthogonal matrices of the following form. 
$$\footnotesize \left(\! \!  \begin{array}{ccc|ccc} 
A_1&&a_1&\tau(a_1)&&A_2 \\
&&\vdots&\vdots& \\
&&a_{k-1}&\tau(a_{k-1})&& \\
b_1&\dots&b_{k}&c_{k}&\dots&c_1 \\ \hline
\tau(b_1)&\dots&\tau(c_{k})&\tau(b_{k})&\dots&\tau(c_1) \\
&&d_1&\tau(d_1)&&\\
&&\vdots&\vdots& \\
A_3&&d_{k-1}&\tau(d_{k-1})&&A_4 \\
\end{array} \! \!   \right),$$ \normalsize where $A_1,A_2,A_3,A_4$ denote $(n-1) \times (n-1)$-matrices with entries in $\Fq(\underline t)$ and $a_i,b_i,c_i,d_i$ denote elements in $\mathbb{F}_{q^2}(\underline t)$.
Therefore, the natural way to define $D$ is $D=\pcs_k(t_1,\dots,t_{k-2},t_{k-1}+xt_{k}, -t_{k-1}+xt_k, s)$, so $D$ is a special orthogonal matrix and the elements  are chosen in a way such that $D \in \SO_n^{-}(\mathbb{F}_{q^2}(\underline t))$, thus $\phi(D) \in \mathcal H(\Fq(\underline t))$.
Note that $\underline t$ is now understood to be $(t_1,\dots,t_k,s)$.

\begin{thm}
Let $q$ be odd and let $(M,\Phi)$ be the $2n$-dimensional Frobenius module over $K=\big(\Fq(\underline t), \phi_q\big) = \big(\Fq(t_1,\dots,t_k,s), \phi_q\big)$ such that the representing matrix of $\Phi$ is $\vphantom{\Big(}\tilde D=\phi(D)$ with respect to some basis of $M$.
Then
\begin{enumerate}
  \item $\Gal^\Phi(M)\cong \mathcal{H}(\Fq) \cong \SO_n^-(q)$.
  \item Define an additive polynomial $f \in K[X]$ as follows, where we abbreviate $\alpha=(t_{k-1}^{q-1}+t_k^{q-1})^{q-1}$.
    If $k\geq 5$,
    \begin{eqnarray*}
      f(X)&=&s^{q^2-1}(t_{k-1}t_k)^{q-1}\alpha X+(-s^{q^2-1}(t_{k-1}t_k)^{q-1}\alpha t_1+s^{q^2-q}t_{k-1}^{q-1}\alpha - s^{q^2-q}t_k^{q^2-q})X^q \\
      &+&(-s^{q^2}(t_{k-1}t_k)^{q-1}\alpha t_2-s^{q^2-q}t_{k-1}^{q-1} \alpha t_1^q +s^{q^2-q}t_k^{q^2-q}t_1^q-1)X^{q^2} \\
      &+&(-s^{q^2}(t_{k-1}t_k)^{q-1}\alpha t_3-s^{q^2}t_{k-1}^{q-1} \alpha t_{2}^q +s^{q^2}t_k^{q^2-q}t_{2}^q+t_{1}^{q^2})X^{q^3} \\
      &+&\sum\limits_{i=4}^{k-2}(-s^{q^2}(t_{k-1}t_k)^{q-1}\alpha t_i-s^{q^2}t_{k-1}^{q-1} \alpha t_{i-1}^q +s^{q^2}t_k^{q^2-q}t_{i-1}^q+s^{q^2}t_{i-2}^{q^2})X^{q^i} \\
      &+&(-s^{q^2}(t_{k-1}^2-x^2t_k^2)(t_{k-1}t_k)^{q-1}\alpha -s^{q^2}t_{k-1}^{q-1}\alpha t_{k-2}^q+s^{q^2}t_k^{q^2-q}t_{k-2}^q+s^{q^2}t_{k-3}^{q^2})X^{q^{k-1}} \\
      &+&(s^{q^2}(t_{k-1}^2-x^2t_k^2)^qt_{k-1}^{q-1}\alpha +s^{q^2}(-t_{k-1}^2-x^2t_k^2)^qt_k^{q^2-q}+s^{q^2}t_{k-2}^{q^2}-2s^{q^2}t_{k-1}^{q^2+q} \\
      &&+s^{q^2}(t_{k-1}t_k)^{q-1}\alpha t_{k-2}^q) X^{q^{k}} \\
      &+&((t_{k-1}t_k)^{q-1}\alpha t_{k-3}^{q^2} 
      -(t_{k-1}^2-x^2t_k^2)^{q^2}+t_{k-2}^{q^2}(t_{k-1}^{q-1}\alpha -t_k^{q^2-q}))s^{q^2}X^{q^{k+1}} \\
      &+&\sum\limits_{i=2}^{k-4}((t_{k-1}t_k)^{q-1}\alpha t_i^{q^{k-1-i}}+(t_{k-1}^{q-1}\alpha -t_k^{q^2-q})t_{i+1}^{q^{k-1-i}}-t_{i+2}^{q^{k-1-i}})s^{q^2}X^{q^{2k-2-i}} \\
      &+&(\frac{1}{s^{q^{k-2}}}t_1^{q^{k-2}}(t_kt_{k-1})^{q-1}\alpha+(t_{k-1}^{q-1}\alpha -t_k^{q^2-q})t_2^{q^{k-2}}-t_3^{q^{k-2}})s^{q^2}X^{q^{2k-3}} \\
      &+&(-\frac{1}{s^{q^{k-1}}}(t_{k-1}t_k)^{q-1}\alpha+\frac{1}{s^{q^{k-1}}}t_1^{q^{k-1}}(t_{k-1}^{q-1}\alpha 
      -t_k^{q^2-q})-t_2^{q^{k-1}})s^{q^2}X^{q^{2k-2}} \\
      &+&(\frac{1}{s^{q^k}}(-t_{k-1}^{q-1}\alpha+t_k^{q^2-q})-\frac{1}{s^{q^k}}t_1^{q^k})s^{q^2}X^{q^{2k-1}} + \frac{s^{q^2}}{s^{q^{k+1}}}X^{q^{2k}}.
    \end{eqnarray*}
    In case $k=4$, let
    \begin{eqnarray*}
      f(X)&=&s^{q^2-1}(t_{3}t_4)^{q-1}\alpha X+(-s^{q^2-1}(t_{3}t_4)^{q-1}\alpha t_1+s^{q^2-q}t_{3}^{q-1}\alpha - s^{q^2-q}t_4^{q^2-q})X^q \\
      &+&(-s^{q^2}(t_{3}t_4)^{q-1}\alpha t_2-s^{q^2-q}t_{3}^{q-1} \alpha t_1^q +s^{q^2-q}t_4^{q^2-q}t_1^q-1)X^{q^2} \\
      &+&(-s^{q^2}(t_{3}^2-x^2t_4^2)(t_{3}t_4)^{q-1}\alpha -s^{q^2}t_{3}^{q-1}\alpha t_{2}^q+s^{q^2}t_4^{q^2-q}t_{2}^q+t_{1}^{q^2})X^{q^{3}} \\
      &+&(s^{q^2}(t_{3}^2-x^2t_4^2)^qt_{3}^{q-1}\alpha +s^{q^2}(-t_{3}^2-x^2t_4^2)^qt_4^{q^2-q}+s^{q^2}t_{2}^{q^2}-2s^{q^2}t_{3}^{q^2+q} \\
      &&+s^{q^2}(t_{3}t_4)^{q-1}\alpha t_{2}^q) X^{q^{4}} \\
      &+&(\frac{1}{s^{q^2}}(t_{3}t_4)^{q-1}\alpha t_{1}^{q^2} 
      -(t_{3}^2-x^2t_4^2)^{q^2}+t_{2}^{q^2}(t_{3}^{q-1}\alpha -t_4^{q^2-q}))s^{q^2}X^{q^{5}} \\
      &+&(-\frac{1}{s^{q^{3}}}(t_{3}t_4)^{q-1}\alpha+\frac{1}{s^{q^{3}}}t_1^{q^{3}}(t_{3}^{q-1}\alpha 
      -t_4^{q^2-q})-t_2^{q^{3}})s^{q^2}X^{q^{6}} \\
      &+&(\frac{1}{s^{q^4}}(-t_{3}^{q-1}\alpha+t_4^{q^2-q})-\frac{1}{s^{q^4}}t_1^{q^4})s^{q^2}X^{q^{7}} + \frac{s^{q^2}}{s^{q^{5}}}X^{q^{8}}.
    \end{eqnarray*}
    Then the solution field of $M$ is generated by the roots of $f$.
    In particular, $\Gal_{\Fq(t_1,\dots,t_k,s)}(f) \cong \SO_n^{-}(q)$.
\end{enumerate}
\end{thm}
\begin{proof}
a) Theorem~\ref{thm:upperbound} asserts that $\Gal^\Phi(M)\leq \mathcal H(\Fq) \cong \SO_n^{-}(q)$.
For the lower bound, we prove that there are specializations $\psi: \Fq[t_1,\dots,t_k,s] \to \Fq$, $t_i \mapsto a_i \in \Fq$, $s \mapsto a_s \in \Fq^{\times}$ such that $\psi(D)$ is $\GL_n(\Fqbar)$-conjugate to the generators ${\genname}_1$ and ${\genname}_2$ given in Table~\ref{table:generators}.
Since the characteristic polynomials of ${\genname}_1$ and ${\genname}_2$ are separable, it suffices to show that there are specializations such that the characteristic polynomial $h$ specializes to the characteristic polynomials of ${\genname}_1$ and ${\genname}_2$. 
By Lemma~\ref{lem:symmetric coefficients}, the latter have the form
$\displaystyle \charpolyorth(X)=X^n+\sum \limits _{i=1}^{k-1}a_iX^{n-i}+a_kX^k+\sum \limits _{i=1}^{k-1}a_iX^{i}+1$ ($a_1,\ldots,a_k \in \Fq$).
The characteristic polynomial of $\pcs_k(t_1,\dots,t_k,s)$ was already computed in~\ref{sec:Dk}.
It follows that the characteristic polynomial of $D$ is
\begin{eqnarray*}
	h(X)&=&X^n-t_1X^{n-1}-s(t_2+1)X^{n-2}+(-s t_3+t_1)X^{n-3} \\
	&&+\sum_{i=4}^{k-2}(-s t_i+s t_{i-2})X^{n-i}+(s(t_{k-1}^2-x^2t_k^2)+s t_{k-3})X^{k+1} \\
	&&+s(2t_{k-2}-(t_{k-1}+xt_k)^2-(t_{k-1}-xt_k)^2)X^k \\
	&&+(s(t_{k-1}^2-x^2t_k^2)+s t_{k-3})X^{k-1} \\
	&&+ \sum_{i=4}^{k-2}(-s t_i+s t_{i-2})X^{i}+(-s t_3+t_1)X^{3}-s(t_2+1)X^{2}-t_1X+1
.\end{eqnarray*}
Therefore, we specialize 
\begin{align*}
t_1 &\mapsto -a_1, \qquad
t_2  \mapsto \frac{1}{s}(-a_2-1), \\
t_3 &\mapsto \frac{1}{s}(-a_3+t_1)=\frac{1}{s}(-a_3-a_1), \\
t_i &\mapsto \frac{1}{s}(-a_i-a_{i-2}-a_{i-4}-\dots), \quad 3\leq i \leq k-2, \\
 s(t_{k-1}^2-x^2t_k^2) &\mapsto a_{k-1}-s t_{k-3}, \\
2s(t_{k-1}^2+x^2t_k^2) &\mapsto 2s t_{k-2}-a_{k}.
\end{align*}
We have to show that the last two equations are solvable for $t_{k-1},t_k \in \Fq$ and $s \in \Fq^{\times}$, where $t_{k-2}$ and $t_{k-3}$ are given by the preceding equation.
We multiply the last but one equation by 2 and -2, respectively, and add it to the last equation.
Plugging in the definition of $t_{k-2}, t_{k-3}$ yields
\begin{subequations}
\begin{align}
4 s t_{k-1}^2 &\;=\; 2\sum_{i=1}^{k-2}(-1)^{k-1+i}a_i+2(-1)^{k-1} +2a_{k-1}-a_k=(-1)^{k-1}\charpolyorth(-1), \label{00020} \\
4 s t_k^2     &\;=\; \frac{1}{x^2}(-2 \sum_{i=1}^{k-2} a_i \;- 2 -2a_{k-1}-a_k)=\frac{-\charpolyorth(1)}{x^2}. \label{00021}
\end{align} 
\end{subequations}
We need to show that we can choose $s \in \Fq^{\times}$ such that when we divide the right hand sides of equations (\ref{00020}) and (\ref{00021}) by $s$ they both become squares in~$\Fq$.
This means that we have to show that $(-1)^{k-1}\charpolyorth(-1)$ and $-\charpolyorth(1)/x^2$ are both either squares or non-squares in $\Fq$.
Thus it suffices to show that $(-1)^{k}\charpolyorth(-1) \charpolyorth(1)\,x^2$ is a square in $\Fq$.
Recall that $x^2$ is a non-square in $\Fq$ because $x^q=-x$.
Since $\charpolyorth$ is of the palindromic form 
$\charpolyorth(X)=\prod \limits _{i=1}^{k}(X-\alpha_i)(X-\alpha_i^{-1})$, we have
$$\charpolyorth(1)\charpolyorth(-1)=\prod \limits_{i=1}^{k} \frac{(1-\alpha_i)^2(-1)^k(\alpha_i+1)^2}{\alpha_i^2}.$$
Therefore, $(-1)^{k}\charpolyorth(-1) \charpolyorth(1)\,x^2$ is a square in $\Fq$ if and only if the element
$$z \coloneqq x\cdot\prod \limits_{i=1}^{k} \frac{(1-\alpha_i)(\alpha_i+1)}{\alpha_i} = x \cdot\prod \limits_{i=1}^{k} (\frac{1}{\alpha_i}-\alpha_i)$$ lies in~$\Fq$.
We must check this for both generators ${\genname}_1$ and ${\genname}_2$.
In the case of ${\genname}_1$, the roots $\alpha_1,\dots\alpha_k$ are equal to $\alpha,\alpha^q,\dots,\alpha^{q^{k-1}}$, where $\alpha^{q^k}=\alpha^{-1}$.
We calculate
\begin{eqnarray*}
z^q &=&  -x\cdot \big(\frac{1}{\alpha^{q^k}}-\alpha^{q^k}\big)
           \cdot \prod_{i=1}^{k-1} \big(\frac{1}{\alpha^{q^i}}-\alpha^{q^i}\big)
\;\;=\;\;-x\cdot \big(\alpha-\frac{1}{\alpha}\big)
           \cdot \prod_{i=1}^{k-1} \big(\frac{1}{\alpha^{q^i}}-\alpha^{q^i}\big) \\
&=&\phantom{-}x\cdot\prod \limits_{i=0}^{k-1} \big(\frac{1}{\alpha^{q^i}}-\alpha^{q^i}\big)
 \;\;=\;\;z.
\end{eqnarray*}
For ${\genname}_2$, the computation is quite similar, hence $z\in\Fq$ in both cases.
We conclude $\psi_i(D)={\genname}_i$ so that $\Gal^{\Phi}(M)$ contains $\GL_{2n}(\Fqbar)$-conjugates to $\phi({\genname}_i)$.
More precisely, $\Gal^{\Phi}(M)$ contains $\mathcal{H}(\Fqbar)$-conjugates to $\phi({\genname}_i)$, say $\phi({\genname}_i)^{u_i}$, by the second part of Theorem~\ref{thm:lowerbound}.
We extend $\phi: \SO_n^{-}(q) \to \mathcal H (\Fq)$ to a morphism of linear algebraic groups 
$\phi: \GL_n(\Fqbar) \to \GL_{2n}(\Fqbar)$,
$$(a_{ij}) \mapsto 
\left(
\begin{pmatrix}
\frac{a_{ij}+a_{ij}^q}{2} & x \frac{a_{ij}-a_{ij}^q}{2}  \\
 \frac{a_{ij}-a_{ij}^q}{2x} & \frac{a_{ij}+a_{ij}^q}{2}\\
\end{pmatrix}_{ij} \right).$$
Now we extend $\phi^{-1}: \mathcal H (\Fq) \to \SU_n(q)$ to a morphism of linear algebraic groups 
$\psi: \mathcal H (\Fqbar) \to \GL_n(\Fqbar)$, 
$$
\left(
\begin{pmatrix}
a_{ij} & x^2b_{ij}  \\
b_{ij} & a_{ij}\\
\end{pmatrix}_{ij} \right) \mapsto (a_{ij}+xb_{ij})_{ij}.
$$ Define $\gentilde_i=\psi(\phi({\genname}_i)^{u_i})$. Now $\psi$ is defined on all of $\mathcal{H}(\Fqbar)$ thus $\gentilde_i={\genname}_i^{\psi_i(u_i)}$ is $\GL_n(\Fqbar)$-conjugate to ${\genname}_i$. Since ${\genname}_1$ and ${\genname}_2$ form a strong pair of generators, $\gentilde_1$ and $\gentilde_2$ generate $\SO_n^{-}(q)$ and thus $\phi({\genname}_1)^{u_1}=\phi(\gentilde_1)$ and $\phi({\genname}_2)^{u_2}=\phi(\gentilde_2)$ generate $\mathcal{H}(\mathbb{F}_q)$. 

\bigskip\noindent
b) We compute the solutions to $\phi(D)\cdot \phi_q(x)=x$ with $\phi(D)$ consisting of $n^2$ $(2\times 2)$-matrices.
For $2 \leq i \leq k-1$, the only nonzero entry in the $i$-th row of $D$ is in the $(i+1)$-th column and equals~$1$.
Therefore, $\phi(D)$ contains the identity matrix $I_2$ at this position and zeros elsewhere in these rows.
The rows and columns of $\phi(D)$ are understood to be rows and columns of $2 \times 2$-matrices.
Let $x=(x_1,y_1,\dots,x_n,y_n)^{tr} \in \overline{\Fq(\underline t, s)}^{2n}$ be a solution to $\phi(D)\cdot \phi_q(x)=x$, where $n=2k$.
Then 
\begin{eqnarray*}
 x_i&=&x_{i-1}^q=x_1^{q^{i-1}}, \ 2\leq i \leq {k-1}, \label{00040} \\
 y_i&=&y_{i-1}^q=y_1^{q^{i-1}}, \ 2\leq i \leq {k-1}.
\end{eqnarray*}

\noindent Recall that the entries 
$-\frac{1}{s}, \frac{1}{s}t_1,t_2,\dots,t_{k-2}$ of $D$ all lie in $\Fq(\underline{t})$ and therefore correspond to the entries $-\frac{1}{s}I_2, \frac{1}{s}t_1I_2,t_2I_2,\dots,t_{k-2}I_2$ in $\phi(D)$.
We can now evaluate the $2k$-th through $(k+2)$-nd row of $\phi(D)$, successively.
\begin{alignat}{2}
x_{2k}&= -\frac{1}{s}x_1^{q^{k-1}}, \qquad \qquad&
y_{2k}&= -\frac{1}{s}y_1^{q^{k-1}}, \nonumber \\
x_{2k-1}&= \frac{1}{s}t_1x_1^{q^{k-1}}-\frac{1}{s^q}x_1^{q^k}, \qquad \qquad&
y_{2k-1}&= \frac{1}{s}t_1y_1^{q^{k-1}}-\frac{1}{s^q}y_1^{q^k}, \nonumber \\
&\;\vdots &&\;\vdots \nonumber \\
x_{k+2}&= t_{k-2}x_1^{q^{k-1}}+t_{k-3}^qx_1^{q^k}+\dots+t_2^{q^{k-4}}&x_1^{q^{2k-5}}&+\frac{1}{s^{q^{k-3}}}t_1^{q^{k-3}}x_1^{q^{2k-4}}-\frac{1}{s^{q^{k-2}}}x_1^{q^{2k-3}}, \nonumber \\
y_{k+2}&= t_{k-2}y_1^{q^{k-1}}+t_{k-3}^qy_1^{q^k}+\dots+t_2^{q^{k-4}}&y_1^{q^{2k-5}}&+\frac{1}{s^{q^{k-3}}}t_1^{q^{k-3}}y_1^{q^{2k-4}}-\frac{1}{s^{q^{k-2}}}y_1^{q^{2k-3}}. \label{00052} 
\end{alignat}
Note that $x_2,y_2,\dots,x_{k-1},y_{k-1},x_{k+2},y_{k+2}\dots,x_n,y_n$ are uniquely and polynomially determined by $x_1$ and $y_1$.
What is left is to evaluate rows $k+1, k, 1$.
The entries in $\phi(D)$ corresponding to $s(t_{k-1}^2-x^2t_k^2),s(t_{k-1}-xt_k),s(t_{k-1}+xt_k)$ are 
$$
\begin{pmatrix}
s(t_{k-1}^2-x^2t_k^2)&0 \\
0&s(t_{k-1}^2-x^2t_k^2)
\end{pmatrix},
\begin{pmatrix}
s t_{k-1}&-x^2s t_k \\
-s t_k&s t_{k-1}
\end{pmatrix},
\begin{pmatrix}
s t_{k-1}&x^2s t_k \\
s t_k&s t_{k-1}
\end{pmatrix}.$$
Therefore,
\begin{subequations}
\begin{eqnarray}
x_{k+1}&=&t_{k-1}x_{k-1}^q+x^2t_ky_{k-1}^q+x_k^q, \label{00043}\\
y_{k+1}&=&t_{k}x_{k-1}^q+t_{k-1}y_{k-1}^q+y_k^q, \label{00044}\\
x_k&=&t_{k-1}x_{k-1}^q-x^2t_ky_{k-1}^q+x_{k+1}^q, \label{00045}\\
y_k&=&-t_{k}x_{k-1}^q+t_{k-1}y_{k-1}^q+y_{k+1}^q, \label{00046}\\
x_1&=&t_1x_1^q+s t_2 x_1^{q^2}+\dots+s t_{k-2}x_1^{q^{k-2}}+s(t_{k-1}^2-x^2t_k^2)x_1^{q^{k-1}}, \nonumber \\
&&+ s t_{k-1}(x_k^q+x_{k+1}^q)+x^2s t_k(-y_k^q+y_{k+1}^q)-s x_{k+2}^q, \label{00041}\\
y_1&=&t_1y_1^q+s t_2 y_1^{q^2}+\dots+s t_{k-2}y_1^{q^{k-2}}+s(t_{k-1}^2-x^2t_k^2)y_1^{q^{k-1}}, \nonumber \\
&&+s t_{k}(-x_k^q+x_{k+1}^q)+s t_{k-1}(y_k^q+y_{k+1}^q)-s y_{k+2}^q. \label{00042}
\end{eqnarray} 
\end{subequations}
For the sake of convenience, we don't replace $x_{k-1},y_{k-1},x_{k+2}$ and $y_{k+2}$, yet.
Note that the variables $x_i$ and $y_i$ are separated in equation (\ref{00041}) and (\ref{00042}), apart from the occurrence of $x_k,x_{k+1},y_k,y_{k+1}$.
These will be eliminated below.
We will get one polynomial in $x_1$ and one in $y_1$.
This is possible because all of the elements in the first row of $D$, except those in the two middle columns, lie in $\Fq(\underline t)$, since $F$ fixes all but the last two simple roots.
Note that the fact that $F$ fixes all but two simple roots makes it hard to realize $\SO_n^{-}(q)$ over $\mathbb{F}_{q^2}(\underline t)$, but on the other hand makes it easy to compute polynomials over $\Fq(\underline t)$.
We return to the elimination process. Equations (\ref{00043}), (\ref{00044}), (\ref{00045}), and (\ref{00046}) yield:
\begin{subequations}
\begin{eqnarray}
-x_k^q+x_{k+1}^q&=&x_k-x_{k+1}+2x^2t_k y_{k-1}^q, \label{00055} \\
x_k^q+x_{k+1}^q&=&x_k+x_{k+1}-2t_{k-1} x_{k-1}^q, \label{00050}\\
-y_k^q+y_{k+1}^q&=&y_k-y_{k+1}+2t_k x_{k-1}^q, \label{00056} \\
y_k^q+y_{k+1}^q&=&y_k+y_{k+1}-2t_{k-1} y_{k-1}^q. \label{00057}
\end{eqnarray}
\end{subequations}
Plugging this into (\ref{00041}) and (\ref{00042}) yields:
\begin{subequations}
\begin{eqnarray}
0&=&-x_1+t_1x_1^q+s t_2 x_1^{q^2}+\dots+s t_{k-2}x_1^{q^{k-2}}-s(t_{k-1}^2-x^2t_k^2)x_1^{q^{k-1}} \nonumber \\
&&+ s t_{k-1}(x_k+x_{k+1})-x^2s t_k(-y_k+y_{k+1})-s x_{k+2}^q, \label{00047}\\
0&=&-y_1+t_1y_1^q+s t_2 y_1^{q^2}+\dots+s t_{k-2}y_1^{q^{k-2}}-s(t_{k-1}^2-x^2t_k^2)y_1^{q^{k-1}} \nonumber \\
&&-s t_{k}(-x_k+x_{k+1})+s t_{k-1}(y_k+y_{k+1})-s y_{k+2}^q. \label{00048}
\end{eqnarray}
\end{subequations}
In order to eliminate $y_k$ and $y_{k+1}$, we raise~(\ref{00047}) to its $q$-th power and add it to~(\ref{00041}) multiplied by $s^{q-1}t_{k}^{q-1}$.
Similarly, to eliminate $x_k$ and $x_{k+1}$ we raise~(\ref{00048}) to its $q$-th power and add it to~(\ref{00042}) multiplied by $s^{q-1}t_{k}^{q-1}$.
This yields:
\begin{subequations}
\begin{eqnarray}
0&=&-s^{q-1}t_k^{q-1}x_1+(s^{q-1}t_k^{q-1}t_1-1)x_1^q+(s^qt_k^{q-1}t_2+t_1^q)x_1^{q^2}+(s^qt_k^{q-1}t_3+s^q t_2^q)x_1^{q^3} \nonumber \\
&&+\dots+(s^qt_k^{q-1}t_{k-2}+s^q t_{k-3}^q)x_1^{q^{k-2}}+(s^q(t_{k-1}^2-x^2t_k^2)t_k^{q-1}+s^qt_{k-2}^q)x_1^{q^{k-1}} \nonumber \\
&&-s^q(t_{k-1}^2-x^2t_k^2)^qx_1^{q^k}-s^qt_k^{q-1}x_{k+2}^q-s^qx_{k+2}^{q^2}+(s^q t_{k-1}t_k^{q-1}+s^q t_{k-1}^q)(x_k^q+x_{k+1}^q), \nonumber \\ &&\label{00049} \\
0&=&-s^{q-1}t_k^{q-1}y_1+(s^{q-1}t_k^{q-1}t_1-1)y_1^q+(s^qt_k^{q-1}t_2+t_1^q)y_1^{q^2}+(s^qt_k^{q-1}t_3+s^q t_2^q)y_1^{q^3} \nonumber \\
&&+\dots+(s^qt_k^{q-1}t_{k-2}+s^q t_{k-3}^q)y_1^{q^{k-2}}+(s^q(t_{k-1}^2-x^2t_k^2)t_k^{q-1}+s^qt_{k-2}^q)y_1^{q^{k-1}} \nonumber \\
&&-s^q(t_{k-1}^2-x^2t_k^2)^qy_1^{q^k}-s^qt_k^{q-1}y_{k+2}^q-s^qy_{k+2}^{q^2}+(s^q t_{k-1}t_k^{q-1}+s^q t_{k-1}^q)(y_k^q+y_{k+1}^q).\nonumber\\
~
\end{eqnarray}
\end{subequations}
These equations are identical (considered as polynomials in $x_1$ and $y_1$).
Moreover, $x_{k+2}$ and $y_{k+2}$ depend in the same way on $x_1$ and $y_1$, respectively.
Similarly, $x_k^q+x_{k+1}^q$ and $y_k^q+y_{k+1}^q$ depend in the same way on $x_k+x_{k+1}, x_{1}$ and $y_k+y_{k+1}, y_{1}$.
Therefore, the polynomial~$f$ derived below with root $x_1$ has $y_1$ as root as well.

\bigskip
\noindent
Plugging (\ref{00050}) into~(\ref{00049}) yields
\begin{equation}
\begin{split}
0 &= -s^{q-1}t_k^{q-1}x_1+(s^{q-1}t_k^{q-1}t_1-1)x_1^q+(s^qt_k^{q-1}t_2+t_1^q)x_1^{q^2}+(s^qt_k^{q-1}t_3+s^q t_2^q)x_1^{q^3}\\
&+\dots+(s^qt_k^{q-1}t_{k-2}+s^q t_{k-3}^q)x_1^{q^{k-2}}+(s^q(-t_{k-1}^2-x^2t_k^2)t_k^{q-1}+s^qt_{k-2}^q-2s^q t_{k-1}^{q+1})x_1^{q^{k-1}} \\
&-s^q(t_{k-1}^2-x^2t_k^2)^qx_1^{q^k}-s^qt_k^{q-1}x_{k+2}^q-s^qx_{k+2}^{q^2}+s^qt_{k-1}( t_k^{q-1}+t_{k-1}^{q-1})(x_k+x_{k+1}). \label{00051} 
\end{split}
\end{equation}
Again, we raise (\ref{00051}) to its $q$-th power and add it to equation (\ref{00049}) multiplied by $(-s^{q^2-q}t_{k-1}^{q-1}(t_{k-1}^{q-1}+t_k^{q-1})^{q-1})$.
While doing so, we substitute $x_{k+2}$ using~(\ref{00052}).
We abbreviate $\alpha=(t_{k-1}^{q-1}+t_k^{q-1})^{q-1}$ and obtain the polynomial $f$ from the statement of the theorem, which therefore satisfies $f(x_1)=f(y_1)=0$.

It remains to show that the splitting field of $f$ equals the solution field of $(M,\Phi)$.
For any element $(x_1,y_1,\dots,x_{2k},y_{2k})^{tr}$ in the solution space of $M$, we computed above that $x_i, y_i$ are uniquely and polynomially determined by $x_1$ and $y_1$ for $2 \leq i \leq k-2, \ k+2\leq i\leq2k$.
From (\ref{00051}) we see that $(x_{k}+x_{k+1})$ depends uniquely and polynomially on $x_1,y_1$.
Similarly, $(y_{k}+y_{k+1})$ depends uniquely and polynomially on $x_1, y_1$.
Using (\ref{00041}) and (\ref{00042}) together with (\ref{00055}) and (\ref{00056}) we also get a unique and polynomial dependence of $(y_k-y_{k+1})$ and $(x_k-x_{k+1})$ on $x_1,y_1$.
Using $x_k=\frac{1}{2}(x_{k}+x_{k+1})+\frac{1}{2}(x_{k}-x_{k+1})$ and similar equalities yields that $x_k,x_{k+1},y_k,y_{k+1}$ are contained in the field extension generated by $x_1,y_1$.
Furthermore, $x_1$ and $y_1$ are both roots of $f$, hence the solution field of $M$ is contained in the splitting field of $f$.
On the other hand, $f$ has degree $q^3$ and there are $q^6$ solutions to $\phi(D)\phi_q(x)=x$, so every combination $(x_1,y_1)$ of roots of $f$ occurs as first two coordinates in a solution vector.
The solution field is generated by the coordinates of the solution vectors, so the splitting field of $f$ is contained in the solution field of $M$.
\end{proof}

%% file: polynomial_for_2B2.tex
\subsection{The Suzuki Groups $\Suz(q)$}

Throughout this section, $q$ denotes an odd power of 2, say $q=2^{2m+1}, \ m\geq0$. We prove that the Suzuki group $\Suz(q)$ occurs as a Galois group over $\mathbb{F}_{q}(t)$, the function field in one variable over $\Fq$.
Let $\G=\operatorname{Sp}_{4}(K)$ be the $4$-dimensional symplectic group over an algebraic closure $K$ of $\mathbb{F}_{q}(t)$. There is a Frobenius map $F:\G\rightarrow \G$ with $F^{2}=\phi_{q}$. The map $F$ is given explicitly in \cite[4.6.]{Geck} The group $\G^{F}\leq \operatorname{Sp}_{4}(\mathbb{F}_{q})$ is of type $\Suz(q)$ and there is only one isomorphism class of groups of type $\Suz(q)$. The order of $\Suz(q)$ is $q^2(q-1)(q^2+1)$.

The simple roots of $\operatorname{Sp}_4(K)$ are $\alpha=e_1-e_2$ and $\beta=2e_2$ with corresponding root subgroups $X_\alpha$ and 
$X_\beta$ and corresponding reflections $\omega_\alpha$ and $\omega_\beta$, respectively.
These are defined as
\begin{alignat*}{2}
x_\alpha(t) &=
\idblock{\mathbf{2},\mathbf{2}}{%
    \begin{pmatrix}
    		1&t\\
    		0&1 
    \end{pmatrix},
    \begin{pmatrix}
    		1&t\\
    		0&1 
    \end{pmatrix}
     }
,&\qquad
x_\beta(t) &=
\idblock{\mathbf{1},\mathbf{2}, \mathbf{1}}{%
    \begin{pmatrix}
    		1&t\\
    		0&1 
    \end{pmatrix}
     },\\
\omega_\alpha &=
\idblock{\mathbf{2},\mathbf{2}}{%
    \begin{pmatrix}
    		0&1\\
    		1&0 
    \end{pmatrix},
    \begin{pmatrix}
    		0&1\\
    		1&0 
    \end{pmatrix}
     },
&\omega_\beta &=
\idblock{\mathbf{1},\mathbf{2}, \mathbf{1}}{%
    \begin{pmatrix}
    		0&1\\
    		1&0 
    \end{pmatrix}
     }.
\end{alignat*}
\noindent
The Frobenius map $F$ permutes the root subgroups in the following way: 
$$F(x_\alpha(t))=x_\beta(t^{2^{m+1}}), \ F(\omega_\alpha)=\omega_\beta, \ F(x_\beta(t))=x_\beta(t^{2^{m}}), \ F(\omega_\beta)=\omega_\alpha. $$

\noindent Define $\zwmat \coloneqq x_\alpha(t) \cdot \omega_\alpha \ \in \G(\mathbb{F}_{q^2}(t))$ and $D \coloneqq \zwmat \cdot F(\zwmat)$.
Then
$$D= x_\alpha(t) \cdot \omega_\alpha \cdot x_\beta(t^{2^{m+1}}) \cdot \omega_\beta =
\begin{pmatrix} 
	t &t^{2^{m+1}} &1&0 \\  
	1 &0 &0 &0 \\ 
	0& t & 0 &1 \\  
	0&1 &0 & 0 
\end{pmatrix}.$$

\begin{lemma}\label{uschrsuz}
There are specializations $\mathbb{F}_{q}[t] \rightarrow  \mathbb{F}_{q}$, $t\mapsto u_{1/2} \in \mathbb{F}_{q}$ such that the specialized matrices $D_{u_{1/2}}$ have order $(q\pm 2^{m+1}+1)$.
\end{lemma}

\begin{proof}
The characteristic polynomial of $D$ is
$$h(X)=X^4+t\cdot X^3 + t^{2^{m+1}}\cdot X^2 + t\cdot X +1.$$
Let $\zeta_{1}$, $\zeta_{2} \in \overline{\F}_{2}$ be roots of unity of order $(q+2^{m+1}+1)$, $(q-2^{m+1}+1)$, respectively. Then $\zeta_{i}^{(q^2+1)}=\zeta_{i}^{(q+2^{m+1}+1)\cdot(q-2^{m+1}+1)}=1$, hence $\zeta_{i}^{q^2}=\zeta_{i}^{-1}$ and $\zeta_{i}^{q^4}=\zeta_{i}.$ 
Therefore, the minimal polynomial $g_{i}$ of $\zeta_i$ over $\mathbb{F}_{q}$ is
\begin{center}
	$\left(X-\zeta_{i}\right) \left(X-\zeta_{i}^q\right) (X-\zeta_{i}^{q^2})  (X-\zeta_{i}^{q^3})=\left(X-\zeta_{i}\right) 				\left(X-\zeta_{i}^q\right)  \left(X-\zeta_{i}^{-1}\right)  \left(X-\zeta_{i}^{-q}\right)$.
\end{center}
Note that the inverse of any root of $g_{i}$ is also a root of $g_{i}$.
By Lemma~\ref{lem:symmetric coefficients} the coefficients of $g_{i}$ have the following symmetry:
$g_{i}(X)=X^4+a_{i} X^3 + b_{i} X^2 + a_{i} X +1$,
where $a_{i} \coloneqq \zeta_{i}+\zeta_{i}^{-1} +\zeta_{i}^q +\zeta_{i}^{-q}$, 
\ \ $b_{i} \coloneqq \zeta_{i}^{1+q}+\zeta_{i}^{-1-q} +\zeta_{i}^{-1+q} +\zeta_{i}^{1-q}$.
Using $\zeta_{i}^{2^{m+1}}=\zeta_{i}^{\pm (q+1)}$, we get 
$a_{i}^{2^{(m+1)}}=\zeta_{i}^{1+q}+\zeta_{i}^{-1-q} +\zeta_{i}^{q+q^2} +\zeta_{i}^{-q-q^2}=\zeta_{i}^{1+q}+\zeta_{i}^{-1-q} +\zeta_{i}^{q-1} +\zeta_{i}^{-q+1}=b_{i}.$
Thus the characteristic polynomial of $D_{a_{i}}$ equals $g_i$, i.e. the Jordan canonical form of $D_{a_{i}}$ over $\mathbb{F}_{q^4}$ is ${\diag}(\zeta_{i},\zeta_{i}^{-1},\zeta_{i}^{q},\zeta_{i}^{-q})$.
Therefore, $D_{a_{1}}$ and $D_{a_{2}}$ are of order $(q+2^{m+1}+1)$ and $(q-2^{m+1}+1)$, respectively.
\end{proof}

\begin{thm}\label{frobmodsuz}
Let $(M,\Phi)$ be the $4$-dimensional Frobenius module over $(\mathbb{F}_{q}(t), \phi_{q})$ such that the representing matrix of~$\Phi$ with respect to a fixed basis of~$M$ equals~$D$.
Then
\begin{enumerate}
  \item $\Gal^{\Phi}(M) \cong {\Suz(q)}$.
  \item The solution field of $M$ is generated by the roots of the additive polynomial
    $$f(X) \coloneqq X^{q^4}+t^q\cdot X^{q^3} + t^{2^{m+1}}\cdot X^{q^2}+t\cdot X^q +X \in  \mathbb{F}_q(t)[X].$$
    In particular, $\Suz(q) \cong \Gal_{\mathbb{F}_q(t)}(f)$.
\end{enumerate}
\end{thm}

\begin{proof}
a) We have $\Gal^{\Phi}(M) \leq \G^F\cong {\Suz(q)}$ by Theorem~\ref{thm:twupperbound}.
In case $m=0$, we know by Lemma~\ref{uschrsuz} that $\Gal^{\Phi}(M)$ contains an element of order $5$ and is contained in $\G^F$, which is of cardinality~$20$.
The specialization $\mathbb{F}_{2}(t) \rightarrow  \mathbb{F}_{2}$, $t\mapsto 0$ maps $D$ to an element of order $4$, which, again by Theorem~\ref{thm:lowerbound}, is a conjugate of an element in $\Gal^{\Phi}(M)$. 
Now in case $m>0$, Theorem~\ref{thm:lowerbound} together with Lemma~\ref{uschrsuz} yield two elements contained in $\Gal^{\Phi}(M)$ of orders $(q+2^{m+1}+1)$, $(q-2^{m+1}+1)$, respectively.
From the list of maximal subgroups of $\Suz(q)$, given in \cite[Thm.\;9]{Suzuki}, it can be easily seen that no proper subgroup contains elements of these orders.

\medskip\noindent
b) The solution field $E \geq \mathbb{F}_{q}$ of $(M,\Phi)$ is generated by the coordinates of the $q^4$ solutions to $D\cdot \phi_{q}({x}) ={x}$, $x=(x_1,x_2,x_3,x_4) \in \Fqbar^{\;4}$.
This equation is equivalent to the system
\begin{center} 
	$\left( \! \! \begin{array}{cccc}   
			t &t^{2^{m+1}} &1&0 \\  
			1 &0 &0 &0 \\ 
			0& t & 0 &1 \\  
			0&1 &0 & 0  
	\end{array} \! \!	\right) \cdot 
	\left( \! \! \begin{array}{cccc}   
			x_{1}^q \\  
			x_{2}^q \\ 
			x_{3}^q \\  
			x_{4}^q  
	\end{array} \! \! \right)=
	\left( \! \! \begin{array}{cccc}   
			x_{1} \\  
			x_{2} \\ 
			x_{3} \\  
			x_{4}  
	\end{array} \! \! \right)$.
\end{center}
Abbreviating $X=x_1$, the second and fourth rows yield $x_{2}=X^q$ and $x_{4}=x_{2}^q=X^{q^2}$.
Next, we get $x_{3}=t\cdot x_2^q+x_4^q=t\cdot X^{q^2}+X^{q^3}$ out of the third row.
Plugging everything into the first row yields $t\cdot X^q + t^{2^{m+1}}\cdot X^{q^2}+t^{q}\cdot X^{q^3}+X^{q^4}=X.$
The theorem now follows from Lemma~\ref{lem:galoisgroup}.
\end{proof}

%% file: polynomial_for_G2_2G2.tex
\subsection{The Dickson Groups $\Dickson(q)$}

In this section, we realize the Dickson groups $\Dickson(q)$ as Galois groups over $\Fq(t,s)$.
Such a realization is already given in~\cite{Malle} but we would like to apply the unified approach by constructing a Frobenius module whose representing matrix lies in the Steinberg cross section.
We obtain slightly simpler polynomials than those in~\cite{Malle}.

\bigskip\noindent
Let first $q=2^r$ be even. The order of $\Dickson(q)$ is $$q^6(q^2-1)(q^6-1)=q^6 (q-1)^2 (q+1)^2 (q^2+q+1) (q^2-q+1).$$
We use the well-known $6$-dimensional representation of $\Dickson$ as linear algebraic group over $\overline{\Fq(t,s)}$. Then $\Dickson(q)=\Dickson(\Fq)$. Details can be found in \cite[§\,2]{Malle}. The group $\Dickson(q)$ has two simple roots $\alpha$ and $\beta$.
The root subgroups are determined in \cite{Malle} and it is easy to compute the corresponding reflections.
\footnotesize
\begin{alignat*}{2}
x_\alpha(t) \! &= \! \! 
\idblock{\mathbf{2},\mathbf{2},\mathbf{2}}{%
    \begin{pmatrix}
    		1&t\\
    		0&1 
    \end{pmatrix},
    \begin{pmatrix}
    		1&t^2\\
    		0&1 
    \end{pmatrix},
    \begin{pmatrix}
    		1&t\\
    		0&1 
    \end{pmatrix}
     },&\qquad
x_\beta(t) \! &= \! \!
\idblock{\mathbf{1},\mathbf{2},\mathbf{2},\mathbf{1}}{%
    \begin{pmatrix}
    		1&t\\
    		0&1 
    \end{pmatrix},
    \begin{pmatrix}
    		1&t\\
    		0&1 
    \end{pmatrix}
     },\\
w_\alpha \! &= \! \!
\idblock{\mathbf{2},\mathbf{2},\mathbf{2}}{%
    \begin{pmatrix}
    		0&1\\
    		1&0 
    \end{pmatrix},
    \begin{pmatrix}
    		0&1\\
    		1&0
    \end{pmatrix},
    \begin{pmatrix}
    		0&1\\
    		1&0
    \end{pmatrix}
     },
&w_\beta \! &= \! \!
\idblock{\mathbf{1},\mathbf{2},\mathbf{2},\mathbf{1}}{%
    \begin{pmatrix}
    		0&1\\
    		1&0 
    \end{pmatrix},
    \begin{pmatrix}
    		0&1\\
    		1&0
    \end{pmatrix}
     }.
\end{alignat*}
\normalsize
\noindent
We now define $D=x_\alpha(t)w_\alpha x_\beta(s) w_\beta$ and calculate
$$D=
\begin{pmatrix}
	t&s&1&0&0&0 \\
	1&0&0&0&0&0 \\
	0&t^2&0&s&1&0 \\
	0&1&0&0&0&0 \\
	0&0&0&t&0&1 \\
	0&0&0&1&0&0 \\
\end{pmatrix}. $$ 

\begin{thm}
Let $q=2^r$ be even and let $(M, \Phi)$ be a $6$-dimensional Frobenius module over $\Fq(t,s)$ such that the representing matrix of $\Phi$ equals $D$ for some fixed basis $B$ of $M$.
Then
\begin{enumerate}
  \item $\Gal^\Phi(M)\cong \Dickson(q)$.
  \item The solution field of $M$ is generated by the roots of the additive polynomial
    $$f(X)=X^{q^6}+t^{q^2}X^{q^5}+s^qX^{q^4}+t^{2q}X^{q^3}+sX^{q^2}+tX^q+X \;\;\in \mathbb{F}_q(t,s)[X].$$
    In particular, $\Gal_{\mathbb{F}_q(t,s)}(f) \cong \Dickson(q)$.
\end{enumerate}
\end{thm}

\begin{proof}
a) Theorem~\ref{thm:upperbound} yields $\Gal^\Phi(M)\leq \Dickson(2^r)$.
For the lower bound, note that the characteristic polynomial of $D$ is
$$h(X)=X^6+tX^5+sX^4+t^2X^3+sX^2+tX+1.$$ 
We construct elements of order $q^2 \pm q+1$ contained in $\Gal^\Phi(M)$.
Let $\zeta_1$ and $\zeta_2$ be roots of unity of order $q^2+q+1$ and $q^2-q+1$, respectively.
Note that $\zeta_1^{q^3}=\zeta_1$ and $\zeta_2^{q^3}=\zeta_2^{-1}$.
Define $h_i(X)=(X-\zeta_i)(X-\zeta_i^q)(X-\zeta_i^{q^2})(X-\zeta_i^{-1})(X-\zeta_i^{-q})(X-\zeta_i^{-q^2}) \in \Fq[X],$ for $i=1,2$.
The inverse of any root of $h_i$ is also a root, so $h_i$ has the following symmetry by Lemma~\ref{lem:symmetric coefficients}.
$$h_i(X)=X^6+\beta_iX^5+\gamma_iX^4+\delta_iX^3+\gamma_iX^2+\beta_iX+1,$$ 
where
\begin{eqnarray*}
	\beta_i&=&\zeta_i+\zeta_i^q+\zeta_i^{q^2}+\zeta_i^{-1}+\zeta_i^{-q}+\zeta_i^{-q^2}, \\
	\gamma_i&=&\zeta_i^{1+q}+\zeta_i^{1+q^2}+\zeta_i^{1-q}+\zeta_i^{1-q^2}+\zeta_i^{q+q^2}+\zeta_i^{q-1}+\zeta_i^{q-q^2}+
			\zeta_i^{q^2-1}+\zeta_i^{q^2-q} \\
			&&+\zeta_i^{-1-q}+\zeta_i^{-1-q^2}+\zeta_i^{-q-q^2}+1, \\
	\delta_i&=&\zeta_i^2+\zeta_i^{2q}+\zeta_i^{2q^2}+\zeta_i^{-2}+\zeta_i^{-2q}+\zeta_i^{-2q^2}=\beta_i^2.
\end{eqnarray*}
We can now define the specializations $\psi_i \colon \Fq[t,s] \to \Fq$, $t \mapsto \beta_i$, $s \mapsto \gamma_i$ for $i=1,2$.
Then $\psi_i(D)$ has the characteristic polynomial $h_i$.
Therefore, the order of $\zeta_i$ divides the order of $\psi_i(D)$.
Applying Theorem~\ref{thm:lowerbound}, we obtain that $\Gal^\Phi(M)$ contains elements of the orders $q^2 \pm q+1$.
For $q\geq8$, no proper subgroup of $\Dickson(2^r)$ contains such elements, which is be easily seen from the list of maximal subgroups in \cite[2.3.-2.5.]{Cooperstein}.
In case $q=4$, it suffices to construct an additional element contained in $\Gal^\Phi(M)$ of order $5$.
Let $\alpha \in \mathbb{F}_{16}$ be an element of order $5$ such that $a=\alpha+\alpha^{-1}\neq 1$.
Then $a \in \mathbb{F}_{4}^{\times}$ and so $a^2+a+1=0$.
Therefore,
\begin{align*}
	&(X^2+aX+1)(X^4+X^3+X^2+X+1) \\
	=\;\;&X^6+(1+a)X^5+aX^4+(1+a)^2X^3+aX^2+(1+a)X+1.
\end{align*}
Define $\psi \colon \Fq[t,s] \to \Fq$, $t \mapsto 1+a$, $s \mapsto a.$
Then the characteristic polynomial of $\psi(D)$ is $(X-\alpha)(X-\alpha^q) \cdot h(X)$, where $h(X)=X^4+X^3+X^2+X+1.$
We compute $h(\alpha)=\alpha^4+\alpha^3+\alpha^2+1 
=(a\alpha+1)^2+\alpha(a \alpha+1)+(a \alpha+1)+1 =\alpha\neq 0.$
Therefore, $(X-\alpha)(X-\alpha^q)$ divides the characteristic polynomial of $\psi(D)$ and its minimal polynomial both exactly once.
Consequently, the Jordan canonical form of $\psi(D)$ over $\mathbb{F}_{16}$ contains the $2 \times 2$ block $\diag(\alpha, \alpha^q)$.
Thus $5$ divides the order of $\psi(D)$.
Again by Theorem~\ref{thm:lowerbound}, $\Gal^\Phi(M)$ contains an element of $5$. 
For $q=2$, the specialization $\psi \colon \Fq[t,s] \to \Fq$, $t \mapsto 1$, $s \mapsto 0$ yields an additional element in $\Gal^\Phi(M)$ conjugate to $\psi(D)$ which is of order $12$, and no proper subgroup of $\Dickson(2)$ contains elements of order $2^2 \pm 2+1$ and $12$.  

\medskip\noindent
b) For $x=(x_1,\dots,x_6) \in \overline{\mathbb{F}_2(t,s)}$, we solve the equation $D \cdot \phi_q(x)=x$.
Letting $X=x_1$ and evaluating the rows successively, we get the following equivalent set of equations:
$x_2=X^q, \ x_4=X^{q^2}, \ x_6=X^{q^3}, \ x_5=tX^{q^3}+X^{q^4}, \ x_3=t^2X^{q^2}+sX^{q^3}+t^qX^{q^4}+X^{q^5}, \ X=tX^q+sX^{q^2}+t^{2q}X^{q^3}+s^qX^{q^4}+t^{q^2}X^{q^5}+X^{q^6}$
Hence the coordinates of a solution to $D\cdot\phi_q(x)=x$ all lie in the field extension generated by the first coordinate.
Since the first coordinate of any solution is a root of $f$, the theorem follows from Lemma~\ref{lem:galoisgroup}
\end{proof}

\begin{rem}
Compare the obtained polynomial with the one computed in \cite[Thm~4.1]{Malle}:
$\Gal_{\mathbb{F}_q(r,u)}(\tilde f) \cong \Dickson(q),$ where $\tilde f$ is defined as follows.
\begin{eqnarray*}
	\tilde f(X)&=&X^{q^6}+u^{e_2}r^{e_4}X^{q^5}+(u^{e_1}r^{e_1}+u^{e_3}r^{e_1}+r^{e_1}+r^{e_3}+1)X^{q^4} \\
					&&u^{e_2}r^{e_4}(r^{q^3+q^2}+r^{q^3-q^2}+1)X^{q^3}+r^{e_1}(u^{e_1}r^{q^2-1}+u^{e_1}r^{q^2+q}+u^{e_1} \\
					&&+u^{e_3}+1)X^{q^2} u^{e_2}r^{q^4+2q^2-q}X^q+u^{e_1}r^{q^4-1}X
\end{eqnarray*}
with $e_1 \coloneqq q^4-q^2, \ e_2 \coloneqq q^4-q^3, \ e_3 \coloneqq q^4+q^3, \ e_4 \coloneqq q^4-q^3+2q^2$. 
\end{rem}

\bigskip\noindent
Let now $q$ be odd.
We use the well-known $7$-dimensional representation of $\Dickson$ as linear algebraic group over $\overline{\Fq(t,s)}$.
Again, details can be found in \cite[§\,2]{Malle}.
Then $\Dickson(q)=\Dickson(\Fq)$ and the root subgroups corresponding to the two simple roots of $\Dickson$ are computed in \cite{Malle}.
It is easy to compute corresponding reflections.

\smallskip\footnotesize \noindent $ x_\alpha(t) \! = \! \! 
\idblock{\mathbf{2},\mathbf{3},\mathbf{2}}{%
    \begin{pmatrix}
    		1&t\\
    		0&1 
    \end{pmatrix},
    \begin{pmatrix}
    		1&t&-t^2\\
    		0&1&-2t\\
    		0&0&1 
    \end{pmatrix},
    \begin{pmatrix}
    		1&-t\\
    		0&1 
    \end{pmatrix}
     },$  \
    $\ x_\beta(t) \! = \! \!
\idblock{\mathbf{1},\mathbf{2},\mathbf{1},\mathbf{2},\mathbf{1}}{%
    \begin{pmatrix}
    		1&t\\
    		0&1 
    \end{pmatrix},
    \begin{pmatrix}
    		1&-t\\
    		0&1 
    \end{pmatrix}
     },$
$w_\alpha \! = \! \!
\idblock{\mathbf{2},\mathbf{3},\mathbf{2}}{%
    \begin{pmatrix}
    		0&1\\
    		-1&0 
    \end{pmatrix},
    \begin{pmatrix}
    		0&0&-1\\
    		0&-1&0\\
    		-1&0&0
    \end{pmatrix},
    \begin{pmatrix}
    		0&-1\\
    		1&0
    \end{pmatrix}
     },$  
  \  $w_\beta \! = \! \!
\idblock{\mathbf{1},\mathbf{2},\mathbf{1},\mathbf{2},\mathbf{1}}{%
    \begin{pmatrix}
    		0&1\\
    		-1&0 
    \end{pmatrix},
    \begin{pmatrix}
    		0&-1\\
    		1&0
    \end{pmatrix}
     }$.
\normalsize We define $D=x_\alpha(t)w_\alpha x_\beta(s) w_\beta$ and calculate
$$D=
\begin{pmatrix}
	-t&-s&1&0&0&0&0 \\
	-1&0&0&0&0&0&0 \\
	0&-t^2&0&-t&s&1&0 \\
	0&-2t&0&-1&0&0&0 \\
	0&1&0&0&0&0&0 \\
	0&0&0&0&-t&0&-1 \\
	0&0&0&0&1&0&0 \\
\end{pmatrix}. $$ 

\begin{thm}
\label{thm:polynomg2oddq}
Let $q$ be odd and $(M, \Phi)$ be a $7$-dimensional Frobenius module over $\Fq(t,s)$ such that the representing matrix of $\Phi$ equals $D$ for some fixed basis $B$ of $M$.
Then
\begin{enumerate}
  \item $\Gal^\Phi(M) \cong \Dickson(q)$.
  \item The solution field of $M$ is generated by the roots of the additive polynomial
    \begin{align*}
      f(X)=\;\;   &X^{q^7}+(t^{q^2-q}+t^{q^3})X^{q^6}+(t^{2q^2-q}-s^{q^2})X^{q^5}+(-s^qt^{q^2-q}-t^{2q^2})X^{q^4} \\
            +&(t^{q^2+q}+s^q)X^{q^3}+(st^{q^2-q}-t^q)X^{q^2}+(-t^{q^2-q+1}-1)X^q-t^{q^2-q}X.
    \end{align*}
    In particular, $\Gal_{\Fq(t,s)}(f) \cong \Dickson(q)$.
\end{enumerate}
\end{thm}

\begin{proof}
a) Theorem~\ref{thm:upperbound} asserts $\Gal^\Phi(M)\leq \Dickson(q)$.
For the lower bound, we first compute the characteristic polynomial $h$ of $D$, obtaining $h(X)=(X-1)(X^6+(t+2)X^5+(2+2t-s)X^4 +(2+2t-2s-t^2)X^3+(2+2t-s)X^2+(t+2)X+1).$
Let $\zeta_1$ and $\zeta_2$ be roots of unity of orders $(q^2+q+1)$ and $(q^2-q+1)$, respectively.
Note that $\zeta_1^{q^3}=\zeta_1$ and $\zeta_2^{q^3}=\zeta_2^{-1}$.
Define 
$h_i(X)=(X-\zeta_i)(X-\zeta_i^q)(X-\zeta_i^{q^2})(X-\zeta_i^{-1})(X-\zeta_i^{-q})(X-\zeta_i^{-q^2}) \in \Fq[X],$ for $i=1,2$.
The inverse of any root of $h_i$ is a root, too, so $h_i$ has the following symmetry by Lemma~\ref{lem:symmetric coefficients}:
$h_i(X)=X^6+\beta_iX^5+\gamma_iX^4+\delta_iX^3+\gamma_iX^2+\beta_iX+1.$
We compute 
\begin{eqnarray*}
  \beta_i  &=& -\zeta_i-\zeta_i^q-\zeta_i^{q^2}-\zeta_i^{-1}-\zeta_i^{-q}-\zeta_i^{-q^2}, \\
  \gamma_i &=&  \zeta_i^{1+q}+\zeta_i^{1+q^2}+\zeta_i^{1-q}+\zeta_i^{1-q^2}+\zeta_i^{q+q^2}+\zeta_i^{q-1}+\zeta_i^{q-q^2}+
                \zeta_i^{q^2-1}+\zeta_i^{q^2-q} \\
           & & +\zeta_i^{-1-q}+\zeta_i^{-1-q^2}+\zeta_i^{-q-q^2}+3, \\
  \delta_i &=& 2\beta_i-2-\zeta_i^2-\zeta_i^{2q}-\zeta_i^{2q^2}-\zeta_i^{-2}-\zeta_i^{-2q}-\zeta_i^{-2q^2}. 
\end{eqnarray*}
Furthermore, $\beta_i^2=\zeta_i^2+\zeta_i^{2q}+\zeta_i^{2q^2}+\zeta_i^{-2}+\zeta_i^{-2q}+\zeta_i^{-2q^2}+2\gamma_i$, so we conclude $\delta_i=2\beta_i-2+2\gamma_i-\beta_i^2.$ Define specializations $\psi_i \colon \Fq[t,s] \to \Fq$, $t\mapsto \beta_i-2$, $s \mapsto 2\beta_i-\gamma_i-2$. 
It follows from the above computations that the characteristic polynomial of $\psi_i(D)$ is $(X-1)h_i(X)$, which is easily seen to be separable.
Applying Theorem~\ref{thm:lowerbound} yields elements $h_1$, $h_2$ contained in $\Gal^\Phi(M)$ such that $h_i$ is conjugate to $\psi_i(D)$.
Therefore, the Jordan canonical form of $h_i$ over $\mathbb{F}_{q^6}$ equals $\diag(\zeta_i,\zeta_i^q,\zeta_i^{q^2},1,\zeta_i^{-q^2},\zeta_i^{-q},\zeta_i^{-1})$ and its order is $q^2 \pm q+1$.
We use the list of maximal subgroups in \cite[Thm.~A]{Kleidman} to see that for $q > 3$, any pair of elements in $\Dickson(q)$ of orders $q^2+q+1$ and $q^2-q+1$ generates $\Dickson(q)$.
The group $\Dickson(3)$ is generated by any such pair and an additional element of order $8$.
The specialization $\psi: \mathbb{F}_3[t,s] \to \mathbb{F}_3, \ t \to 0, s \to 1$ yields such an element.

\medskip\noindent
b) For $x=(x_1,\dots,x_7) \in \overline{\Fq(t,s)}^7$, consider the equation $D\cdot\phi_q(x)=x$ which is equivalent to the following system of equations (with $X=x_1$).
\begin{subequations}
\begin{align}
x_2 &= -X^q, \hspace*{12ex}
x_5  = -X^{q^2}, \hspace*{12ex}
x_7  = -X^{q^3}, \\
x_6 &= tX^{q^3}+X^{q^4}, \hspace*{6ex}
x_4  = 2tX^{q^2}-x_4^q, \label{00071} \\
x_3 &= t^2X^{q^2}-sX^{q^3}+t^qX^{q^4}+X^{q^5}-tx_4^q \nonumber \\
    &= -t^2X^{q^2}-sX^{q^3}+t^qX^{q^4}+X^{q^5}+tx_4, \\
X &= -tX^q+sX^{q^2}-t^{2q}X^{q^3}-s^qX^{q^4}+t^{q^2}X^{q^5}+X^{q^6}+t^qx_4^q. \label{00072}
\end{align}
\end{subequations}

\noindent
Note that all coordinates lie in the field extension generated by $X$ and $x_4$.
Moreover, even $x_4$ lies in the extension generated by $X$.
This can be seen after plugging the expression for~$x_4^q$ obtained from~(\ref{00071}) into~(\ref{00072}):
$-X-tX^q+(s+2t^{q+1})X^{q^2}-t^{2q}X^{q^3}-s^qX^{q^4}+t^{q^2}X^{q^5}+X^{q^6}-t^qx_4=0$.
We raise this equation to its $q$-th power and add it to equation (\ref{00072}) multiplied by $t^{q^2-q}$ to get $f(X)=0$. 
The rest follows from Lemma~\ref{lem:galoisgroup}.
\end{proof}

\begin{rem}
Compare the obtained polynomial with the one given in \cite[Thm~4.3]{Malle}: 
$$\Gal_{\Fq(r,u)}(\tilde f) \cong \Dickson(q),$$ where $\tilde f$ is defined as follows:
\begin{multline}
 \tilde f(X) =
    X^{q^7}
 +  u^{e_1}r^{e_4}(r^{e_6}+1) \cdot X^{q^6}
 - (r^{e_2}u^{e3}+(r^{q^5+q^2}+r^{e_2})u^{e_2}+r^{e_3}+r^{e_2}+1) \cdot X^{q^5}\\
 -  u^{e_1}r^{e_4} \cdot \big(r^{e_5}(u^{q^4+q^3}+u^{e_5})+(r^{e_6}+1)(r^{q^4+q^3}+r^{e_5}+1)\big) \cdot X^{q^4} \\	
 +  r^{e_2} \cdot \big(u^{e_3}+(r^{e_6}+1)(r^{e_6}+r^{q^3-q}+1)u^{e_2}+1 \big) \cdot X^{q^3} \\
 +  u^{e_1}r^{q^5+q^3-2q^2} \cdot \big(u^{q^4+q^3}+(r^{q^2+q}+r^{q^2-1}+1)u^{e_5}+r^{e_6}+1\big) \cdot X^{q^2} \\
 -  u^{e_2}r^{q^5-q}(r^{e_6}+1) \cdot X^q
 -  u^{q^5-q^2}r^{q^5+q^3-q^2-1} \cdot X.
\end{multline}
Here the $e_i$ are defined as $e_1=q^5-q^4, \ e_2=q^5-q^3, \ e_3=q^5+q^4, \ e_4=q^5-q^4+q^3-q^2, \ e_5=q^4-q^2, \ e_6=q^3+q^2.$
\end{rem}

\subsection{The Ree Groups $\Ree(q)$}
We realize the Ree groups of type $\Ree$ over $\Fq(t)$ and obtain polynomials that are specializations of the polynomials for $\Dickson(q)$ derived before.
The groups $\Ree(q)$ only exist for $q=3^{2m+1}$, $m\geq 0$. Let $\G=\Dickson$.
There exists a Frobenius map $F\colon \G \to \G$ such that $F^2=\phi_q$.
The corresponding twisted group $\G^F$ is of type $\Ree$ and there is only one isomorphism class of groups of this type.
The cardinality of $\Ree(q)$ is $q^3(q-1)(q^3+1)=q^3(q-1)(q+1)(q+3^{m+1}+1)(q-3^{m+1}+1).$
We use the seven-dimensional matrix representation of $\Ree(q)$ given in \cite{KLM}.
In this article, the authors use a Frobenius map $F: \Dickson(\Fqbar) \to \Dickson(\Fqbar)$ such that $\Dickson^F\cong {\Ree(q)}$.
They compute the images of all root subgroups of $\Dickson$ under $F$ and give explicit matrix generators.
The same arguments carry over to the case $F: \Dickson(\overline {\Fq(t)}) \to \Dickson(\overline {\Fq(t)})$.
Let $\alpha$ and $\beta$ denote the simple roots of $\Dickson(\overline {\Fq(t)})$.
The root subgroups consist of elements of the following form, where $s$ denotes an element in $\overline {\Fq(t)}$.
\footnotesize
\begin{align*}
x_\alpha(s) \! &= \! \! 
\idblock{\mathbf{2},\mathbf{3},\mathbf{2}}{%
    \begin{pmatrix}
    		1&s\\
    		0&1 
    \end{pmatrix},
    \begin{pmatrix}
    		1&s&-s^2\\
    		0&1&-2s\\
    		0&0&1 
    \end{pmatrix},
    \begin{pmatrix}
    		1&-s\\
    		0&1 
    \end{pmatrix}
     }, \quad
\ x_\beta(s) \! &= \! \!
\idblock{\mathbf{1},\mathbf{2},\mathbf{1},\mathbf{2},\mathbf{1}}{%
    \begin{pmatrix}
    		1&s\\
    		0&1 
    \end{pmatrix},
    \begin{pmatrix}
    		1&-s\\
    		0&1 
    \end{pmatrix}
     },  \\
x_{-\alpha}(s) \! &= \! \! 
\idblock{\mathbf{2},\mathbf{3},\mathbf{2}}{%
    \begin{pmatrix}
    		1&0\\
    		s&1 
    \end{pmatrix},
    \begin{pmatrix}
    		1&0&0\\
    		-s&1&0\\
    		-s^2&-s&1 
    \end{pmatrix},
    \begin{pmatrix}
    		1&0\\
    		-s&1 
    \end{pmatrix}
     },
\ x_{-\beta}(s) \! &= \! \!
\idblock{\mathbf{1},\mathbf{2},\mathbf{1},\mathbf{2},\mathbf{1}}{%
    \begin{pmatrix}
    		1&0\\
    		s&1 
    \end{pmatrix},
    \begin{pmatrix}
    		1&0\\
    		-s&1 
    \end{pmatrix}
     }.
\end{align*}
\normalsize
The Frobenius map $F$ permutes the root subgroups as follows.
$$F(x_\alpha(s))=x_\beta(s^{3^{m+1}}), \ F(x_\beta(s))=x_\alpha(s^{3^{m}}),$$ $$F(x_{\text{-}\alpha}(s))=x_{\text{-}\beta}(s^{3^{m+1}}), \ F(x_{\text{-}\beta}(s))=x_{\text{-}\alpha}(s^{3^{m}}).$$

\noindent Therefore, $F^2=\phi_{3^{2m+1}}=\phi_q$, so $F$ is indeed a Frobenius map. We compute the corresponding reflections as $w_\alpha=x_\alpha(1)x_{-\alpha}(-1)x_\alpha(1)$ and $w_\beta=F(w_\alpha)=x_\beta(1)x_{-\beta}(-1)x_\beta(1)$:
\footnotesize 
$$w_\alpha \! = \! \!
\idblock{\mathbf{2},\mathbf{3},\mathbf{2}}{%
    \begin{pmatrix}
    		0&1\\
    		-1&0 
    \end{pmatrix},
    \begin{pmatrix}
    		0&0&-1\\
    		0&-1&0\\
    		-1&0&0
    \end{pmatrix},
    \begin{pmatrix}
    		0&-1\\
    		1&0
    \end{pmatrix}
     },
  w_\beta \! = \! \!
\idblock{\mathbf{1},\mathbf{2},\mathbf{1},\mathbf{2},\mathbf{1}}{%
    \begin{pmatrix}
    		0&1\\
    		-1&0 
    \end{pmatrix},
    \begin{pmatrix}
    		0&-1\\
    		1&0
    \end{pmatrix}
     }.$$
\normalsize
Define $\zwmat=x_\alpha(t)w_\alpha$ and $\widetilde{D}=\zwmat \cdot F(\zwmat) = x_\alpha(t) \; w_\alpha \; x_\beta(t^{3^{m+1}}) \; w_\beta$.
Note that $\widetilde{D}$ is contained in the
Steinberg cross section $X_\alpha w_\alpha X_\beta w_\beta$ of $\Dickson$.
Moreover, we defined $D=x_\alpha(t)w_\alpha x_\beta(s)w_\beta$ in the previous section.
Therefore, $D$ transforms to $\widetilde{D}$ under the specialization $\Fq[t,s]\to \Fq[t]$, $s\mapsto t^{3^{m+1}}$.
We conclude
$$\widetilde{D}=
\begin{pmatrix}
	\text{-}t&\text{-}t^{3^{m+1}}&1&0&0&0&0 \\
	\text{-}1&0&0&0&0&0&0 \\
	0&\text{-}t^2&0&\text{-}t&t^{3^{m+1}}&1&0 \\
	0&t&0&\text{-}1&0&0&0 \\
	0&1&0&0&0&0&0 \\
	0&0&0&0&\text{-}t&0&\text{-}1 \\
	0&0&0&0&1&0&0 \\
\end{pmatrix}. $$

\begin{thm}
For $q=3^{2m+1}, m\geq 0$, let $(M,\Phi)$ be a $7$-dimensional Frobenius module over $(\Fq(t),\phi_q)$ such that the representing matrix of $\Phi$ equals $\widetilde{D}$ with respect to some basis of $M$.
Then
\begin{enumerate}
  \item $\Gal^\Phi(M)\cong {\Ree(q)}$.
  \item The solution field of $M$ is generated by the roots of the additive polynomial
    \begin{eqnarray*} 
      f(X) &=&    X^{q^7}+(t^{q^2-q}+t^{q^3})X^{q^6}+(t^{2q^2-q}-t^{3^{m+1}q^2})X^{q^5} \\
           & & +(-t^{3^{m+1}q+q^2-q}-t^{2q^2})X^{q^4}+(t^{q^2+q}+t^{3^{m+1}q})X^{q^3} \\
           & & +(t^{3^{m+1}+q^2-q}-t^q)X^{q^2}+(-t^{q^2-q+1}-1)X^q-t^{q^2-q}X.
    \end{eqnarray*}
    In particular, $\Gal_{\Fq(t)}(f) \cong {\Ree(q)}$.
\end{enumerate}
\end{thm}
\begin{proof}
a) Theorem~\ref{thm:twupperbound} yields $\Gal^\Phi(M) \leq {\Ree(q)}$ so we only need to verify the lower bound.
If $q=3$, define specializations $\psi_i \colon \mathbb{F}_3[t] \to \mathbb{F}_3$, $t \mapsto i$ for $i=0,1,-1$.
Then $\psi_0(D)$ has order $6$, $\psi_1(D)$ has order $9$ and $\psi_{-1}(D)$ has order $7$.
By Theorem~\ref{thm:lowerbound}, $\Gal^\Phi(M)$ contains elements of order $6, \ 7$ and $9$.
Thus $\Gal^\Phi(M) \cong {\Ree(3)}$, by the list of maximal subgroups in $\Ree(q)$ given in \cite[Thm.~C]{Kleidman}.
Using the same reference for $q>3$, we conclude that it suffices to show that there are specializations $\psi_i: \Fq[t] \to \Fq$, $i=1,2$ such that $\psi_1(D)$ and $\psi_2(D)$ are of orders $q+3^{m+1}+1$ and $q-3^{m+1}+1$, respectively.
We first compute the characteristic polynomial of $\widetilde{D}$.
It equals the characteristic polynomial of $D$ under the specialization $s \mapsto t^{3^{m+1}}$: $h(X)=(X-1)(X^6+(t-1)X^5+(-1-t-t^{3^{m+1}})X^4 +(-1-t+t^{3^{m+1}}-t^2)X^3+(-1-t-t^{3^{m+1}})X^2+(t-1)X+1).$
Let $\zeta_1$ and $\zeta_2$ be roots of unity of order $q+3^{m+1}+1$ and $q-3^{m+1}+1$, respectively. 
Note that $\zeta_i^{q^3}=\zeta_i^{-1}$, since $(q+3^{m+1}+1)(q-3^{m+1}+1)(q+1)=q^3+1$.
Therefore, the minimal polynomial $f_i$ of $\zeta_i$ is $h_i(X)=(X-\zeta_i)(X-\zeta_i^q)(X-\zeta_i^{q^2})(X-\zeta_i^{-1})(X-\zeta_i^{-q})(X-\zeta_i^{-q^2}) \in \Fq[X],$ for $i=1,2$.
The inverse of any root of $h_i$ is a root, too, so $h_i$ has the following symmetry by Lemma~\ref{lem:symmetric coefficients}.
$$h_i(X)=X^6+\beta_iX^5+\gamma_iX^4+\delta_iX^3+\gamma_iX^2+\beta_iX+1.$$ 
Using $\zeta_i^{q^2-q+1}=1$, we compute 
\begin{eqnarray*}
 \beta_i  &=& -\zeta_i-\zeta_i^q-\zeta_i^{q^2}-\zeta_i^{-1}-\zeta_i^{-q}-\zeta_i^{-q^2}, \\
 \gamma_i &=& -\beta_i+\zeta_i^{1+q}+\zeta_i^{1-q^2}+\zeta_i^{q+q^2}+\zeta_i^{q^2-1}+\zeta_i^{-1-q}+\zeta_i^{-q-q^2}, \\
 \delta_i &=& -\beta_i+1-\zeta_i^2-\zeta_i^{2q}-\zeta_i^{2q^2}-\zeta_i^{-2}-\zeta_i^{-2q}-\zeta_i^{-2q^2}. 
\end{eqnarray*}
Furthermore, $\beta_i^2=\zeta_i^2+\zeta_i^{2q}+\zeta_i^{2q^2}+\zeta_i^{-2}+\zeta_i^{-2q}+\zeta_i^{-2q^2}+2\gamma_i$, so we conclude $\delta_i=-\beta_i+1+2\gamma_i-\beta_i^2.$
Next we compute $\beta_i^{3^{m+1}}$ using $\zeta_i^{3^{m+1}}=\zeta_i^{\mp(q+1)}$.
\begin{eqnarray*} 
\beta_i^{3^{m+1}}&=&-\zeta_i^{-(q+1)}-\zeta_i^{- (q^2+q)}-\zeta_i^{1-q^2}-\zeta_i^{(q+1)}-\zeta_i^{(q^2+q)}-\zeta_i^{(-1+q^2)}. 
\end{eqnarray*}
We conclude $$\gamma_i = -\beta_i^{3^{m+1}}-\beta_i, \qquad \delta_i = 1+\beta_i^{3^{m+1}} -\beta_i^2.$$
Setting $t=\beta_i+1$, we obtain $-t^{3^{m+1}}-t-1=-\beta_i^{3^{m+1}}-\beta_i=\gamma_i$.
Also, $t^{3^{m+1}}-t^2-t-1=\beta_i^{3^{m+1}}-\beta_i^2+1 = \delta_i$.
Define specializations $\psi_i:\Fq[t] \to \Fq$, $t \mapsto \beta_i+1$, $i=1,2$.
We just showed that the characteristic polynomial $h^{\psi_i}$ of $\psi_i(D)$ equals $(X-1)h_i(X)$.
Therefore, the Jordan canonical form of $\psi_i(D)$ over $\mathbb{F}_{q^6}$ is $\diag(\zeta_i,\zeta_i^q,\zeta_i^{q^2},1,\zeta_i^{-q^2},\zeta_i^{-q},\zeta_i^{-1})$ and $\psi_i(D)$ has the desired order.  

\medskip\noindent
b) We solve the equation $\widetilde{D} \cdot \phi_q(x)=x$, for $x \in \overline{\Fq(t)}^7$.
Since $\widetilde{D}=\varphi(D)$, where $\varphi:\Fq[t,s]\to \Fq[t]$, $s \mapsto t^{3^{m+1}}$, we obtain the polynomial $\tilde f^\varphi$.
Replacing $s$ by $t^{3^{m+1}}$ in Theorem~\ref{thm:polynomg2oddq}, we obtain the given polynomial which has the correct Galois group by Lemma~\ref{lem:galoisgroup}.
\end{proof}

\begin{rem}
By construction, there exists a polynomial $\tilde f \in \Fq[t,s][X]$ with $\Gal_{\Fq(t,s)}(\tilde f)\cong \Dickson(q)$  such that $f=\tilde f^\varphi$, where $\varphi$ is the specialization homomorphism $\Fq[t,s]\to \Fq[t]$, $s \mapsto t^{3^{m+1}}$.
This might be an indication that $f$ is possibly a generic polynomial for $\Dickson(q)$.
\end{rem}

%% file: polynomial_for_3D4.tex
\subsection{Steinberg's Triality Group $\Triality(q)$}

Let $\G=\operatorname{PSO}_8$ be the projective orthogonal group over $\overline{\Fq(t,s)}$ for an odd $q$ and let $X_1$, $X_2$, $X_3$, and $X_4$ be the projection of the root subgroups of $\SO_8$.
There exists a morphism of linear algebraic groups $\gamma: \operatorname{PSO}_8 \to \operatorname{PSO}_8$ which fixes $X_2$ and permutes $X_1$, $X_3$ and $X_4$.
With $F \coloneqq \gamma \circ \phi_q$, we have $\Triality(q) = \G^F \leq \operatorname{PSO}_8(\F_{q^3})$.
Let $\proj{A} \in \operatorname{PSO}_8$ denote the projection of an element $A \in \SO_8$.
We use the reflections $w_i=x_i(1)x_{-i}(-1)x_i(1)$, where $x_{i}$ denotes the generic element in the root subgroup $X_i$ and where $x_{-i}$ belongs to the negative root of~$x_i$.
Define $\zwmat = x_1(t)w_1x_2(s)w_2 \in \SO_8(\F_{q^3}(t,s))$.
Then
$$F(\proj{\zwmat})=\proj{x_3(t^q)w_3x_2(s^q)w_2},\quad F^2(\proj{\zwmat})=\proj{x_4(t^{q^2})w_4x_2(s^{q^2})w_2}.$$
We thus define $D=x_1(t)w_1\, x_2(s)w_2\cdot x_3(t^q)w_3\, x_2(s^q)w_2\cdot x_4(t^{q^2}) w_4\, x_2(s^{q^2})w_2$ and compute
\footnotesize $$D=\begin{pmatrix}
-t&(-s^{1+q+q^2}-s^{q^2}t^q-t^{q^2}s)&(s^{1+q}+t^q)&-t^{q^2}&-s&s^{q^2}&1&0 \\
-1&0&0&0&0&0&0&0 \\
0&(-s^{q+q^2}-t^{q^2})&s^q&0&-1&0&0&0\\
0&-s^{q^2}&1&0&0&0&0&0 \\
0&-t^q&0&s^q&0&1&0&0 \\
0&-s&0&-1&0&0&0&0 \\
0&-t&0&0&0&0&0&-1 \\
0&1&0&0&0&0&0&0 \\
\end{pmatrix}.$$
\normalsize

\noindent
The projective orthogonal group $\operatorname{PSO}_8$ becomes a linear algebraic group via the natural embedding
$\phi: \operatorname{PSO}_8(\overline{\Fq(t,s)}) \to \GL_{64}(\overline{\Fq(t,s)}).$
Let $(M,\Phi)$ be the $64$-dimensional Frobenius module over $(\F_{q^3},\phi_{q^3})$ such that the representing matrix of $\Phi$ with respect to a fixed basis of $M$ equals $\phi(\proj{D})$. By construction of $D$ and Theorem~\ref{thm:twupperbound} we have $\GalPhi(M) \leq {\Triality(q)}$.

\bigskip
\noindent
For the lower bound, we need to find specializations of $D$ such that any of their conjugates generate $\Triality(q)$,
so we consider the maximal tori of $\Triality(q)$. These have been determined by Gager in \cite{Gager} and can also be found in \cite{DM}. Each of them is conjugate to a subgroup of the diagonal torus by some $g_i \in \operatorname{PSO}_8(\Fqbar)$.
For three of the seven maximal tori we fix an element $t_i$:
\begin{eqnarray*}
t_1&=&\diag(\alpha_1,\alpha_1^{-1},\alpha_1^{-1} ,\alpha_1^{-2q-1},\alpha_1^{2q+1},\alpha_1,\alpha_1,\alpha_1^{-1})^{g_1}, \\
t_2&=&\diag(\alpha_2,\alpha_2^{-1},\alpha_2^{-1} ,\alpha_2^{2q-1},\alpha_2^{-2q+1},\alpha_2,\alpha_2,\alpha_2^{-1})^{g_2}, \\
t_3&=&\diag(1,\alpha_1,\alpha_1^{q+1},\alpha_1^{q},\alpha_1^{-q},\alpha_1^{-q-1}, \alpha_1^{-1},1)^{g_3}, 
\end{eqnarray*}
Here, $\alpha_1$ and $\alpha_2$ are roots of unity of orders $q^2+q+1$ and $q^2-q+1$, respectively.
From the list of maximal subgroups given in \cite{Kleidman2} it is easy to see that any set $\{x_1, x_2, x_3 \}$ of elements in $\Triality(q)$ such that $x_i$ and $t_i$ have the same order ($i=1,2,3$) generates $\Triality(q)$.
Let $f_i$ denote the characteristic polynomial of $t_i$ ($i \leq 3$) and let $f$ be the characteristic polynomial of $D$.
We compute
\begin{align*}
f(X) = X^8
 &+ (t-s^q) \cdot X^7
  + (-st^{q^2}-s^qt-s^{q^2}t^q-s^{1+q+q^2}) \cdot X^6 \\
 &+ (-t-t^{q+q^2}-s^{1+q^2}+s^q+st^q+s^{q^2}t^{q^2}) \cdot X^5 \\
 &+ (-2+2s^qt+s^2+s^{2q^2}+2s^{1+q+q^2}+t^{2q}+t^{2q^2}) \cdot X^4 \\
 &+ (-t-t^{q+q^2}-s^{1+q^2}+s^q+st^q+s^{q^2}t^{q^2}) \cdot X^3 \\
 &+ (-st^{q^2}-s^qt-s^{q^2}t^q-s^{1+q+q^2}) \cdot X^2+(t-s^q) \cdot X + 1.
\end{align*}
Note that the coefficients of $f$ are symmetric.
We would like to find sufficient conditions for $f$ to specialize to $f_i$.
Therefore, we take a closer look at $f_i$, $i \leq 3$.
We denote the coefficients of $f_i$, which are contained in $\F_{q^3}$, as follows:
$$f_i(X)=X^8+\beta_iX^7+\gamma_iX^6+\delta_iX^5+\epsilon_i X^4+\delta_iX^3+\gamma_i X^2+\beta_iX+1.$$

\begin{prop} \label{identitaetencharpoly}
For all $i \leq 3$ we have $\gamma_i \in \Fq$. Moreover, the following identities hold: $\delta_i=-\beta_i^{q^2+q}-\beta_i, \ \epsilon_i=\beta_i^{2q}+\beta_i^{2q^2}-2\gamma_i-2$.
\end{prop}
\begin{proof}
This follows from straight-forward, though lengthy, computations that are very similar to those in the last two sections.
We omit them here.
\end{proof}

\begin{cor} \label{polyspez3d4}
Assume there exists an $s_i \in \mathbb F_{q^3}$ with $$\operatorname{T}(s_i(s_i+\beta_i^{q^2}))+\operatorname{N}(s_i)=-\gamma_i$$ for an $i \leq 3$, where $\operatorname{T}(a)=a+a^q+a^{q^2}$ and $\operatorname{N}(a)=a^{1+q+q^2}$ denote the trace and norm map $\F_{q^3}\to \Fq$. Consider the specialization $\psi_i: \F_{q^3}(t,s) \to \F_{q^3}$, $t \mapsto \beta_i+s_i^q$, $s \mapsto s_i$. Then $f^{\psi_i}=f_i$.
\end{cor}
\begin{proof}
The coefficient at position $X^7$ of $f$ is $t-s^q$ and thus specializes to~$\beta_i$.
The next coefficient is 
$(-st^{q^2}-s^qt-s^{q^2}t^q-s^{1+q+q^2})$ which specializes to $-\operatorname{T}(s_i(s_i+\beta_i^{q^2}))-\operatorname{N}(s_i)=\gamma_i$.
The coefficient at position $X^5$ is $(-t-t^{q+q^2}-s^{1+q^2}+s^q+st^q+s^{q^2}t^{q^2})$ which specializes to 
$-\beta_i-s_i^q-(\beta_i^q+s_i^{q^2})(\beta_i^{q^2}+s_i)-s_i^{1+q^2}+s_i^q+s_i(\beta_i^q+s_i^{q^2})+s_i^{q^2}(\beta_i^{q^2}+s_i)=-\beta_i-\beta_i^{q+q^2}$.
The latter equals $\delta_i$, by Proposition~\ref{identitaetencharpoly}.
Finally the middle coefficient of $f$ is $(-2+2s^qt+s^2+s^{2q^2}+2s^{1+q+q^2}+t^{2q}+t^{2q^2})$ and hence specializes to 
$-2+2s_i^q\beta_i+2s_i^{2q}+s_i^2+s_i^{2q^2}+2s_i^{1+q+q^2}+\beta^{2q}+s_i^{2q^2}+2\beta^qs_i^{q^2}+\beta^{2q^2}+s_i^2+2\beta^{q^2}s_i=2\operatorname{T}(s_i(s_i+\beta_i^{q^2}))+2s_i^{1+q+q^2}+\beta^{2q}+\beta^{2q^2}-2=-2\gamma+\beta^{2q}+\beta^{2q^2}-2=\epsilon_i$, where the last identity again follows from Proposition~\ref{identitaetencharpoly}.
\end{proof}

\begin{prop}
The representing matrix $D$ specializes to elements with characteristic polynomial $f_1$ and $f_2$. 
\end {prop}
\begin{proof}
We give explicit solutions to $\operatorname{T}(s_i(s_i+\beta_i^{q^2}))+\operatorname{N}(s_i)=-\gamma_i$, for $i=1,2$.
As before, let $\alpha_1$ and $\alpha_2$ be roots of unity of orders $q^2+q+1$ and $q^2-q+1$, respectively. We compute
\begin{eqnarray*}
\beta_1&=&-3\alpha_1-3\alpha_1^{-1}-\alpha_1^{2q+1}-\alpha_1^{-2q-1}, \\
\beta_2&=&-3\alpha_2-3\alpha_2^{-1}-\alpha_2^{2q-1}-\alpha_2^{-2q+1}, \\
\gamma_i&=&10+3\operatorname{T}(\alpha_i^2)+3\operatorname{T}(\alpha_i^{-2}), \ \ i=1,2.
\end{eqnarray*}
We define $s_i=\alpha_i^{q^2}+\alpha_i^{-q^2}$. Note that $s_i \in \F_{q^3}$, since $\alpha_i^{q^3}=\alpha_i^{\pm1}$. Then for $i=1,2$,
\begin{eqnarray}
\operatorname{N}(s_i)&=&2+\operatorname{T}(\alpha_i^2)+\operatorname{T}(\alpha_i^{-2}). \nonumber 
\end{eqnarray}
Finally,
\begin{eqnarray*}
\operatorname{T}(s_1(s_1+\beta_1^{q^2}))&=&\operatorname{T}((\alpha_1^{q^2}+\alpha_1^{-q^2})(-2\alpha_1^{q^2}-2\alpha_1^{-q^2}-\alpha_1^{1-q}-\alpha_1^{-1+q})) \\
&=&-12-4\operatorname{T}(\alpha_1^2)-4\operatorname{T}(\alpha_1^{-2}), \\
\operatorname{T}(s_2(s_2+\beta_2^{q^2}))&=&\operatorname{T}((\alpha_2^{q^2}+\alpha_2^{-q^2})(-2\alpha_2^{q^2}-2\alpha_2^{-q^2}-\alpha_2^{1+q}-\alpha_2^{-1-q})) \\
&=&-12-4\operatorname{T}(\alpha_2^2)-4\operatorname{T}(\alpha_2^{-2}).
\end{eqnarray*}
Therefore, $s_i$ is a solution to $\operatorname{T}(s_i(s_i+\beta_i^{q^2}))+\operatorname{N}(s_i)=-\gamma_i$.
By Corollary~\ref{polyspez3d4}, $D$ specializes to elements $A_1$ and~$A_2$ with characteristic polynomials $f_1$ and~$f_2$.
\end{proof}

\noindent 
The following Proposition is due to P. M\"uller.
\begin{prop}\label{lsgi12}
Let $\beta$ and $\gamma$ be arbitrary elements in $\Fq$. Then there exists an $s \in \mathbb{F}_{q^3} \backslash \Fq$ such that $$ \operatorname{T}(s(s+\beta))+\operatorname{N}(s)=-\gamma.$$ In particular, $D$ specializes to an element with characteristic polynomial $f_3$. 
\end {prop}

\begin{proof}
The minimal polynomial of an element of $\mathbb{F}_{q^3}\backslash \Fq$ over $\Fq$ is $$(X-s)(X-s^q)(X-s^{q^2})=X^3-\sigma_1(s) X^2+\sigma_2(s) X-\sigma_3(s)$$ with $\sigma_1(s)=T(s)$, $\sigma_2(s)=s^{1+q}+s^{q+q^2}+s^{q^2+1}$ and $\sigma_3(s)=N(s)$. Since $\beta \in \Fq$, we have $T(s(s+\beta))=\sigma_1(s)^2-2\sigma_2(s)+\beta \sigma_1(s)$. Thus we are looking for an $s \in \mathbb{F}_{q^3} \backslash \Fq$ such that $-\sigma_3(s)=\sigma_1(s)^2-2\sigma_2(s)+\beta \sigma_1(s)+\gamma$. It is therefore sufficient to show that there exist elements $\sigma_1, \sigma_2 \in \Fq$ such that the polynomial 
$$f_{\sigma_1, \sigma_2}=X^3-\sigma_1X^2+\sigma_2X+(\sigma_1^2-2\sigma_2+\beta \sigma_1 +\gamma)$$ is irreducible in $\Fq[X]$(any zero of such a polynomial is then by construction a solution).\\
We have $f_{\sigma_1, \sigma_2}(2)=8-4\sigma_1+\sigma_1^2+\beta \sigma_1+\gamma$, hence we can fix a $\sigma_1 \in \Fq$ such that $f_{\sigma_1, \sigma_2}(2) \neq 0$ for all choices of $\sigma_2 \in \Fq$ (recall that $q$ is odd). For each $x \in \Fq \backslash \{2\}$, there exists a unique $\sigma_2 \in \Fq$ such that $f_{\sigma_1,\sigma_2}(x)=0$. Thus there are at most $q-1$ elements $\sigma_2 \in \Fq$ such that $f_{\sigma_1, \sigma_2}$ has a zero in $\Fq \backslash \{2\}$. We conclude that we can fix an element $\sigma_2$ such that $f_{\sigma_1, \sigma_2}$ has no zeroes in $\Fq$ and is thus irreducible. \\
Note that $\beta_3=-2-T(\alpha_1+\alpha_1^{-1})$ is contained in $\Fq$, hence $s$ as constructed above is a solution of the equation in Corollary \ref{polyspez3d4} for $i=3$ and thus $D$ specializes to an element with characteristic polynomial $f_3$.   
\end{proof}

\begin{thm} \label{suffcond3d4}
Let $(M, \Phi)$ be as before the $64$-dimensional Frobenius module over $\mathbb{F}_{q^3}$ with representing matrix $\phi(\overline{D})$. Then $\GalPhi(M) \cong {\Triality(q)}$.
\end{thm}
\begin{proof}
By Theorem \ref{thm:lowerbound}, $\GalPhi(M)$ contains elements conjugate to $\phi(\psi_i(D))$, for $i=1,2,3$. These have the same orders as $\psi_i(D)$ (which are by construction conjugate to $t_i$) and thus generate $\Triality(q)$.
\end{proof}

\bigskip
\noindent
We remark that the 8-dimensional Frobenius module $(M_2, \Phi_2)$ with representing matrix~$D$ instead of~$\phi(\proj{D})$ over $(\F_{q^3},\phi_{q^3})$ has Galois group either $\Triality(q)$ or $\Triality(q)\times \operatorname{Z}_2$.
It is possible to compute an additive polynomial for $\GalPhi(M_2)$ in a similar manner as for the other groups by solving the system of equations~$D\cdot \phi_{q^3}(x) = x$ which defines the solution space of~$M_2$.
Unfortunately, the equations are intertwined to a certain extent and variable elimination requires a number of recursive steps during which the coefficients of the resulting univariate polynomial become exorbitantly large.
They can be computed with a computer algebra system but the polynomial for $q=3$, e.g., occupies~$\sim\!\!11,800$\nolinebreak{} lines when written out so that it cannot be reproduced it here.
Nonetheless, we describe the steps which lead to a solution of the system.

We let $x=(x_1,\ldots,x_8) \in M$ and in what follows abbreviate $x_1$ by $X$.
Then $D\cdot \phi_{q^3}(x) = x$ is equivalent to an \kxk{8}-system of equations.
The second equation is~$x_2 = -X^{q^3}$.
Combined with the last one we get $x_8 = x_2^{\,q^3} = -X^{q^6}$.
The seventh and fourth equations now yield
\[
  x_7 = -t x_2^{\,q^3} - x_8^{\,q^3} = t X^{q^6} + X^{q^9}, \qquad \qquad
  x_4 = s^{q^2} X^{q^6} + x_3^{\,q^3}.
\]
The expression for~$x_4$ is substituted in the sixth equation, which results in $x_6 = s X^{q^6} - s^{q^5} X^{q^9} - x_3^{\,q^6}$.
We plug both expressions into the fifth equation:
\[
  x_5 = t^q X^{q^6} + (s^{q^5+q} + s^{q^3}) X^{q^9} - s^{q^8} X^{q^{12}} + s^q x_3^{\,q^6} - x_3^{\,q^9},
\]
This in turn is substituted in the third equation:
\begin{multline}
  b^q X^{q^6} - t^{q^4} X^{q^9} - (s^{q^8+q^4} + s^{q^6}) X^{q^{12}} + s^{q^{11}} X^{q^{15}}\\
- x_3 + s^q x_3^{\,q^3} - s^{q^4} x_3^{\,q^9} + x_3^{q^{12}} = 0
\label{eqn:3D4 first}
\end{multline}
Here, we abbreviate $b=s^{q+1}+t^q$ for simplicity.
Substituting everything obtained so far (except the last equation) into the first equation yields:
\begin{multline}
  -X - tX^{q^3} - aX^{q^6} + \big( s^{q^3+q^2} + t^{q^3} - s^{q^5} t^{q^2} - st^{q^4} - s^{q^8+q^2} \big) \cdot X^{q^9}\\
  - \big( s^{q^8+q^4+q} + s^{q^6+q} - 1 \big) \cdot X^{q^{12}} + s^{q^{11}+q}\, X^{q^{15}}\\
\;+\; x_3^{\,q^3} - t^{q^2} x_3^{\,q^6} - (s^{q^4+q}+1) \cdot x_3^{\,q^9} - s\, x_3^{\,q^{12}} = 0
\label{eqn:3D4 second}
\end{multline}%
We have arrived at a system of two equations involving only the variables $X$ and~$x_3$.
Now the following recursive steps can be used to eliminate~$x_3$.
First, we add suitable multiples of~(\ref{eqn:3D4 first}) and (\ref{eqn:3D4 second}) so that the highest terms~$x_3^{\,q^{12}}$ cancel, leaving $x_3^{\,q^9}$ as the highest remaining power.
We apply $\phi_{q^3}$ to the resulting equation and again add to it a suitable multiple of~(\ref{eqn:3D4 first}) or~(\ref{eqn:3D4 second}) to eliminate the newly created term involving~$x_3^{\,q^{12}}$.
This yields a second equation with~$x_3^{\,q^9}$ as the highest power.
The same procedure is applied to the two newly derived equations, resulting in two further equations which involve only $x_3^{\,q^2}$ as highest powers, etc.
After two further steps there only remain linear terms of~$x_3$ and these can be eliminated by one further addition.
The final result is a univariate additive polynomial~$f(X)$ which is easily seen to have $\GalPhi(M_2)$ as Galois group (the reasoning goes along the same lines as for the other groups).